\documentclass{amsart}

\usepackage[OT2, T1]{fontenc}
\usepackage[utf8]{inputenc}
\usepackage[usenames,dvipsnames]{xcolor}
\usepackage{amsmath,amsfonts,amssymb,amsthm}
\usepackage{mathtools}
\usepackage{pgf,tikz}
\usepackage{bbm}
\usetikzlibrary{matrix,arrows,calc}
\usepackage{tikz-cd}
\usepackage{shuffle}
\usepackage{tensor}
\usepackage{enumitem}
\usepackage{subfigure}

\definecolor{amethyst}{rgb}{0.6, 0.4, 0.8}
\usepackage[pagebackref=true, colorlinks, allcolors=amethyst]{hyperref}

\renewcommand*{\backref}[1]{}
\renewcommand*{\backrefalt}[4]{%
	\ifcase #1 %
	\relax
	\or
	$\uparrow$#2.%
	\else
	$\uparrow$#2.%
	\fi%
}

\newcommand{\rF}{{\mathrm F}}

\newcommand{\rH}{{\mathrm H}}

\newcommand{\rW}{{\mathrm W}}

\newcommand{\rd}{{\mathrm d}}

\newcommand{\bC}{{\mathbb C}}
\newcommand{\bD}{{\mathbb D}}

\newcommand{\bF}{{\mathbb F}}
\newcommand{\bG}{{\mathbb G}}

\newcommand{\bN}{{\mathbb N}}

\newcommand{\bP}{{\mathbb P}}
\newcommand{\bQ}{{\mathbb Q}}
\newcommand{\bR}{{\mathbb R}}

\newcommand{\bZ}{{\mathbb Z}}

\newcommand{\cA}{{\mathcal A}}

\newcommand{\cC}{{\mathcal C}}

\newcommand{\cE}{{\mathcal E}}

\newcommand{\cJ}{{\mathcal J}}

\newcommand{\cO}{{\mathcal O}}
\newcommand{\cP}{{\mathcal P}}

\newcommand{\cT}{{\mathcal T}}
\newcommand{\cU}{{\mathcal U}}
\newcommand{\cV}{{\mathcal V}}

\newcommand{\cX}{{\mathcal X}}
\newcommand{\cY}{{\mathcal Y}}
\newcommand{\cZ}{{\mathcal Z}}

\newcommand{\ff}{{\mathfrak f}}
\newcommand{\fg}{{\mathfrak g}}

\newcommand{\fr}{{\mathfrak r}}

\newcommand{\sB}{{\mathsf B}}

\newcommand{\sD}{{\mathsf D}}

\DeclareSymbolFont{cyrletters}{OT2}{wncyr}{m}{n}
\DeclareMathSymbol{\Sha}{\mathalpha}{cyrletters}{"58}

\newcommand{\fSET}{\mathbf{fSet}}

\newcommand{\FEt}{\mathbf{F\acute Et}}

\DeclareMathOperator{\Hom}{Hom}
\DeclareMathOperator{\Map}{Map}

\DeclareMathOperator{\Ext}{Ext}

\DeclareMathOperator{\im}{im}

\DeclareMathOperator{\coker}{coker}

\newcommand{\id}{\mathrm{id}}
\newcommand{\pr}{\mathrm{pr}}

\DeclareMathOperator{\gr}{gr}
 
\DeclareMathOperator{\rank}{rk}

\newcommand{\llbrack}{[\![}
\newcommand{\rrbrack}{]\!]}
\newcommand{\llparen}{(\!(}
\newcommand{\rrparen}{)\!)}

\newcommand{\Fbar}{{\overline{F}}}
\newcommand{\Qbar}{{\overline{\bQ}}}
\newcommand{\Kbar}{{\overline{K}}}
\newcommand{\kbar}{{\overline{k}}}
\newcommand{\Zhat}{{\hat{\bZ}}}

\renewcommand{\Im}{\mathrm{Im}}

\DeclareMathOperator{\Lie}{Lie}
\renewcommand{\div}{\operatorname{div}}

\DeclareMathOperator{\Div}{Div}
\DeclareMathOperator{\Pic}{Pic}
\DeclareMathOperator{\bPic}{\mathbf{Pic}}

\DeclareMathOperator{\Jac}{Jac}
\DeclareMathOperator{\NS}{NS}
\DeclareMathOperator{\Res}{\mathrm{Res}}

\DeclareMathOperator{\Spec}{Spec}

\DeclareMathOperator{\Gal}{Gal}

\newcommand{\alg}{\mathrm{alg}}
\newcommand{\an}{\mathrm{an}}

\newcommand{\et}{\mathrm{\acute{e}t}}

\newcommand{\cris}{\mathrm{cris}}

\newcommand{\st}{{\mathrm{st}}}

\newcommand{\sm}{\mathrm{sm}}

\newcommand{\deR}{\mathrm{dR}}

\newcommand{\Sel}{\mathrm{Sel}}

\DeclareMathOperator{\loc}{loc}

\mathchardef\mhyphen="2D 

\newcommand{\tors}{\mathrm{tors}}

\newcommand{\ab}{\mathrm{ab}}

\theoremstyle{theorem}
\newtheorem{theorem}{Theorem}
\newtheorem*{theorem*}{Theorem}
\newtheorem{theoremABC}{Theorem} 

\newtheorem{innertheoremcustom}{Theorem}
\newenvironment{theoremcustom}[1]
{\renewcommand\theinnertheoremcustom{#1}\innertheoremcustom}
{\endinnertheoremcustom} 
\newtheorem{lemma}[theorem]{Lemma}
\newtheorem{proposition}[theorem]{Proposition}
\newtheorem{corollary}[theorem]{Corollary}
\newtheorem{conjecture}[theorem]{Conjecture}

\theoremstyle{remark}
\newtheorem{remark}[theorem]{Remark}
\newtheorem{example}[theorem]{Example}

\theoremstyle{definition}
\newtheorem{definition}[theorem]{Definition}

\numberwithin{theorem}{section}
\numberwithin{equation}{section}

\definecolor{shadecolor}{RGB}{186,238,186}
\definecolor{softlimegreen}{RGB}{186,238,186}
\definecolor{limegreen}{RGB}{208,243,208}
\definecolor{questioncolor}{RGB}{135, 173, 241}
\definecolor{warningcolor}{RGB}{240,120,134}
\definecolor{verypaleyellow}{RGB}{255,255,194}

\tikzset{
	symbol/.style={
		draw=none,
		every to/.append style={
			edge node={node [sloped, allow upside down, auto=false]{$#1$}}}
	}
}

\newcommand{\T}{\mathrm{T}}
\newcommand{\AT}{\mathrm{AT}}

\newcommand{\Alb}{\mathrm{Alb}}

\newcommand{\alb}{\mathrm{alb}}
\newcommand{\qalb}{\mathrm{qalb}}

\newcommand{\prin}{\mathrm{prin}}
\newcommand{\GQC}{\mathrm{GQC}}
\newcommand{\BK}{\mathrm{BK}}
\newcommand{\EL}{\mathrm{EL}}
\newcommand{\spl}{\mathrm{spl}}

\newcommand{\I}{\mathrm{I}}
\newcommand{\II}{\mathrm{II}}
\newcommand{\III}{\mathrm{III}}
\newcommand{\IV}{\mathrm{IV}}

\newcommand{\LMFDBec}[1]{\href{https://www.lmfdb.org/EllipticCurve/Q/#1/}{\texttt{#1}}}

\newcommand{\printorarxiv}[2]{#2} 

\title{On a non-abelian analogue of a Conjecture of Michael Stoll}
\author{L.~Alexander Betts}

\begin{document}
\maketitle

\begin{abstract}
	We formulate a non-abelian generalisation of a conjecture of Stoll, which conjecturally describes the structure of the loci cut out by Kim's method of non-abelian Chabauty. We prove the rank~$0$ quadratic case of this conjecture, which in particular determines the structure of the quadratic Chabauty locus for once-punctured elliptic curves of rank~$0$. The proof involves using a variant of the geometric quadratic Chabauty method of Edixhoven and Lido to reduce to an unlikely intersections problem, and ultimately to known results about the relative Manin--Mumford Conjecture.
\end{abstract}

\setcounter{tocdepth}{1}
\tableofcontents

\section{Kim's Conjecture}

Let~$X$ be a smooth projective curve\footnote{For us, a \emph{curve} is a geometrically connected one-dimensional reduced scheme over a field. We will only ever consider smooth curves.} of genus~$g\geq2$ over the rationals with a rational base point~$b\in X(\bQ)$, and let~$p$ be a prime of good reduction for~$X$. To any finite-dimensional Galois-equivariant quotient~$U^p$ of the $\bQ_p$-pro-unipotent \'etale fundamental group~$\pi_1^{\bQ_p}(X_\Qbar;b)$, Minhyong Kim's method of \emph{non-abelian Chabauty} associates an obstruction locus
\[
X(\bQ_p)_{U^p} \subseteq X(\bQ_p)
\]
containing the rational points~$X(\bQ)$ \cite{kim:siegel,kim:selmer,balakrishnan-dan-cohen-kim-wewers:non-abelian_conjecture}. In this way, the locus~$X(\bQ_p)_{U^p}$ constrains where the rational points of~$X$ can lie inside the $\bQ_p$-points. It is a conjecture of Kim that this obstruction should exactly cut out the set of rational points, i.e.~$X(\bQ_p)_{U^p}=X(\bQ)$, when the quotient~$U^p$ is taken large enough \cite[Conjecture~3.1]{balakrishnan-dan-cohen-kim-wewers:non-abelian_conjecture}. In this article, we will discuss a number of ideas arising from the question:\smallskip

\noindent For which quotients~$U^p$ should we expect the equality $X(\bQ_p)_{U^p}=X(\bQ)$ to hold?

There is a natural heuristic that suggests when this might occur. The non-abelian Chabauty locus~$X(\bQ_p)_{U^p}$ is, by definition, the common zero locus of some finite number of Coleman-analytic functions $X(\bQ_p)\to\bQ_p$. Let~$c(U^p)$ be the size of a maximal algebraically independent subset of these defining functions\footnote{Equivalently and more precisely, $c(U^p)$ is the codimension of the image of the localisation map $\loc_p\colon\Sel_{U^p}(X/\bQ)\to\rH^1_f(\bQ_p,U^p)$. We will explain what this means in \S\ref{ss:non-ab_Chab}.}. When~$c(U^p)\geq2$, one might expect that $X(\bQ_p)_{U^p}=X(\bQ)$: after all, two general Coleman-analytic functions have no common zeroes at all for dimension reasons. However, this expectation turns out to be rather too na\"ive: in the case that~$X$ is a genus~$3$ hyperelliptic curve and~$U^p=U^p_1$ is the abelianisation of the fundamental group, examples were found in \cite{balakrishnan-bianchi-cantoral-farfan-ciperiani-etropolski:chabauty-coleman_experiments} where~$c(U^p_1)=2$ but $X(\bQ_p)_{U_1}$ contains non-rational points. Similar examples for~$X$ a genus~$3$ Picard curve were found in~\cite{hashimoto-morrison:chabauty-coleman_computations}, and in the realm of non-abelian quotients~$U^p$, analogous examples were found when~$X$ is replaced by a once-punctured elliptic curve and~$U^p=U^p_2$ is the maximal two-step unipotent quotient of the fundamental group in \cite{bianchi:bielliptic}.

What unites all of these examples is that, in all cases, all of the non-rational points appearing in the locus $X(\bQ_p)_{U^p}$ are \emph{algebraic} over~$\bQ$, something which seems highly surprising given the $p$-adic nature of~$X(\bQ_p)_{U^p}$. That this should always occur when~$U^p=U_1^p$ is the subject of a conjecture of Stoll. In fact, Stoll's Conjecture posits something considerably stronger: not only should the loci~$X(\bQ_p)_{U_1^p}$ only contain algebraic points, but in fact they should be interpolated by a single finite subscheme of~$X$ independent of~$p$.

\begin{conjecture}[Stoll, {\cite[Conjecture~9.5]{stoll:finite_descent_old}\footnote{We are taking some liberties with the formulation given in \cite{stoll:finite_descent_old}. Namely, Stoll's result is stated more generally for morphisms from~$X$ to an abelian variety, and also it concerns the Chabauty locus for~$X$ rather than the Chabauty--Coleman locus. These loci can be different \cite{hashimoto-spelier:geometric_linear_chabauty}.}}]\label{conj:stoll}
	Suppose that~$X$ is a smooth projective curve of genus~$g\geq2$, whose Jacobian has Mordell--Weil rank at most~$g-2$. Then there is a finite subscheme $Z\subset X$, defined over~$\bQ$, such that
	\[
	X(\bQ_p)_{U_1^p} \subseteq Z(\bQ_p)
	\]
	for all primes~$p$ in a set of Dirichlet density~$1$, where~$U_1^p$ is the abelianisation of the fundamental group of~$X_\Qbar$. (Because~$Z$ is finite and defined over~$\bQ$, $Z(\bQ_p)$ consists only of algebraic points.)
\end{conjecture}

Inspired by Stoll's Conjecture, we propose a more general version, which simultaneously generalises it to the non-abelian Chabauty setting and allows the curve~$X$ to be replaced by a not-necessarily-projective curve~$Y$.

\begin{conjecture}[Non-abelian Stoll's Conjecture]\label{conj:non-abelian_stoll}
	Let~$Y/\bQ$ be a smooth hyperbolic curve with a regular model\footnote{For us, a regular model of~$Y$ means a regular flat separated $\bZ$-scheme of finite type with generic fibre~$Y$.}~$\cY/\bZ$, and let~$b$ be a $\bQ$-rational base point on~$Y$, possibly tangential. Let~$p$ be a prime of good reduction for~$(\cY,b)$, and let~$U^p$ be a finite dimensional $\Gal_\bQ$-equivariant quotient of the $\bQ_p$-pro-unipotent \'etale fundamental group of~$Y$ which is of motivic origin.
	
	Suppose that~$c(U^p)\geq2$. Then there exists a finite subscheme $Z\subset Y$, defined over~$\bQ$ and independent of~$p$, such that
	\[
	\cY(\bZ_q)_{U^q} \subseteq Z(\bQ_q)
	\]
	for all primes~$q$ in a set of density~$1$ (or perhaps all but finitely~$q$).
\end{conjecture}

\begin{remark}
	In the statement of the conjecture, we leave the term ``of motivic origin'' deliberately undefined, though in practice the meaning should be clear. One property which one should expect of quotients~$U^p$ of motivic origin, which is already used in the statement of the conjecture, is that there should be corresponding quotients~$U^q$ of the $\bQ_q$-pro-unipotent \'etale fundamental group of~$Y_\Qbar$ for all primes~$q$. Moreover, if~$c(U^p)\geq2$ for one prime~$p$, then~$c(U^q)\geq2$ for all~$q$.
\end{remark}

In this paper, we are going to prove some cases of non-abelian Stoll's Conjecture, which include some cases where the quotient~$U^p$ is genuinely non-abelian. Let~$Y$, $b$ and~$\cY$ be as in Conjecture~\ref{conj:non-abelian_stoll}, write~$Y=X\smallsetminus D$ for a reduced divisor~$D$ on a smooth projective curve~$X$, and let~$J$ denote the Jacobian of~$X$. Let~$U^p_\T$ denote the largest quotient of~$\pi_1^{\bQ_p}(Y_\Qbar;b)$ which is an extension of $V_pJ$ by a representation of the form~$\bQ_p(1)^r$. The locus~$\cY(\bZ_p)_{U^p_\T}$ is commonly known as the \emph{quadratic Chabauty locus}\footnote{In fact, there are several different loci which go by the name ``quadratic Chabauty locus'' in the literature, namely the locus corresponding to either $U^p_\T$ (largest extension of~$V_pJ$ by~$\bQ_p(1)^r$) or $U^p_\AT$ (largest extension of~$V_pJ$ by an Artin--Tate representation of weight~$-2$) or~$U^p_2$ (largest two-step unipotent quotient). In the nomenclature \printorarxiv{of the article by Netan Dogra in this volume}{introduced by Netan Dogra}, these are the ``skimmed'', ``semi-skimmed'' and ``full fat'' quadratic Chabauty loci, respectively. We are discussing the skimmed version in this paper.}, and it has been used to great effect in several high-profile computations of rational points \cite{balakrishnan-dogra-mueller-tuitman-vonk:cursed_curve,balakrishnan-besser-bianchi-mueller:explicit_quadratic_chabauty_number_fields,balakrishnan-best-bianchi-lawrence-mueller-triantafillou-vonk:two_approaches,dogra-le_fourn:quadratic_chabauty_modular_forms,adzaga-chidambaram-keller-padurariu:atkin-lehner_quotients_coverings,adzaga-arul-beneish-chen-chidambaram-keller-wen:quadratic_chabauty_atkin-lehner_quotients,arul-mueller:X0+125,balakrishnan-dogra-mueller-tuitman-vonk:quadratic_chabauty_algorithms_examples,bianchi-padurariu:bielliptic_in_LMFDB}. We prove the rank zero case of Conjecture~\ref{conj:non-abelian_stoll} for this~$U^p_\T$.

\begin{theoremABC}\label{thm:main_tate}
	Let~$\rho_\bQ\coloneqq\rank(\NS(J)(\bQ))$ denote the rational Picard number of~$J$, and let~$n_0$ denote the number of $\Gal_\bQ$-orbits in~$D(\Qbar)$. Suppose that $\rank(J(\bQ))=0$ and
	\[
	g + \rho_\bQ + n_0 - 1 \geq 2 \,.
	\]
	Then there exists a finite subscheme $Z\subset Y$, defined over~$\bQ$ and independent of~$p$, such that
	\[
	\cY(\bZ_p)_{U^p_\T} = \cZ(\bZ_p)
	\]
	for all primes~$p$ of good reduction for~$(\cY,b)$, where~$\cZ\subset\cY$ is the closure of~$Z$ in~$\cY$. Also, we have $c(U^p_\T)\geq2$ for all such primes.
\end{theoremABC}

The primary case of interest in Theorem~\ref{thm:main_tate} is when~$Y=E\smallsetminus\{\infty\}$ for~$E$ an elliptic curve of Mordell--Weil rank~$0$, and~$\cY=\cE\smallsetminus\overline{\{\infty\}}$ with~$\cE$ the minimal regular model of~$E$. In this case, $U^p_\T=U^p_2$ is the full two-step unipotent fundamental group, and we have $g=\rho_\bQ=n_0=1$ so Theorem~\ref{thm:main_tate} applies. This setting was previously studied by Bianchi in \cite{bianchi:bielliptic}, in which she proved that every element of~$\cY(\bZ_p)_{U^p_2}$ is always a torsion point on~$E$ -- in particular it is algebraic. She also gave some sufficient criteria for a torsion point to lie in~$\cY(\bZ_p)_{U^p_2}$, and also a separate necessary criterion.

Our Theorem~\ref{thm:main_tate} extends Bianchi's work in two ways. Firstly, we establish the existence of the subscheme~$Z$ independent of~$p$, which for example implies that the size of~$\cY(\bZ_p)_{U^p_2}$ is bounded as~$p$ ranges over all good reduction primes. Secondly, we describe the subscheme~$Z$ in Theorem~\ref{thm:main_tate} explicitly enough as to obtain a criterion for a torsion point to lie in $\cY(\bZ_p)_{U^p_\T}$ which is simultaneously necessary and sufficient. See Theorem~\ref{thm:main_elliptic_curve} for the statement and \S\ref{ss:elliptic_examples} for some examples.

\subsection{Geometry and unlikely intersections}

Our proof of Theorem~\ref{thm:main_tate} will be ultimately geometric. To see the connection to geometry, let us say that a quotient~$U^p$ of~$\pi_1^{\bQ_p}(Y_\Qbar;b)$ is \emph{realised} by a morphism
\[
f\colon Y \to V
\]
to a smooth geometrically connected variety~$V/\bQ$ if the induced map on $\bQ_p$-pro-unipotent \'etale fundamental groups is surjective and factors through an isomorphism $U^p\xrightarrow\sim\pi_1^{\bQ_p}(V_\Qbar;f(b))$. We declare by \emph{fiat} that any quotient realised by a morphism should be motivic in the sense of Conjecture~\ref{conj:non-abelian_stoll}. Indeed, if a morphism $f\colon Y\to V$ induces a surjection on $\bQ_p$-pro-unipotent \'etale fundamental groups for some~$p$ then, by the usual comparison theorems, it also induces a surjection on $\bQ_q$-pro-unipotent \'etale fundamental groups (as well as $\bQ$-pro-unipotent Betti fundamental groups, de Rham fundamental groups, etc.), and in particular we have corresponding quotients~$U^q$ of~$\pi_1^{\bQ_q}(Y_\Qbar;b)$ for all other primes~$q$ and the statement of Conjecture~\ref{conj:non-abelian_stoll} makes sense.

\begin{example}
	If~$Y=X$ is projective, then the Abel--Jacobi embedding $X\hookrightarrow J$ induces a surjection on pro-unipotent fundamental groups, and realises the abelianisation~$U^p_1$ of~$\pi_1^{\bQ_p}(X_\Qbar;b)$. In this case, $c(U^p_1)\geq g-\rank(J(\bQ))$, and so Stoll's original Conjecture~\ref{conj:stoll} is recovered as a special case of Conjecture~\ref{conj:non-abelian_stoll}.
\end{example}

The quotient $U^p_\T$ appearing in Theorem~\ref{thm:main_tate} turns out to be realised by a morphism to a variety~$V$ of a particular shape. Let us say that a variety~$P$ is an \emph{AT variety}\footnote{AT stands for ``abelian-by-torus''. The fact that it also stands for ``Artin--Tate'' is, in light of Theorem~\ref{thm:quadratic_albanese_intro}, a fortuitous coincidence.} if there exists a torus~$G$, an abelian variety~$A$ and an $A$-torsor~$T$ such that~$P$ is isomorphic as a variety to a $G$-torsor over~$T$. (Neither the torus~$G$, abelian variety~$A$, nor torsor~$T$ are officially part of the data of an AT variety, but we will show that they can all be recovered canonically from~$P$.) We call an AT variety \emph{split} if the torus~$G$ is split. Semiabelian varieties are examples of AT varieties, but there exist AT varieties which are not semiabelian, e.g.~because they have non-abelian fundamental groups. It is a result, probably well-known to experts, that the quotient~$U^p_\T$ is always realised by a morphism to an AT variety, as is the larger quotient~$U^p_\AT$ which is the largest quotient which is an extension of~$V_pJ$ by an Artin--Tate representation of weight~$-2$. We give a careful proof in this paper.

\begin{theoremABC}\label{thm:quadratic_albanese_intro}
	Let~$Y$ be a smooth curve over a perfect field~$k$. Then there exists an AT variety~$Q$ over~$k$ and a morphism $\qalb\colon Y\to Q$ which is initial among all morphisms from~$Y$ to an AT variety. When~$k=K$ is a number field and~$Y$ has a $K$-rational base point~$b$ (possibly tangential), then~$\qalb$ realises the quotient~$U^p_\AT$.
\end{theoremABC}

\begin{theoremcustom}{B'}\label{thm:split_quadratic_albanese_intro}
	Let~$Y$ be a smooth curve over a perfect field~$k$. Then there exists a split AT variety~$Q_\spl$ over~$k$ and a morphism $\qalb_\spl\colon Y\to Q_\spl$ which is initial among all morphisms from~$Y$ to a split AT variety. When~$k=K$ is a number field and~$Y$ has a $K$-rational base point~$b$ (possibly tangential), then~$\qalb_\spl$ realises the quotient~$U^p_\T$.
\end{theoremcustom}

The variety~$Q$ appearing in Theorem~\ref{thm:quadratic_albanese_intro} (resp.~$Q_\spl$ in Theorem~\ref{thm:split_quadratic_albanese_intro}) we call the (split) \emph{quadratic Albanese variety} of~$Y$, and the morphism $\qalb$ (resp.~$\qalb_\spl$) we call the (split) \emph{quadratic Albanese map} of~$Y$. Over~$\bC$, these are related to the second higher Albanese manifold of Hain--Zucker \cite{hain-zucker:unipotent_variations}, see Remark~\ref{rmk:higher_albanese_manifolds}.

The variety~$Q_\spl$ is a torsor under a split torus of dimension~$\rho_\bQ+n_0$ over~$\bPic^1(X)$. So Theorem~\ref{thm:main_tate} is a special case of the following more general theorem.

\begin{theoremcustom}{A'}\label{thm:main}
	Let~$Y/\bQ$ be a smooth curve with a regular model~$\cY/\bZ$, and let~$b$ be a $\bQ$-rational base point on~$Y$, possibly tangential. Let $f\colon Y\to P$ be a morphism from~$Y$ to a $\bG_m^r$-torsor over an abelian variety~$A$ which induces a surjection on pro-unipotent fundamental groups.
	
	Suppose that~$A$ has Mordell--Weil rank~$0$ and~$\dim(P)\geq2$. Then there exists a finite subscheme $Z\subset Y$, defined over~$\bQ$, such that
	\[
	\cY(\bZ_p)_{U^p} = \cZ(\bZ_p)
	\]
	for all primes of good reduction for~$(\cY,b)$, where~$U^p$ is the quotient realised by~$f$ and~$\cZ\subset\cY$ is the closure of~$Z$ in~$\cY$. Also, we have $c(U^p)\geq2$ for all such primes. In particular, Conjecture~\ref{conj:non-abelian_stoll} holds for~$U^p$.
\end{theoremcustom}

Our strategy for proving Theorem~\ref{thm:main} proceeds in two steps, each of their own independent interest. The first step is to give a purely geometric description of the non-abelian Chabauty locus~$\cY(\bZ_p)_{U^p}$ in terms of~$P$ and the morphism~$f$, removing all $p$-adic analysis. This is of course inspired by the geometric quadratic Chabauty method of Edixhoven and Lido \cite{edixhoven-lido:geometric_quadratic_chabauty,coupek-lilienfeldt-xiao-yao:geometric_quadratic_chabauty_number_fields}, which for smooth projective curves~$X$ uses the theory of the Poincar\'e biextension to describe an obstruction locus~$X(\bQ_p)^\GQC\subseteq X(\bQ_p)$ containing the rational points. However, the geometric quadratic Chabauty obstruction is not actually equivalent to quadratic Chabauty: when~$p$ is a prime of good reduction, we have
\[
X(\bQ_p)^\GQC \subseteq X(\bQ_p)_{U^p_\T}
\]
but the containment can sometimes be strict \cite{duque-rosero-hashimoto-spelier:geometric_quadratic_chabauty}. So what we need is a variant of Edixhoven--Lido's geometric quadratic Chabauty construction which recovers $\cY(\bZ_p)_{U^p}$ exactly, and not a subset thereof.

This we accomplish in \S\ref{s:isogeny_method} where we define, for any smooth curve~$Y/\bQ$ and any morphism $f\colon Y\to P$ with~$P$ a $\bG_m^r$-torsor over an abelian variety of Mordell--Weil rank~$0$, an \emph{isogeny geometric quadratic Chabauty locus}
\[
\cY(\bZ_p)_f \subseteq \cY(\bZ_p)
\]
containing~$\cY(\bZ)$, for any regular model~$\cY$ of~$Y$. The details will be given in~\S\ref{s:isogeny_method}; for now, suffice it to say that we will construct from~$P$ a filtered system of AT varieties~$(P_n)_{n\geq1}$ and a compatible system of morphisms $\beta_n\colon P\to P_n$, from which we obtain a compatible system of morphisms $f_n\colon Y\to P_n$ by precomposing with~$f$. We also define a finite subset~$W\subseteq\varinjlim_nP_n(\bQ)$, and the locus~$\cY(\bZ_p)_f$ is then defined to be the inverse image of~$W$ under the function
\[
f_\infty\colon \cY(\bZ_p) \to \varinjlim_nP_n(\bQ_p)
\]
induced by the~$f_n$. We will show that this construction exactly recovers the non-abelian Chabauty locus.

\begin{theoremABC}\label{thm:comparison_intro}
	Suppose in the above setting that~$f$ induces a surjection on pro-unipotent fundamental groups, that~$Y$ has a rational base point~$b$, possibly tangential, and that~$p$ is a prime of good reduction for~$(\cY,b)$. Then
	\[
	\cY(\bZ_p)_{U^p} = \cY(\bZ_p)_f
	\]
	where~$U^p$ is the quotient of~$\pi_1^{\bQ_p}(Y_\Qbar;b)$ realised by~$f$.
\end{theoremABC}

\begin{remark}
	The isogeny geometric quadratic Chabauty method which we set up is both more general and more restrictive than quadratic Chabauty in the original formulation. It is more restrictive in that the definition we give in this paper is valid only when~$P$ is a torsor over an abelian variety of Mordell--Weil rank~$0$, while original quadratic Chabauty is well-defined for all quotients realised by AT varieties. It is more general in that we do not require that the curve~$Y$ has a rational base point, nor do we place any restriction on the reduction type of~$\cY$ at~$p$. It seems likely that the definition of the non-abelian Chabauty locus~$\cY(\bZ_p)_{U^p}$ will eventually be extended to also remove these assumptions, in which case one should expect Theorem~\ref{thm:comparison_intro} to also hold in these cases.
\end{remark}

The second step in the proof of Theorem~\ref{thm:main} is to exploit our understanding of the structure of~$\cY(\bZ_p)_{U^p} = \cY(\bZ_p)_f$, namely that it is the union of finitely many fibres of the function $f_\infty\colon\cY(\bZ_p)\to\varinjlim_nP_n(\bQ_p)$. Ultimately, Theorem~\ref{thm:main} reduces to showing that the fibres of the corresponding function
\begin{equation}\label{eq:f_infty}
	f_\infty\colon Y(\Qbar) \to \varinjlim_nP_n(\Qbar)
\end{equation}
on~$\Qbar$-points are all finite. Let us call a fibre of~\eqref{eq:f_infty} a \emph{quadratic torsion packet} in~$Y_\Qbar$ relative to~$f$. This is by way of analogy with the torsion packets on a smooth projective curve~$X$, which are the fibres of the function $X(\Qbar)\to\bQ\otimes J(\Qbar)$ arising from the Abel--Jacobi map $X_\Qbar\to J_\Qbar$.

Quadratic torsion packets can be fruitfully studied from the perspective of unlikely intersections. That is, studying quadratic torsion packets is equivalent to studying the intersection of the locally closed subvariety $\im(f)\subseteq P_\Qbar$ with the fibres of the function
\begin{equation}\label{eq:beta_infty}
	\beta_\infty\colon P(\Qbar) \to \varinjlim_nP_n(\Qbar)
\end{equation}
induced by the morphisms~$\beta_n\colon P\to P_n$ (because~$f_\infty=\beta_\infty\circ f$). If we think of the points in a fixed fibre of~$\beta_\infty$ as being ``special points'' on~$P_\Qbar$, this problem lies squarely in the realm of unlikely intersections and is in a similar vein to Manin--Mumford for semiabelian varieties and Andr\'e--Oort for mixed Shimura varieties (though apparently is a special case of neither). Accordingly, one would expect that there is a class of ``weakly special'' subvarieties of~$P_\Qbar$ such that~$\im(f)$ only contains finitely many special points unless its closure is a one-dimensional weakly special subvariety. It turns out that there is such a class of weakly special subvarieties, and it is exactly the \emph{AT subvarieties}: closed subvarieties of~$P_\Qbar$ which are themselves AT varieties. The final step in the proof of Theorem~\ref{thm:main} is then a version of Manin--Mumford in the setting of AT varieties.

\begin{theoremABC}[Manin--Mumford for curves in AT varieties]\label{thm:manin-mumford}
	Let~$Y/\Qbar$ be a smooth curve, and let~$f\colon Y\to P$ be a morphism from~$Y$ to an AT variety over~$\Qbar$. Suppose that~$\dim(P)\geq2$ and that the image of~$f$ is not contained in any strict AT subvariety of~$P$. Then every quadratic torsion packet in~$Y$ relative to~$f$ is finite.
\end{theoremABC}

With one exception, Theorem~\ref{thm:manin-mumford} can be deduced from Manin--Mumford for semiabelian varieties. The exception is the case when~$P$ is a $\bG_m$-torsor over an elliptic curve, where we instead must use~$f$ to produce a semiabelian scheme over~$Y$ with a section, and apply a hard relative Manin--Mumford result of Bertrand--Masser--Pillay--Zannier \cite{bertrand-masser-pillay-zannier:relative_manin-mumford}. We give the full proof in \S\ref{s:unlikely_intersections}.

\begin{remark}
	The author's original interest in Stoll's Conjecture and its generalisation came from its applications to finite descent and the section conjecture. In \cite[Theorem~A]{betts-kumpitsch-luedtke} with Theresa Kumpitsch and Martin L\"udtke, the author showed that if~$Y/\bQ$ is a hyperbolic curve with a model~$\cY$ such that there exists a finite subscheme $Z\subset Y$ such that $\cY(\bZ_p)_\infty\subseteq Z(\bQ_p)$ for a density~$1$ set of primes~$p$, then the Selmer Section Conjecture holds for~$\cY$. Combining this with Theorem~\ref{thm:main_tate} gives a new proof that the Selmer Section Conjecture holds whenever $\rank(J(\bQ))=0$, $\Sha(J/\bQ)$ is finite\footnote{The relevance of this condition may not be apparent at first. It is because \cite{betts-kumpitsch-luedtke} uses a different definition of Selmer schemes and the Chabauty--Kim locus compared with this paper, cf.~Remark~\ref{rmk:balakrishnan-dogra_modification}.}, and $g+\rho_\bQ+n_0-1\geq2$. This result was already known to experts, and is true even without the assumption that $g+\rho_\bQ+n_0-1\geq2$ (when~$g\geq1$, \cite[Corollary~6.2(2)]{stoll:finite_descent} shows that the finite descent locus of~$\cY$ is contained in a finite subscheme of~$Y$, and one can conclude using \cite[Theorem~2.6]{stoll:finite_descent}; the case~$g=0$ was handled by Stix \cite[Corollary~6]{stix:local_birational}).
\end{remark}

\subsection*{Structure of the paper}

We begin in \S\ref{s:AT_varieties} by developing the basic theory of AT varieties, including the construction of the quadratic Albanese variety of a smooth curve. This allows us to set up the isogeny geometric quadratic method in \S\ref{s:isogeny_method}, a variant on Edixhoven--Lido's geometric quadratic Chabauty method. The technical heart of the paper comes in the following two sections. In \S\ref{s:comparison}, we show that our isogeny geometric quadratic Chabauty method is equivalent to original quadratic Chabauty by systematically describing the relevant Selmer schemes and Kummer maps in terms of AT varieties. In \S\ref{s:unlikely_intersections} which follows, we develop the basic theory of unlikely intersections inside AT varieties, proving some initial results on quadratic torsion packets and AT subvarieties to eventually prove our Manin--Mumford Theorem~\ref{thm:manin-mumford}. Once this is accomplished, the main Theorem~\ref{thm:main} is proved.

We also include a short Appendix~\ref{appx:non-realisability}, where we consider the question of which other quotients of the fundamental group are realisable by morphisms to smooth varieties. The answer turns out to be very restrictive: even the full depth~$2$ quotient~$U^p_2$ is not realisable for any smooth projective curve of genus~$\geq4$.

\subsection*{Acknowledgements}

I am very grateful to David Urbanik and Ziyang Gao for many helpful discussions about unlikely intersections, and especially to the latter for introducing me to the notion of Ribet sections in relative Manin--Mumford. I am also grateful to Guido Lido and Bas Edixhoven for taking the time many years ago to explain to me their geometric approach to quadratic Chabauty, which heavily inspired this paper. I am also indebted to Martin L\"udtke, who provided many helpful comments on an earlier version of this paper.

\printorarxiv{I reserve my deepest and sincerest thanks for Minhyong Kim, in whose honour this volume is dedicated. Minhyong has been a help and inspiration for me over the years, in so many ways both mathematical and in life more broadly, and I am profoundly grateful for his untiring support.}{I reserve my deepest and sincerest thanks for Minhyong Kim, who has been a help and inspiration for me over the years, in so many ways both mathematical and in life more broadly.} I hope that he gains some pleasure from seeing developed in this article several ideas which he initially introduced me to some eleven years ago.

\section{AT varieties and quadratic Albanese varieties}\label{s:AT_varieties}

In this section, we will develop the basic geometry of AT varieties over a field~$k$, describing their structure, and giving the construction of the quadratic Albanese variety. We will also describe the relationship with fundamental groups. We begin with a foundational result on morphisms between AT varieties.

\begin{lemma}\label{lem:AT_morphisms}
	Suppose that~$T$ and~$T'$ are torsors under abelian varieties~$A$ and~$A'$, and that~$P$ and~$P'$ are torsors under tori~$G$ and~$G'$ over~$T$ and~$T'$, respectively.
	
	Then for any morphism $\phi_P\colon P\to P'$ of $k$-varieties, there exists a unique homomorphism $\phi_A\colon A\to A'$, a unique morphism $\phi_T\colon T\to T'$, and a unique homomorphism $\phi_G\colon G\to G'$ such that the squares
	\begin{equation}\label{eq:morphism_over_under_diagrams}
		\begin{tikzcd}
			A\times T \arrow[r,"\phi_A\times \phi_T"]\arrow[d,"\tau"] & A'\times T' \arrow[d,"\tau'"] \\
			T \arrow[r,"\phi_T"] & T'
		\end{tikzcd}
		\quad
		\begin{tikzcd}
			P \arrow[r,"\phi_P"]\arrow[d,"\pi"] & P' \arrow[d,"\pi'"] \\
			T \arrow[r,"\phi_T"] & T'
		\end{tikzcd}
		\quad
		\begin{tikzcd}
			G\times P \arrow[r,"\phi_G\times \phi_P"]\arrow[d,"\rho"] & G'\times P' \arrow[d,"\rho'"] \\
			P \arrow[r,"\phi_P"] & P'
		\end{tikzcd}
	\end{equation}
	commute, where~$\pi$ and~$\pi'$ are the projection maps and~$\tau$, $\tau'$, $\rho$ and~$\rho'$ are the action maps.
\end{lemma}
\begin{proof}
	By fpqc descent for morphisms, it suffices to prove the result in the case that~$k=\kbar$ is algebraically closed. We may then suppose further that~$T=A$ and~$T'=A'$ are both trivial torsors. Since the only morphisms from tori to abelian varieties are constant, the composition $P\xrightarrow{\phi_P}P'\xrightarrow{\pi'}A'$ must be constant on the fibres of~$\pi\colon P\to A$, and so factors through a unique morphism~$\phi_T\colon A\to A'$ of $k$-varieties. The morphism~$\phi_T$ is automatically a translate of a unique homomorphism $\phi_A\colon A\to A'$. So we have shown commutativity of the left and middle square in~\eqref{eq:morphism_over_under_diagrams}. 
	
	To construct~$\phi_G$, consider the two compositions
	\[
	G \times P \to P \xrightarrow{\phi_P} P'
	\]
	where the first map is either the action~$\rho$ or the second projection. These two maps have the same composition with~$\pi'\colon P'\to A'$, so they differ by the action of a unique morphism $h\colon G\times P\to G'$. Explicitly, we have the identity
	\[
	\rho'(h(g,x),\phi_P(x)) = \phi_P(\rho(g,x))
	\]
	for all~$x\in P$ and~$g\in G$. For any fixed~$x$, the restriction of~$h$ to~$G\times\{x\}$ is a morphism between tori, so is a translate of a homomorphism $G\to G'$. In fact, since~$h(1,x)=1$, the restriction of~$\phi$ to~$G\times\{x\}$ is a homomorphism. Thus~$h$ induces a morphism from~$P$ to the Hom-scheme $\Hom(G,G')$. Since the Hom-scheme is discrete and~$P$ is connected, this factored map must be constant, hence~$h(g,x)=\phi_G(g)$ for some homomorphism $\phi_G\colon G\to G'$ independent of~$x$. This shows commutativity of the second square in~\eqref{eq:morphism_over_under_diagrams}.
\end{proof}

\begin{definition}\label{def:over_under}
	We say that the morphism~$\phi_P\colon P \to P'$ lies \emph{over} the morphism $\phi_T\colon T\to T'$ and \emph{under} the homomorphism~$\phi_G\colon G\to G'$. We say that~$\phi_T$ lies \emph{under} the homomorphism~$\phi_A\colon A\to A'$.
\end{definition}

One consequence of Lemma~\ref{lem:AT_morphisms} is that if~$P$ is a $G$-torsor over an $A$-torsor~$T$, as in the setup of the lemma, then~$G$, $A$ and~$T$ can be recovered up to unique isomorphism from the variety structure on~$P$ alone. We will often use this notation without comment: if~$P$ is an AT variety, then~$G$ will always denote the torus which acts on~$P$, $T$ will always denote the torsor over which~$P$ lies, and~$A$ will always denote the abelian variety which acts on~$T$. Similarly, the AT variety~$P'$ will have corresponding torus~$G'$, abelian variety~$A'$ and torsor~$T'$, and so on.

There are two fundamental constructions on AT varieties. Firstly, if~$P\to T$ is an AT variety and~$\phi_T\colon T'\to T$ is a morphism with~$T'$ a torsor under another abelian variety~$A'$, then the \emph{pullback} $\phi_T^*P\coloneqq P\times_TT'$ is naturally a $G$-torsor over~$T'$, coming with a morphism $\phi_T^*P\to P$ lying over~$\phi_T$ and under the identity~$\id_G$. The dual notion is that of the \emph{pushout}: if~$\phi_G\colon G\to G'$ is a homomorphism of tori, then the pushout~$\phi_{G*}P$ is the quotient of~$G'\times P$ by the action of~$G$ by~$g\cdot(g',\tilde x)=(g'\phi_G(g)^{-1},g\tilde x)$. The pushout is a $G'$-torsor over~$T$, and comes with a morphism $P\to \phi_{G*}P$ of torsors lying over the identity~$\id_T$ and under~$\phi_G$.

\begin{lemma}\label{lem:decomposition_of_morphisms}
	Let~$P$ and~$P'$ be AT varieties, with associated tori~$G$ and~$G'$, abelian varieties~$A$ and~$A'$, and torsors~$T$ and~$T'$, respectively. Let~$\phi_G\colon G\to G'$ be a homomorphism and~$\phi_T\colon T\to T'$ be a morphism. Then every morphism $\phi\colon P\to P'$ lying over~$\phi_T$ and under~$\phi_G$ factors uniquely as a composition
	\[
	P \to \phi_{G*}P \simeq \phi_T^*P' \to P' \,,
	\]
	where the middle isomorphism is an isomorphism of~$G'$-torsors over~$T$.
	
	In particular, there exists such an~$\phi$ if and only if~$\phi_{G*}P$ and~$\phi_T^*P'$ are isomorphic as torsors, and when they are, the set of all possible morphisms~$\phi$ is a torsor under~$\Map(T,G)=G(k)$.
\end{lemma}
\begin{proof}
	By the universal property of fibre products, the morphism $\phi\colon P\to P'$ factors uniquely through a morphism $\phi'\colon P\to \phi_T^*P'$ lying over~$\phi_T$. The morphism~$\phi'$ lies under~$\phi_G$, so the map~$G'\times P\to \phi_T^*P'$ given by~$(g',\tilde x)\mapsto g'\cdot \phi'(\tilde x)$ is $G$-invariant, and so factors uniquely through a morphism $\phi_{G*}P\to \phi_T^*P'$ lying under~$\id_{G'}$ and over~$\id_T$. This final morphism is automatically an isomorphism, because both~$\phi_{G*}P$ and $\phi_T^*P'$ are torsors.
\end{proof}

Among all morphisms of AT varieties, there is a class which will play a special role in the theory. We call a morphism $\phi\colon P\to P'$ of AT varieties an \emph{isogeny} just when the homomorphisms $\phi_G\colon G\to G'$ and $\phi_A\colon A\to A'$ are isogenies. For later use, we record one property of isogenies.

\begin{lemma}\label{lem:isogenies_finite}
	Any isogeny $\phi\colon P\to P'$ of AT varieties is finite.
\end{lemma}
\begin{proof}
	By fpqc descent, it suffices to prove the result when~$k=\kbar$ is algebraically closed, so we may assume that~$T=A$, $T'=A'$, $G=\bG_m^r$ and $G'=\bG_m^{r'}$. We may also assume that~$\phi_T=\phi_A$. Then the morphism $\phi_A^*P'\to P'$ is finite (it is a pullback of the finite morphism $\phi_A\colon A\to A'$), and the morphism $P\to\phi_{G*}P$ is also finite (it is, Zariski-locally on~$A$, identified with the map $\phi_G\times\id_A\colon G\times A \to G'\times A$). So~$\phi$ itself is finite by Lemma~\ref{lem:decomposition_of_morphisms}.
\end{proof}

\begin{remark}
	It is not too hard to show that any \emph{separable} isogeny of AT varieties is finite \'etale. It is reasonable to conjecture that, conversely, any finite \'etale morphism $P'\to P$ with~$P$ an AT variety and~$P'$ geometrically connected should be a separable isogeny between AT varieties -- for abelian varieties this is a theorem of Lang--Serre \cite[Th\'eor\`eme~2]{lang-serre:revetements}\cite[Proposition~5.6.8]{szamuely:galois_groups_fundamental_groups}. Because we will not need to know this, we do not investigate this further here.
\end{remark}

\subsection{The quadratic Albanese variety}

For simplicity, we now suppose that the base field~$k$ is perfect, and consider a smooth curve~$Y$ over~$k$. We are going to show that there is a universal morphism from~$Y$ to an AT variety, which we will call the \emph{quadratic Albanese variety} of~$Y$. In some cases, the quadratic Albanese variety admits a particularly simple description, which should help give the general idea:
\begin{example}\leavevmode
	\begin{itemize}
		\item If~$Y=\bP^1\smallsetminus\{0,1,\infty\}$ is the thrice-punctured line, then its quadratic Albanese variety is~$\bG_m^2$.
		\item If~$Y=E\smallsetminus\{\infty\}$ is a once-punctured elliptic curve, then its quadratic Albanese variety is the $\bG_m$-torsor corresponding to the line bundle~$\cO(\infty)$ on~$E$.
		\item If~$Y=X$ is proper of genus~$\geq1$ and has a $k$-rational point, and the Galois action on~$\NS(J_\kbar)$ is trivial, then there are line bundles~$L_1,\dots,L_{\rho-1}$ on~$J$ whose restrictions to~$X$ are trivial, and whose images in~$\NS(J)$ form a basis of the kernel of the restriction map $\NS(J)\to\NS(X)=\bZ$. The fibre product of the $\bG_m$-torsors corresponding to~$L_1,\dots,L_{\rho-1}$ is a $\bG_m^{\rho-1}$-torsor over~$J$, and it is the quadratic Albanese variety of~$X$.
	\end{itemize}
\end{example}

For the general construction, let~$X$ be the smooth compactification of~$Y$, and let~$D=X\smallsetminus Y$ be the reduced boundary divisor. We have the following description of morphisms from~$Y$ to $\bG_m$-torsors over abelian varieties.

\begin{lemma}\label{lem:morphisms_to_Gm_torsor}
	Let~$P$ be a $\bG_m$-torsor over an abelian variety~$A$, corresponding to a line bundle~$L$ on~$A$. Then morphisms $f\colon Y\to P$ are in bijection with triples $(\hat f,E,\alpha)$ consisting of
	\begin{itemize}
		\item a morphism $\hat f\colon X\to A$;
		\item a divisor~$E\in\Div(X)$ supported on~$D$; and
		\item an isomorphism~$\alpha\colon \hat f^*L\cong\cO(E)$ of line bundles on~$X$.
	\end{itemize}
\end{lemma}
\begin{proof}
	Given a morphism $f\colon Y\to P$, the composition
	\[
	Y \xrightarrow{f} P \to A
	\]
	extends uniquely to a morphism $\hat f\colon X\to A$. The pullback~$\hat{f}^*P$ is a $\bG_m$-torsor over~$X$ equipped with a trivialisation over~$Y$ (coming from the induced map~$Y\to \hat f^*P$). This trivialisation gives a rational section~$s$ of~$\hat f^*L$ which is regular and non-vanishing on~$Y$, so the divisor~$E\coloneqq\div(s)$ is supported on~$D$. There is thus a unique isomorphism $\alpha\colon \hat f^*L\cong\cO(E)$ of line bundles on~$X$ taking the section~$s$ of~$\hat f^*L$ to the section of~$\cO(E)$ given by the rational function~$1$.
	
	This describes how to associate a triple~$(\hat f,E,\alpha)$ to a morphism~$f$. The construction is clearly reversible, and so yields a bijection.
\end{proof}

Using this lemma, we can now construct the quadratic Albanese variety of~$Y$. Let~$J=\bPic^0(X)$ denote the Jacobian variety of the projective curve~$X$, and let~$J^1=\bPic^1(X)$ denote the $J$-torsor of degree~$1$ line bundles on~$X$. Let~$\alb\colon X\to J^1$ denote the Abel--Jacobi map, and let~$\Lambda$ denote the kernel of the homomorphism
\begin{equation}\label{eq:map_from_neron-severi}
	\NS(J_{\kbar}) \oplus \bZ\cdot D(\kbar) \to \bZ
\end{equation}
given by\footnote{To make sense of this formula we choose an arbitrary trivialisation $J^1_\kbar\simeq J_\kbar$; the homomorphism~\eqref{eq:map_from_neron-severi} is independent of this choice.} $([L],E)\mapsto \alb^*([L])-\deg(E)$. Then~$\Lambda$ is a lattice endowed with a continuous action of the Galois group~$\Gal(\kbar/k)$; let~$H$ be the torus over~$k$ whose character lattice is~$\Lambda$.

\begin{theorem}\label{thm:quadratic_albanese}
	Let~$Y$ be a smooth curve over a perfect field~$k$. Then there is an AT variety~$Q$ and a morphism $\qalb\colon Y\to Q$ of $k$-varieties which is initial among all morphisms from~$Y$ to an AT variety. Explicitly, this means that for any AT variety~$P$ and any morphism $f\colon Y\to P$, there is a unique morphism $\phi\colon Q\to P$ such that~$f=\phi\circ\qalb$.
	
	Moreover, $Q$ is a torsor under~$H$ over~$J^1$, and $Q$ and the morphism~$f$ are compatible with extensions of the base field.
\end{theorem}
\begin{proof}
	Let~$\cC_Y$ denote the category whose objects are AT varieties~$P$ equipped with a morphism $f\colon Y\to P$, with the obvious definition of morphisms. We want to show that the category~$\cC_Y$ has an initial object.
	
	We first consider the case when~$k=\kbar$ is algebraically closed, so all tori over~$k$ are split and~$J^1\cong J$. In this case, objects of~$\cC_Y$ can be described as tuples
	\[
	(A,L_1,\dots,L_r;\hat f,E_1,\dots,E_r) \,,
	\]
	where~$A$ is an abelian variety, $L_1,\dots,L_r$ are elements of~$\Pic(A)$, $\hat f\colon X\to A$ is a morphism, and $E_1,\dots,E_r$ are divisors on~$X$, supported on~$D$, such that~$\hat f^*L_i\simeq\cO(E_i)$ for all~$i$. Indeed, the line bundles~$L_1,\dots,L_r$ correspond to a~$\bG_m^r$-torsor~$P$ over~$A$, and Lemma~\ref{lem:morphisms_to_Gm_torsor} ensures that the morphism~$\hat f$ and divisors~$E_1,\dots,E_r$ determine a morphism $f\colon Y\to P$. This morphism depends on a choice of isomorphisms $\alpha_i\colon \hat f^*L_i\xrightarrow\sim\cO(E_i)$, but the resulting object~$(P,f)$ of~$\cC_Y$ is independent of these choices up to isomorphism (they just scale $f$ by an element of~$\bG_m^r(k)$).
	
	Morphisms in~$\cC_Y$ can be described in a similarly combinatorial manner. If~$(P,f)$ and~$(P',f')$ are objects of~$\cC_Y$, corresponding to tuples $(A,L_1,\dots,L_r;\hat f,E_1,\dots,E_r)$ and $(A',L'_1,\dots,L'_{r'};\hat f',E'_1,\dots,E'_{r'})$, respectively, then morphisms $(P,f)\to(P',f')$ correspond to pairs~$(\phi_T,M)$ where~$\phi_T\colon A\to A'$ is a map of $k$-varieties (automatically a translate of a homomorphism) and~$M=(a_{ij})_{i,j}$ is an~$r'\times r$ integer matrix such that $\hat f'=\phi_T\circ\hat f$ and
	\begin{equation}\label{eq:morphism_over_Y}\tag{$\ast$}
		\phi_T^*L'_i\simeq\bigotimes_{j=1}^rL_j^{\otimes a_{ij}} \quad\text{and}\quad E'_i=\sum_{j=1}^ra_{ij}E_j
	\end{equation}
	for all~$1\leq i\leq r'$. Indeed, the matrix~$M$ corresponds to a homomorphism of tori~$\phi_G\colon\bG_m^r\to\bG_m^{r'}$ (so that~$M$ is the induced homomorphism of cocharacter lattices~$\bZ^{\oplus r}\to\bZ^{\oplus r'}$). The pushout~$\phi_{G*}P$, equipped with the induced map from~$Y$, then corresponds to the tuple
	\[
	\Bigl(A,\bigotimes_{j=1}^rL_j^{\otimes a_{1j}},\dots,\bigotimes_{j=1}^rL_j^{\otimes a_{r'j}};\hat f,\sum_{j=1}^ra_{1j}E_j,\dots,\sum_{j=1}^ra_{r'j}E_j\Bigr) \,,
	\]
	while the pullback~$\phi_T^*P'$, equipped with its induced map from~$Y$, corresponds to the tuple
	\[
	\Bigl(A,\phi_T^*L'_1,\dots,\phi_T^*L'_{r'};\hat f,E'_1,\dots,E'_{r'}\Bigr) \,.
	\]
	There exists a morphism $(P,f)\to(P',f')$ in~$\cC_Y$ over~$\phi_T$ and under~$\phi_G$ if and only if condition~\eqref{eq:morphism_over_Y} holds, in which case~$\phi$ is unique.
	
	Using this combinatorial description of~$\cC_Y$, we can construct an initial object. Fix a trivialisation $J^1\cong J$ and consider the commuting diagram with exact rows
	\begin{equation}\label{diag:picard_sequences}\tag{$\dagger$}
		\begin{tikzcd}[column sep = small]
			0 \arrow[r] & \Pic^0(J) \arrow[r]\arrow[d,"\alb^*","\wr"'] & \Pic(J) \oplus \bZ\cdot D(k) \arrow[r]\arrow[d] & \NS(J) \oplus \bZ\cdot D(k) \arrow[r]\arrow[d,"\eqref{eq:map_from_neron-severi}"] & 0 \\
			0 \arrow[r] & \Pic^0(X) \arrow[r] & \Pic(X) \arrow[r] & \NS(X) \arrow[r] & 0 \,,
		\end{tikzcd}
	\end{equation}
	in which the middle vertical arrow is given by $(L,E)\mapsto\alb^*L\otimes\cO(-E)$. The leftmost vertical arrow is pullback along~$\alb$, which is an isomorphism, and so the middle and rightmost vertical arrows have the same kernel~$\Lambda$. Let~$(L_1,E_1),\dots,(L_r,E_r)$ be a basis of~$\Lambda$ considered as a submodule of~$\bZ\cdot D(k)\oplus\Pic(J)$.
	
	Let~$(Q,f)$ be the object of~$\cC_Y$ corresponding to the tuple
	\[
	(J,L_1,\dots,L_r;\alb,E_1,\dots,E_r) \,.
	\]
	We claim that~$(Q,f)$ is initial. To see this, if~$(P',f')$ is another object, corresponding to a tuple~$(A,L'_1,\dots,L'_{r'};\hat f',E'_1,\dots,E'_{r'})$, then the morphism $\hat f'\colon X\to A$ factors uniquely as~$\phi_T\circ\alb$ for some morphism $\phi_T\colon J\to A$. Because this tuple defines an object of~$\cC_Y$, we have $\alb^*\phi_T^*L'_i=\hat f^{\prime*}L'_i\simeq\cO(E'_i)$ for all~$i$, and so each pair~$(\phi_T^*L'_i,E'_i)$ is an element of the kernel of the middle vertical arrow in~\eqref{diag:picard_sequences}. So we can express it in terms of our chosen basis~$(L_j,E_j)$ as
	\[
	(\phi_T^*L'_i,E'_i) = \sum_{j=1}^ra_{ij}(L_j,E_j)
	\]
	for some unique integers~$(a_{ij})_{i,j}$. This identity is exactly the same as~\eqref{eq:morphism_over_Y}, and so we have found the unique morphism $(Q,f)\to(P,f')$ in~$\cC_Y$.
	
	Finally, for general perfect~$k$, note that the absolute Galois group~$\Gal(\kbar/k)$ acts on the category~$\cC_{Y_\kbar}$. It necessarily preserves the initial object~$(Q_\kbar,f_\kbar)$: for each~$\sigma\in\Gal(\kbar/k)$ there is a unique isomorphism $\sigma^*(Q_\kbar,f_\kbar)\cong(Q_\kbar,f_\kbar)$. These isomorphisms define descent data on~$(Q_\kbar,f_\kbar)$, under which it descends to an object~$(Q,f)\in\cC_Y$. It is easy to check that the universal property of~$f_\kbar$ also descends, and so $(Q,f)$ is an initial object of~$\cC_Y$ as desired. It is clear that this construction is stable under change of base field.
\end{proof}

\begin{definition}\label{def:quadratic_albanese}
	We call the variety~$Q$ above the \emph{quadratic Albanese variety} of~$Y$, and call the morphism $\qalb\colon Y\to Q$ the \emph{quadratic Albanese map}.
\end{definition}

\begin{remark}\label{rmk:higher_albanese_manifolds}
	Hain and Zucker have developed a very general theory of the \emph{higher Albanese manifolds} $\Alb_n(Y)$ associated to a complex algebraic variety~$Y$ \cite{hain-zucker:unipotent_variations}. These are constructed from the Hodge structure on nilpotent quotients of the fundamental group of~$Y$, and have the property that there is a holomorphic map $Y^\an\to\Alb_n(Y)$ realising the maximal $n$-step unipotent quotient of the topological fundamental group of~$Y$. One can show that our quadratic Albanese variety of~$Y$ sits between the first and second higher Albanese manifolds: there are canonical holomorphic submersions
	\[
	\Alb_2(Y) \twoheadrightarrow Q^\an \twoheadrightarrow \Alb_1(Y) \,,
	\]
	neither of which is a biholomorphism in general. We further remark that although~$Q^\an$ and~$\Alb_1(Y)$ are always algebraic ($\Alb_1(Y)$ is the generalised Jacobian of~$Y$), $\Alb_2(Y)$ is not an algebraic variety in general. In fact, when~$Y=X$ is a smooth projective curve of genus~$\geq4$, then~$\Alb_2(X)$ is \emph{never} an algebraic variety (see Corollary~\ref{cor:U2_not_realisable}). Rogov has proved a similar result for $\Alb_n(X)$ whenever~$n\geq3$ and~$g\geq2$ \cite{rogov:o-minimal_higher_albanese}.
\end{remark}

\subsection{Fundamental groups of AT varieties}

Consider an AT variety~$P$ over a field~$k$ with a chosen $k$-rational base point~$\tilde0\in P(k)$. The image of~$\tilde0$ under the projection $P\to T$ induces a trivialisation $T\cong A$ of the torsor~$T$, so we are free to assume that~$T=A$, with~$\tilde0$ lying in the fibre over the identity~$0\in A(k)$. This fibre is then a trivial $G$-torsor over~$\Spec(k)$, so we may identify it with~$G$ in such a way that~$\tilde0\in P(k)$ corresponds to the identity~$1\in G(k)$.

When~$k=\bC$, the sequence $G^\an \hookrightarrow P^\an \twoheadrightarrow A^\an$ is a locally trivial fibre sequence, so the long exact sequence of homotopy groups provides a short exact sequence
\begin{equation}\label{eq:homotopy_sequence}
1 \to \pi_1(G^\an;1) \to \pi_1(P^\an;\tilde0) \to \pi_1(A^\an;0) \to 1 \,.
\end{equation}
The outer two groups in~\eqref{eq:homotopy_sequence} are both abelian, so can be identified with the homology groups~$\rH_1(G^\an,\bZ)$ and $\rH_1(A^\an,\bZ)$, respectively.

\begin{lemma}\label{lem:central_extension}
	\eqref{eq:homotopy_sequence} is a central extension, and the multiplication map~$\pi_1(G^\an;1)\times\pi_1(P^\an;\tilde0)\to\pi_1(P^\an;\tilde0)$ is the map induced by the action $\rho\colon G\times P\to P$.
\end{lemma}
\begin{proof}
	We use a version of the Eckmann--Hilton argument. Because the fundamental group functor preserves products, the action does indeed induce a homomorphism
	\[
	\rho_*\colon \pi_1(G^\an;1)\times\pi_1(P^\an;\tilde0)\to\pi_1(P^\an;\tilde0) \,.
	\]
	The restriction of~$\rho_*$ to~$1\times\pi_1(P^\an;\tilde0)$ is the identity (because the same is true of~$\rho$), and the restriction to~$\pi_1(G^\an;1)\times1$ is the map~$i_*$ induced by the inclusion~$G\hookrightarrow P$ of the fibre above~$0\in A(\bQ)$, i.e.~the left-hand arrow in~\eqref{eq:homotopy_sequence}. Because~$\rho_*$ is a homomorphism, we have
	\[
	\rho_*(\gamma,\delta) = \rho_*((\gamma,1)\cdot(1,\delta))=i_*(\gamma)\cdot\delta \,,
	\]
	i.e.~$\rho_*$ is the multiplication map. The fact that the multiplication map is a homomorphism of groups forces that the image of~$i_*$ is a central subgroup.
\end{proof}

Because the extension~\eqref{eq:homotopy_sequence} is a central extension, we can form the commutator pairing $\bigwedge^2\rH_1(A^\an,\bZ)\to\rH_1(G^\an,\bZ)$. We have the following criterion for surjectivity of the commutator pairing.

\begin{lemma}\label{lem:fundamental_group_of_Gm-torsor}
	Let~$A$ be an abelian variety over~$\bC$, and let~$P$ be the pointed $\bG_m$-torsor over~$A$ corresponding to a line bundle~$L\in\Pic(A)$. If the image of~$L$ in~$\NS(A)$ is primitive\footnote{not a non-trivial multiple of any element of~$\NS(A)$}, then the commutator pairing
	\[
	\bigwedge\nolimits^{\!\!2}\rH_1(A^\an,\bZ) \to \rH_1(\bC^\times,\bZ) = \bZ(1)
	\]
	coming from the extension~\eqref{eq:homotopy_sequence} is surjective.
\end{lemma}
\begin{proof}
	By GAGA, we may work in the complex-analytic category. Write~$A^\an=V/\Lambda$ for a complex vector space~$V$ and full rank sublattice~$\Lambda\subset V$. By the Appell--Humbert Theorem \cite[p.~20]{mumford:abelian_varieties}, $\Pic(A)$ is isomorphic to the group of pairs~$(H,\alpha)$, where~$H$ is a Hermitian form on~$V$ whose imaginary part~$E\coloneqq\Im H$ is integer-valued on~$\Lambda$, and~$\alpha$ is a function from~$\Lambda$ to the unit circle~$U(1)$ satisfying the identity
	\begin{equation}\label{eq:appell-humbert_alpha_identity}
		\alpha(u+v) = e^{\pi iE(u,v)}\alpha(u)\alpha(v)
	\end{equation}
	for all~$u,v\in\Lambda$. The subgroup~$\Pic^0(A)$ corresponds to the subgroup of pairs where~$H=0$, and so~$\NS(A)$ is the group of Hermitian forms~$H$ such that~$E$ is integer-valued on~$\Lambda$.
	
	Let~$(H,\alpha)$ be the pair corresponding to~$L$. Our assumption that~$[L]$ is primitive in~$\NS(A)$ ensures that
	\[
	E\colon \bigwedge\nolimits^{\!\!2}\Lambda \to \bZ
	\]
	is surjective; we are going to show that the commutator pairing from~\eqref{eq:homotopy_sequence} is equal to~$2\pi iE$, which completes the proof. For this, the Appell--Humbert Theorem tells us that the total space of~$L$ is the quotient of~$\bC\times V$ by the action of $\Lambda$ given by
	\[
	u\colon (\lambda,z) \mapsto (\alpha(u)e^{\pi H(z,u)+\frac12\pi H(u,u)}\cdot\lambda,z+u) \,.
	\]
	So~$P^\an$ is the quotient of~$\bC^\times\times V$ by the same action. We are free to assume that the base point on~$P^\an$ is the image of~$(1,0)\in\bC^\times\times V$.
	
	Now let~$\Pi$ denote the group whose elements are pairs~$(s,u)$ where~$u\in\Lambda$ and~$s\in i\bR$ is an element such that~$e^s=\alpha(u)$. The group law on~$\Pi$ is given by
	\[
	(s,u)\cdot(t,v) \coloneqq (s+t+\pi iE(u,v),u+v) \,.
	\]
	The group~$\Pi$ is a central extension of~$\Lambda$ by~$\bZ(1)=2\pi i\bZ$, and it acts on~$\bC\times V$ from the right by
	\[
	(w,z)\cdot(s,u) \coloneqq (w+s+\pi H(z,u)+\frac12\pi H(u,u),z+u) \,.
	\]
	This action is properly discontinuous, and the quotient $(\bC\times V)/\bZ(1)$ is~$\bC^\times\times V$, with the same induced action of~$\Lambda$ as before. Thus we have $(\bC\times V)/\Pi=P^\an$, i.e.~$\bC\times V$ is the universal covering of~$P^\an$, with group of right deck transformations~$\Pi$. In particular, we have $\pi_1(P^\an;\tilde0)=\Pi$, and the commutator pairing in~$\Pi$ is equal to~$2\pi iE(u,v)$ so we are done.
\end{proof}

\begin{remark}
	We are following the functional conventions for path-composition in fundamental groups, i.e.~$\gamma_1\gamma_2$ is the path ``$\gamma_2$ followed by $\gamma_1$''. This is opposite to the usual convention in topology. Under topologists' conventions, the commutator pairing in~\eqref{eq:homotopy_sequence} is $-2\pi iE$ rather than~$2\pi iE$.
\end{remark}

Now we study the fundamental group of quadratic Albanese varieties.

\begin{proposition}\label{prop:surjection_on_pi1}
	Let~$Y$ be a smooth curve over the complex numbers, and let~$\qalb\colon Y\to Q$ be the quadratic Albanese map. Then the induced map
	\begin{equation}\label{eq:map_on_topological_pi1}\tag{$\ast$}
		\qalb_*\colon \pi_1(Y^\an) \to \pi_1(Q^\an)
	\end{equation}
	on topological fundamental groups is surjective (for any choice of base points).
\end{proposition}
\begin{proof}
	The result to be proved is independent of the choice of base points, so let~$b\in Y^\an$ be a choice of base point, and use~$\tilde0\coloneqq \qalb(b)$ as a base point for~$Q^\an$. We fix notation as in the proof of Theorem~\ref{thm:quadratic_albanese}, so~$Q$ is a torsor under the torus~$H$ over $J^1=\bPic^1(X)$. The image of~$\tilde0$ in~$J^{1,\an}$ is the class of the line bundle~$\cO(b)$ on~$X$, so we may identify~$J^1=J$ so that the map $\alb\colon X\to J$ is the Abel--Jacobi map based at~$b$.
	
	This implies that the composition
	\[
	\pi_1(Y^\an;b) \to \pi_1(Q^\an;\tilde0) \to \pi_1(J^\an;0)
	\]
	is surjective: it is equal to the composition of the surjections $\pi_1(Y^\an;b) \twoheadrightarrow \pi_1(X^\an;b)$ and $\pi_1(X^\an;b)\twoheadrightarrow \pi_1(J^\an;0)$. Thus the image of~\eqref{eq:map_on_topological_pi1} is a central extension of~$\pi_1(J^\an;0)$ by a subgroup~$B\subseteq\pi_1(H^\an;1)\simeq\bZ^r$. In order to show that~\eqref{eq:map_on_topological_pi1} is surjective, it suffices to show that~$\chi_*(B)=\bZ$ for every surjective homomorphism $\chi_*\colon\pi_1(H^\an;1)\twoheadrightarrow \bZ$. The homomorphism~$\chi_*$ is induced by a character~$\chi\colon H\to\bG_m$, and surjectivity of~$\chi_*$ implies that~$\chi$ is primitive. Letting~$P=\chi_*Q$ denote the pushout and $f\colon Y\to P$ the induced pointed morphism, showing that $\chi_*(B)=\bZ$ is equivalent to showing that the induced homomorphism
	\begin{equation}\label{eq:map_on_topological_pi1_pushout}\tag{$\ast'$}
	f_*\colon \pi_1(Y^\an;b) \to \pi_1(P^\an;\tilde 0)
	\end{equation}
	is surjective.
	
	By the description of the character lattice of~$H$ in Theorem~\ref{thm:quadratic_albanese}, the primitive character~$\chi$ corresponds to an element~$([L],E)$ of the kernel of the homomorphism $\NS(J)\oplus\bZ\cdot D\to\bZ$. By the proof of Theorem~\ref{thm:quadratic_albanese}, there is a unique line bundle~$L\in\Pic(J)$ with N\'eron--Severi class~$[L]$ such that~$\alb^*L\cong\cO(E)$ as line bundles on~$X$. If we enumerate~$D$ as~$x_1,\dots,x_n$, then we may write~$E=\sum_{m=1}^na_mx_m$ for some integers~$a_m$ and $[L]=[L_0^{\otimes a}]$ for some integer~$a$ and primitive element~$[L_0]$ of~$\NS(J)$. Since~$([L],E)$ is primitive, we must have that~$a$ and~$a_1,\dots,a_n$ have no common factor.
	
	On the one hand, the image of~\eqref{eq:map_on_topological_pi1_pushout} surjects onto~$\pi_1(J^\an;0)$, and so also contains the commutator subgroup of~$\pi_1(P^\an;\tilde0)$. By Lemma~\ref{lem:fundamental_group_of_Gm-torsor}, this commutator subgroup is~$a\bZ(1)\subseteq\bZ(1)=\pi_1(\bC^\times;1)$.
	
	On the other hand, the pullback~$\alb^*L$ is isomorphic to~$\cO(E)$, with the section of~$\alb^*L$ over~$Y$ induced by~$f$ corresponding to the section of~$\cO(E)$ given by the rational function~$1$. It follows that locally around each point~$x_m\in D$, $f^*P^\an$ is biholomorphic to~$\bC^\times\times\bD$, and the induced map $Y^\an\to f^*P^\an$ is the map $\bD^*\to\bC^\times\times\bD$ given by $z\mapsto(z^{a_m},z)$, where~$\bD$ is the complex unit disc and~$\bD^*\subset\bD$ is the punctured disc. It follows from this description that if~$\gamma_m\in\pi_1(Y^\an;b)$ is conjugate to a small positively oriented loop around the puncture~$x_m$, then the image of~$\gamma_m$ in~$\pi_1(f^*P^\an;f^{\prime-1}(\tilde0))$ is~$2\pi ia_m\in\bZ(1)=\pi_1(\bC^\times;1)$. This implies that the image of~\eqref{eq:map_on_topological_pi1_pushout} contains~$a_m\bZ(1)$.
	
	All told, we have shown that the image of~\eqref{eq:map_on_topological_pi1_pushout} surjects onto~$\pi_1(J^\an;0)$ and its intersection with~$\pi_1(\bC^\times;1)=\bZ(1)$ contains~$a\bZ(1)$, as well as~$a_m\bZ(1)$ for all $1\leq m\leq n$. Since $a$ and the integers~$a_m$ have no common factor, this means that~\eqref{eq:map_on_topological_pi1_pushout} is surjective, which is what we wanted to prove.
\end{proof}

By the usual comparison theorems, Proposition~\ref{prop:surjection_on_pi1} implies that the induced map
\[
\qalb_*\colon \pi_1^{\bQ_p}(Y_\Kbar;b) \to \pi_1^{\bQ_p}(Q_\Kbar;f(b))
\]
on $\bQ_p$-pro-unipotent \'etale fundamental groups is surjective for any smooth curve~$Y$ defined over a characteristic~$0$ field~$K$. In other words, $Q$ realises a quotient of $\pi_1^{\bQ_p}(Y_\Kbar;b)$. When~$K$ is a number field, it turns out that this quotient is~$U^p_\AT$: the maximal quotient which is an extension of a representation which is pure of weight~$-1$ by an Artin--Tate representation of weight~$-2$

\begin{proposition}\label{prop:artin-tate_quotient_realised}
	Let~$Y$ be a pointed smooth curve over a number field~$K$ with a $K$-rational base point~$b\in Y(K)$. Then the quotient of~$\pi_1^{\bQ_p}(Y_\Kbar;b)$ induced by the quadratic Albanese map $\qalb\colon Y\to Q$ is~$U_\AT^p$.
\end{proposition}
\begin{proof}
	We follow the argument of \cite[Lemma~3.2]{balakrishnan-dogra:quadratic_chabauty_1}. We may suppose that~$Y$ is not $\bP^1$, as otherwise all the fundamental groups appearing are trivial. By the usual description of fundamental groups of curves (see e.g.~\cite[\S2.2]{balakrishnan-dogra:quadratic_chabauty_1} in the projective case, or \cite[Theorem~A.1]{betts-litt:weight-monodromy} for a much more general version), the $-1$st and $-2$nd graded pieces of the weight filtration on~$\pi_1^{\bQ_p}(Y_\Kbar;b)$ are given by
	\begin{align*}
		\gr^\rW_{-1}\pi_1^{\bQ_p}(Y_\Kbar;b) &= \rH_1^\et(J_\Kbar,\bQ_p) \\
		\gr^\rW_{-2}\pi_1^{\bQ_p}(Y_\Kbar;b) &= \coker\bigl(\bQ_p(1) \hookrightarrow \bigwedge\nolimits^{\!\!2}\rH_1^\et(J_\Kbar,\bQ_p)\oplus\bQ_p(1)\cdot D(\Kbar)\bigr) \,,
	\end{align*}
	where the inclusion appearing in the second line is induced from the dual of the cup product map and the sum of the elements of~$D(\Kbar)$. By the Tate Conjecture for $\rH^2$ of abelian varieties \cite[Theorem~(b)]{faltings:endlichkeitssaetze}, the dimension of the largest Artin--Tate subrepresentation of~$\bigwedge\nolimits^{\!2}\rH_1^\et(J_\Kbar,\bQ_p)$ is equal to~$\rank(\NS(J_\Kbar))$. This implies that
	\[
	\dim_{\bQ_p}U_\AT^p = 2g + \rank(\NS(J_\Kbar)) + \deg(D) - 1 \,.
	\]
	On the other hand, the quadratic Albanese map induces a quotient of~$U_\AT^p$ which is an extension of~$\pi_1^{\bQ_p}(J_\Kbar;0)$ by~$\pi_1^{\bQ_p}(H_\Kbar;1)$, which have dimensions~$2g$ and $\rank(\NS(J_\Kbar))+\deg(D)-1$, respectively. So the induced quotient must be exactly~$U_\AT^p$.
\end{proof}

There is a variant of all of the above theory for AT varieties where the torus~$G$ is split (isomorphic to~$\bG_m^r$). Namely, there is a universal morphism from a curve~$Y$ to an AT variety~$Q_\spl$ of this type, this quotient induces a surjection on fundamental groups, and over a number field~$K$ the quotient of~$\pi_1^{\bQ_p}(Y_\Kbar;b)$ induced by~$Q_\spl$ is the maximal quotient which is a central extension of a pure representation of weight~$-1$ by a representation of the form~$\bQ_p(1)^r$. Indeed, $Q_\spl$ is is the pushout of~$Q$ to the torus whose character lattice is~$\Lambda^{\Gal_K}$; the remaining details are easy to check.

\subsubsection{Tangential base points}\label{sss:tangential_points}

In the non-abelian Chabauty method for affine curves, as well as being allowed to consider a curve~$Y$ with a rational base point~$b$, one is also allowed to consider a curve~$Y$ with a rational tangential base point~$\vec b$. Accordingly, we should also generalise all of the results of the previous section to also cover AT varieties with a rational tangential base point. Fortunately, this turns out not to be necessary, and one can systematically replace rational tangential points on AT varieties with \emph{bona fide} rational points, as we will explain now. A reader interested in quadratic Chabauty only for projective curves can safely skip the following discussion.

For our purposes, it will be most convenient to take the perspective that a \emph{$k$-rational tangential point} on a smooth $k$-variety~$V$ means a $k\llparen t\rrparen$-rational point, where~$k\llparen t\rrparen$ is the field of Laurent series. One can make sense of the profinite \'etale fundamental group of~$V_\kbar$ based at any tangential point $\vec b\in V(k\llparen t\rrparen)$; indeed, $\vec b$ determines canonically a geometric point of~$V_\kbar$ valued in the algebraic closure of~$k\llparen t\rrparen$, and $\pi_1^\et(V_\kbar;\vec b)$ simply means the \'etale fundamental group of~$V_\kbar$ based at this geometric point. When $k=K$ has characteristic~$0$ (for simplicity), the algebraic closure of~$K\llparen t\rrparen$ is the same as the algebraic closure of the field~$F=K\llparen t^\bQ\rrparen$ of Puiseux series, and the inclusion $K\hookrightarrow F$ induces an isomorphism on absolute Galois groups by the Newton--Puiseux Theorem. The absolute Galois group $\Gal(\Kbar/K)=\Gal(\Fbar/F)$ acts on~$\Spec(\Fbar)$ in such a way that the morphism $\vec b\colon \Spec(\Fbar) \to V_\Kbar$ is Galois-equivariant, so there is an induced action of~$\Gal(\Kbar/K)$ on~$\pi_1^\et(V_\Kbar;\vec b)$. This is the same as the usual Galois action on~$\pi_1^\et(V_\Fbar;\vec b)$ (coming from $F$-rationality of~$\vec b$), and when~$\vec b=b$ is actually $K$-rational, it agrees with the usual Galois action on~$\pi_1^\et(V_\Kbar;b)$. A similar discussion holds for path-torsors.

\begin{remark}
	There are several different ways of treating tangential points, of which our approach of treating them as $k\llparen t\rrparen$-valued points is simply the most convenient for our purposes. Let us briefly say something about other approaches for making sense of tangential points.
	
	In \cite[\S15]{deligne:groupe_fondamental}, Deligne considers the case where~$V$ has a strict normal crossings compactification~$\hat V$. For each irreducible component~$D_i$ of the boundary divisor $D=\hat V\smallsetminus V$, let~$D_i^0$ be the complement of the other components in~$D_i$. One can consider the $\bG_m$-torsor $N_i^0$ on~$D_i$ corresponding to its normal bundle, and write~$N_i^{00}$ for its restriction to~$D_i^0$. Deligne then defines a ``restriction to the normal bundle'' functor, from finite \'etale coverings of~$V$ which are tame along~$D$ to finite \'etale coverings of~$N_i^{00}$ which are tame along the zero section. In particular, for any $k$-rational point~$\vec b_0\in N_i^{00}(k)$, one obtains a fibre functor as the composition
	\[
	\FEt_t(V) \to \FEt_t(N_i^{00}) \to \fSET \,,
	\]
	where the second functor is the usual fibre functor associated to~$\vec b_0$. One can show that if~$\vec b$ is a $k\llparen t\rrparen$-valued point on~$V$ whose image under $\hat V(k\llparen t\rrparen)=\hat V(k\llparen t\rrparen)\to\hat V(k[t]/(t^2))$ is a tangent vector~$\hat b_0$ based at a point on~$D_i^0$ whose image in the normal bundle is non-zero, then the fibre functor associated to~$\vec b_0$ in Deligne's construction is canonically isomorphic to the fibre functor associated to the geometric point~$\vec b$.
\end{remark}

In order to replace rational tangential points on AT varieties with actual rational points, we need to show that every rational tangential point is connected by a Galois-invariant path to an actual rational point. The target of this path is picked out by the following proposition.

\begin{proposition}\label{prop:principal_part}
	Let~$P$ be an AT variety over a field~$k$. Then there is a function
	\[
	\prin\colon P(k\llparen t\rrparen) \to P(k)
	\]
	uniquely characterised by the following properties:
	\begin{enumerate}[label = \emph{(\alph*)}, ref = (\alph*)]
		\item\label{condn:prin_naturality} $\prin$ is natural with respect to morphisms of AT varieties;
		\item\label{condn:prin_power_series} on $P(k\llbrack t\rrbrack)$, $\prin$ restricts to the function given by specialisation at~$t=0$; and
		\item\label{condn:prin_normalisation} when~$P=\bG_m$, we have $\prin(t)=1$.
	\end{enumerate}
	We call~$\prin$ the \emph{principal part}.
\end{proposition}

\begin{remark}
	It follows from the characterisation that~$\prin$ is a section of the inclusion $P(k)\hookrightarrow P(k\llparen t\rrparen)$, and if~$P=B$ is a semiabelian variety, then $\prin$ is a homomorphism of abelian groups (apply~\ref{condn:prin_naturality} to the group law~$B\times B\to B$).
\end{remark}

For a torus~$G$ over~$k$, let us consider the homomorphism
\begin{equation}\label{eq:hom_from_cocharacters}
\Hom_k(\bG_m,G) \to G(k\llparen t\rrparen)
\end{equation}
sending a $k$-rational cocharacter $\psi\colon \bG_m\to G$ to $\psi(t)$, the image of~$t\in\bG_m(k\llparen t\rrparen)$ under~$\psi$. It is easy to check that the homomorphism~\eqref{eq:hom_from_cocharacters} is natural in both the torus~$G$ and the field~$k$. That is, for any homomorphism of tori $\phi\colon G\to G'$, resp.~any embedding of fields~$\iota\colon k\hookrightarrow k'$, we have commuting squares
\begin{center}
\begin{tikzcd}[column sep=small]
	\Hom_k(\bG_m,G) \arrow[r,"\eqref{eq:hom_from_cocharacters}"]\arrow[d,"\phi_*"] & G(k\llparen t\rrparen) \arrow[d,"\phi_*"] \\
	\Hom_k(\bG_m,G') \arrow[r,"\eqref{eq:hom_from_cocharacters}"] & G'(k\llparen t\rrparen) \,,
\end{tikzcd}
\quad\text{resp.}\quad
\begin{tikzcd}[column sep=small]
	\Hom_k(\bG_m,G) \arrow[r,"\eqref{eq:hom_from_cocharacters}"]\arrow[d] & G(k\llparen t\rrparen) \arrow[d] \\
	\Hom_{k'}(\bG_{m,k'},G_{k'}) \arrow[r,"\eqref{eq:hom_from_cocharacters}"] & G_{k'}(k'\llparen t\rrparen) \,.
\end{tikzcd}
\end{center}

\begin{lemma}
	The homomorphism~\eqref{eq:hom_from_cocharacters} is injective.
\end{lemma}
\begin{proof}
	If~$\psi$ is a non-trivial cocharacter then its kernel is a finite $k$-subscheme of~$\bG_m$, and in particular the transcendental element $t\in\bG_m(k\llparen t\rrparen)$ cannot be contained in the kernel $\ker(\psi)$.
\end{proof}

We let $G(k\llparen t\rrparen)_1\subseteq G(k\llparen t\rrparen)$ denote the image of the homomorphism~\eqref{eq:hom_from_cocharacters}. For example, when $G=\bG_m$, then $\bG_m(k\llparen t\rrparen)_1=t^\bZ$ is the multiplicative subgroup generated by~$t$ inside~$k\llparen t\rrparen^\times$. This subgroup gives us a direct product decomposition
\[
G(k\llparen t\rrparen) = G(k\llparen t\rrparen)_1\times G(k\llbrack t\rrbrack)
\]
for all $k$-tori~$G$. This is a special case of the following lemma.

\begin{lemma}\label{lem:shift_by_G1}
	Let~$G$ be a torus over~$k$ and let~$P$ be a $G$-torsor over a proper variety~$T$. Then any $k\llparen t\rrparen$-point $\vec x\in P(k\llparen t\rrparen)$ can be uniquely expressed as~$g\cdot\vec x_0$ for some~$g\in G(k\llparen t\rrparen)_1$ and some~$\vec x_0\in P(k\llbrack t\rrbrack)$.
\end{lemma}
\begin{proof}
	Consider first the case when $G=\bG_m^r$ is a split torus. The image of~$\vec x$ in~$T(k\llparen t\rrparen)$ is, by the valuative criterion for properness, the generic fibre of a unique $k\llbrack t\rrbrack$-point $\hat x\in T(k\llbrack t\rrbrack)$. Let~$P_{\hat x}$ denote the pullback of~$P$ along $\hat x\colon\Spec(k\llbrack t\rrbrack)\to T$. Then~$P_{\hat x}$ is a $\bG_m^r$-torsor over~$\Spec(k\llbrack t\rrbrack)$, and is necessarily trivial because~$k\llbrack t\rrbrack$ is a local ring. Fixing a trivialisation $P_{\hat x}\simeq\bG_{m,k\llbrack t\rrbrack}^r$ gives an identification
	\[
	P_{\hat x}(k\llparen t\rrparen) \simeq \bigl(k\llparen t\rrparen^\times\bigr)^r \,,
	\]
	from which we see that $P_{\hat x}(k\llparen t\rrparen)=\bG_m^r(k\llparen t\rrparen)_1\times P_{\hat x}(k\llbrack t\rrbrack)$ (because $k\llparen t\rrparen^\times=t^\bZ\cdot k\llbrack t\rrbrack$). In particular, $\vec x$ can be expressed as $g\times\vec x_0$ for some~$g\in G(k\llparen t\rrparen)_1$ and~$\tilde x_0\in P(k\llbrack t\rrbrack)$. This expression is unique, for if we had $\vec x=g'\cdot\vec x_0'$, then the image of~$\vec x_0'$ in~$T$ would have to be~$\hat x$, and so~$\vec x_0'$ would be a $k\llbrack t\rrbrack$-point on~$P_{\hat x}$. This proves the lemma when~$G$ is split.
	
	For the general case, choose a finite separable extension $k'/k$ over which~$G$ splits. We can then write~$\vec x$ uniquely as $g\cdot\vec x_0$ for $g\in G(k'\llparen t\rrparen)_1$ and $\vec x_0\in P(k'\llbrack t\rrbrack)$. Because the action map
	\[
	G(k'\llparen t\rrparen)\times P(k'\llparen t\rrparen) \to P(k'\llparen t\rrparen)
	\]
	arises from a morphism of $k$-varieties, it is automatically equivariant for the action of~$\Gal(k'/k)$. The Galois action on~$P(k'\llparen t\rrparen)$ preserves $P(k'\llbrack t\rrbrack)$ setwise, and the Galois action on~$G(k'\llparen t\rrparen)$ preserves~$G(k'\llparen t\rrparen)_1$ setwise (because~\eqref{eq:hom_from_cocharacters} is natural in the base field), and so we deduce by unicity that both~$g$ and~$\vec x_0$ must be Galois-invariant. So $\vec x_0\in P(k'\llbrack t\rrbrack)^{\Gal(k'/k)}=P(k\llbrack t\rrbrack)$ and $g\in G(k'\llparen t\rrparen)_1^{\Gal(k'/k)}=G(k\llparen t\rrparen)_1$ by Galois descent and the fact that $G(k'\llparen t\rrparen)_1=\Hom_{k'}(\bG_{m,k'},G_{k'})$. So we are done.
\end{proof}

Lemma~\ref{lem:shift_by_G1} allows us to define the principal part function in Proposition~\ref{prop:principal_part}: if~$\vec x\in P(k\llparen t\rrparen)$ is equal to $g\cdot\vec x_0$, then we define $\prin(\vec x)\in P(k)$ to be the specialisation of~$\vec x_0$ at~$t=0$. Property~\ref{condn:prin_power_series} is clear from this definition, as is property~\ref{condn:prin_normalisation}. For property~\ref{condn:prin_naturality}, if~$\phi\colon P\to P'$ is a morphism of AT varieties over~$k$ lying under the homomorphism of tori~$\phi_G\colon G\to G'$, then for any~$\vec x=g\cdot\vec x_0\in P(k\llparen t\rrparen)$ we have $\phi(\vec x)=\phi_G(g)\cdot\phi(\vec x_0)$ where $\phi(\vec x_0)\in P'(k\llbrack t\rrbrack)$ and $\phi_G(g)\in G'(k\llparen t\rrparen)_1$ by naturality of~\eqref{eq:hom_from_cocharacters} in the torus. So $\prin(\phi(\vec x))$ is the specialisation of~$\phi(\vec x_0)$ at~$t=0$, which is~$\phi(\prin(\vec x))$. So we have shown that the function~$\prin$ we defined satisfies properties~\ref{condn:prin_naturality}--\ref{condn:prin_normalisation}.

It remains to show why~$\prin$ is uniquely characterised by these properties. Suppose that~$\prin'$ were another function satisfying the same properties. If~$P$ is an AT variety over~$k$ and $\vec x\in P(k\llparen t\rrparen)$, then we may write~$\vec x$ as $\psi(t)\cdot\vec x_0$ for some~$\vec x_0\in P(k\llbrack t\rrbrack)$ and some cocharacter $\psi\colon \bG_m \to G$. Applying~\ref{condn:prin_naturality} to the composition
\[
\bG_m\times P \xrightarrow{\psi\times\id_P} G\times P \xrightarrow{\rho} P
\]
(where~$\rho$ is the action map) shows that $\prin'(\vec x)=\psi(\prin'(g))\cdot\prin'(\vec x_0)$. But we have~$\prin'(g)=1$ by~\ref{condn:prin_normalisation} and~$\prin'(\vec x_0)$ is the specialisation of~$\vec x_0$ at~$t=0$ by~\ref{condn:prin_power_series}, and so we have~$\prin'(\vec x)=\prin(\vec x)$ as claimed. This completes the proof of Proposition~\ref{prop:principal_part}. \qed

Now that we have defined the principal part, we can return to considering the profinite \'etale fundamental group of an AT variety over a field~$K$ of characteristic~$0$.

\begin{proposition}\label{prop:galois-invariant_path}
	Let~$P$ be an AT variety over the characteristic~$0$ field~$K$. Then for any point $\vec x\in P(K\llparen t\rrparen)$, there is a canonical Galois-invariant path
	\[
	\gamma_{\vec x}\in\pi_1^\et(P_\Kbar;\vec x,\prin(\vec x))^{\Gal(\Kbar/K)} \,.
	\]
	In particular, the profinite \'etale fundamental groups $\pi_1^\et(P_\Kbar;\vec x)$ and $\pi_1^\et(P_\Kbar;\prin(\vec x))$ are canonically and Galois-equivariantly isomorphic (by conjugating along~$\gamma_{\vec x}$).
\end{proposition}
\begin{proof}
	We begin with two special cases. Suppose first that~$P=\bG_m$ and that~$\vec x=t\in K\llparen t\rrparen^\times$. Then the connected finite \'etale coverings of~$\bG_{m,\Kbar}$ are exactly given by the $n$th power maps $[n]\colon \bG_{m,\Kbar}\to \bG_{m,\Kbar}$ for~$n\geq1$. The fibre of this map over the point~$1=\prin(t)$ is the $n$th roots of~$1$, while the fibre over~$t$ is the $n$th roots of~$t$. We can define a bijection
	\[
	\sqrt[n]{t} \xrightarrow\sim \mu_n
	\]
	between these fibres by multiplying by the element $t^{-1/n}\in F=K\llparen t^\bQ\rrparen$ (the field of Puiseux series). This bijection is clearly natural in the covering and equivariant for the action of the absolute Galois group of~$F$, and so defines a Galois-invariant path
	\[
	\gamma_t\in\pi_1^\et(\bG_{m,\Kbar};t,1)^{\Gal(\Kbar/K)} \,.
	\]
	
	Suppose second that~$P$ is general, but that~$\vec x=\vec x_0$ is defined over~$K\llbrack t\rrbrack$. Because the scheme $\Spec(\Kbar\otimes_KK\llbrack t\rrbrack)$ is connected and has no non-trivial finite \'etale coverings (because $K\llbrack t\rrbrack$ is Henselian \cite[Proposition~I.4.4]{milne:etale_cohomology}), there is a unique path from its geometric generic fibre $\Spec(\Fbar)$ to its geometric special fibre $\Spec(\Kbar)$. The image of this unique path under $\vec x_0\colon \Spec(K\llbrack t\rrbrack) \to P$ is an \'etale path
	\[
	\gamma_{\vec x_0}\in\pi_1^\et(P_\Kbar;\vec x_0,\prin(\vec x_0))
	\]
	from $\vec x_0$ to its specialisation at~$t=0$. This path is necessarily Galois invariant because $\vec x_0\colon\Spec(K\llbrack t\rrbrack) \to P$ is a morphism of $K$-schemes.
	
	To handle the general case, write~$\vec x=\psi(t)\cdot\vec x_0$ for~$\vec x_0\in P(K\llbrack t\rrbrack)$ and $\psi\colon\bG_m\to G$ a cocharacter. The composition
	\[
	\bG_m\times P \xrightarrow{\psi\times\id_P} G\times P \xrightarrow{\rho} P
	\]
	induces a Galois equivariant map
	\[
	\pi_1^\et(\bG_{m,\Kbar};t,1)\times\pi_1^\et(P_\Kbar;\vec x_0,\prin(\vec x_0)) \to \pi_1^\et(P_\Kbar;\vec x,\prin(\vec x)) \,,
	\]
	and the image of~$(\gamma_t,\gamma_{\vec x_0})$ under this map is the Galois-invariant path from~$\vec x$ to~$\prin(\vec x)$ we seek.
\end{proof}

We will use this lemma in the following setup. Suppose that~$Y$ is a curve over~$K$ with a tangential base point $\vec b\in Y(K\llparen t\rrparen)$, and that~$f\colon Y\to P$ is a morphism from~$Y$ to an AT variety which induces a surjection on pro-unipotent fundamental groups. Let~$\tilde 0=\prin(f(\vec b))$ denote the principal part of $f(\vec b)$. Then we have Galois-equivariant homomorphisms
\[
\pi_1^{\bQ_p}(Y_\Kbar;\vec b) \twoheadrightarrow \pi_1^{\bQ_p}(P_\Kbar;f(\vec b)) \xrightarrow\sim \pi_1^{\bQ_p}(P_\Kbar;\tilde0)
\]
between $\bQ_p$-pro-unipotent fundamental groups, making~$\pi_1^{\bQ_p}(P_\Kbar;\tilde0)$ into a quotient of~$\pi_1^{\bQ_p}(Y_\Kbar;\vec b)$. In particular, the quotient realised by~$f$ is the fundamental group of~$P$ based at an actual $K$-rational point, rather than a tangential point. This allows us to avoid having to think in any depth about tangential points on~$P$.

\section{Isogeny geometric quadratic Chabauty}\label{s:isogeny_method}

Let~$Y$ be a smooth hyperbolic curve over the rationals~$\bQ$, and let~$\cY$ be a regular model over~$\bZ$ (by which we simply mean a regular, flat, separated~$\bZ$-scheme of finite type with generic fibre~$Y$). Let~$b$ be a $\bZ$-integral base point on~$\cY$, possibly tangential, so to every $\Gal_\bQ$-equivariant quotient~$U^p$ of~$\pi_1^{\bQ_p}(Y_\Qbar;b)$, the non-abelian Chabauty method associates a locus
\[
\cY(\bZ_p)_{U^p} \subseteq \cY(\bZ_p) \,,
\]
where~$p$ is a prime of good reduction for~$\cY$.

When the quotient~$U^p$ is realised by a morphism to a smooth variety~$V$, a natural question is whether the locus~$\cY(\bZ_p)_{U^p}$ can be described purely geometrically in terms of~$V$, without any reference to fundamental groups. For~$V=A$ an abelian variety, this is essentially the Chabauty--Coleman method. That is, if we assume for simplicity that~$Y=X$ is projective and~$\cY=\cX$ is its (proper) minimal regular model, then we have a natural commuting square
\begin{equation}\label{diag:chabauty-coleman_square}
\begin{tikzcd}
	X(\bQ) \arrow[r,hook]\arrow[d] & X(\bQ_p) \arrow[d] \\
	\bQ_p\otimes A(\bQ) \arrow[r] & \Lie(A_{\bQ_p}) \,,
\end{tikzcd}
\end{equation}
where the vertical maps are built out of the map~$f\colon X\to A$ and the logarithm map~$A(\bQ_p)\to\Lie(A_{\bQ_p})$. If~$U^p$ is the quotient realised by~$f\colon X\to A$, then the locus~$X(\bQ_p)_{U^p}$ can be described as the inverse image in~$X(\bQ_p)$ of the image of the bottom map in~\eqref{diag:chabauty-coleman_square}. Crucially, this describes $X(\bQ_p)_{U^p}$ purely in terms of~$A$ and the map $f\colon X\to A$, and not in terms of $p$-adic analysis.

\subsection{Geometric quadratic Chabauty}

In the more general setting that~$U^p$ is realised by a morphism from~$Y$ to a $\bG_m^r$-torsor~$P$ over an abelian variety~$A$, an answer of sorts is given by the geometric quadratic Chabauty method of Edixhoven and Lido \cite{edixhoven-lido:geometric_quadratic_chabauty}. Let us describe a variant of their construction, adapted to our setup. (For some comments on how our formulation compares to theirs, see Remark~\ref{rmk:comparison_vs_GQC}.) First, rather than considering all integral points on~$\cY$ simultaneously, it becomes advantageous to break up the search into smaller parts.

\begin{definition}
	We call a smooth and separated $\bZ$-scheme~$\cV$ \emph{simple} if its special fibres~$\cV_{\bF_\ell}$ are all connected, and $\cV_{\bF_\ell}(\bF_\ell)\neq\emptyset$ for all~$\ell$. This implies that the special fibres~$\cV_{\bF_\ell}$ are geometrically connected.
	
	Let~$\cY_\sm$ be the smooth locus of~$\cY\to\Spec(\bZ)$. A \emph{simple open} in~$\cY$ is an open subscheme~$\cU\subseteq\cY_\sm$ formed by deleting all but one connected component of each special fibre~$\cY_{\sm,\bF_\ell}$, such that~$\cU$ is simple (i.e.~the remaining component has an $\bF_\ell$-rational point).
\end{definition}

\begin{remark}
	The terminology of simple opens is taken from \cite{duque-rosero-hashimoto-spelier:geometric_quadratic_chabauty}, except that they only require that the special fibre~$\cU_{\bF_\ell}$ be geometrically connected, not necessarily to have an $\bF_\ell$-rational point.
\end{remark}

Because~$\cY$ is regular, all of its integral points lie in the smooth locus~$\cY_\sm$, and the components of the special fibres~$\cY_{\sm,\bF_\ell}$ which they intersect have $\bF_\ell$-points (tautologically). So we have a finite partition of~$\cY(\bZ)$ as
\[
\cY(\bZ) = \coprod_\cU\cU(\bZ) \,,
\]
the union being taken over simple opens in~$\cY$. We now focus on one simple open at a time.

We are going to study~$\cU(\bZ)$ via a map $f\colon Y\to P$ from~$Y$ to a $\bG_m^r$-torsor over an abelian variety~$A$, for which we will need to produce a suitable model of~$P$. The exact kind of model we will need is as follows.

\begin{definition}\label{def:simple_model}
	Let~$P$ be a $\bG_m^r$-torsor over an abelian variety~$A/\bQ$, and let~$\cA/\bZ$ be the N\'eron model of~$A$. A \emph{simple model} of~$P$ is a $\bG_m^r$-torsor~$\cP$ over one component~$\cT$ of~$\cA$, together with an isomorphism of its generic fibre with~$P$, such that~$\cP$ is simple.
\end{definition}

\begin{lemma}\label{lem:producing_simple_models}
	Let~$\cV$ be a simple $\bZ$-scheme, and let~$f\colon V\to P$ be a morphism from the generic fibre of~$\cV$ to a $\bG_m^r$-torsor over an abelian variety~$A$. Then there is a unique simple model~$\cP$ of~$P$ such that~$f$ is the generic fibre of a morphism
	\[
	f\colon \cV \to \cP
	\]
	of $\bZ$-schemes.
\end{lemma}
\begin{proof}
	By the N\'eron mapping property, the composition $V\xrightarrow{f} P \twoheadrightarrow A$ is the generic fibre of a unique morphism
	\[
	\hat f\colon \cV \to \cA
	\]
	of~$\bZ$-schemes. Because all of the special fibres of~$\cV$ are connected, the image of~$\hat f$ is contained in one connected component of~$\cA$, call this~$\cT$. Because~$\cV_{\bF_\ell}(\bF_\ell)\neq\emptyset$, we automatically have~$\cT_{\bF_\ell}(\bF_\ell)\neq\emptyset$ for all~$\ell$, and so~$\cT$ is simple.
	
	Since $\cT$ has geometrically connected fibres, it follows that the generic fibre map~$\Pic(\cT)\to\Pic(A)$ is an isomorphism, and so~$P$ is the generic fibre of a $\bG_m^r$-torsor $\cP$ over~$\cT$. Because~$\cP\to\cT$ is Zariski-locally trivial, it follows that~$\cP_{\bF_\ell}(\bF_\ell)\to\cT_{\bF_\ell}(\bF_\ell)$ is surjective, and so~$\cP$ is automatically simple. The pullback $\hat f^*\cP$ is a $\bG_m^r$-torsor over~$\cV$ with generic fibre~$\hat f^*P$, which is a trivial torsor. Because~$\cV$ is simple, the generic fibre map~$\Pic(\cV)\to\Pic(V)$ is also an isomorphism, and so~$\hat f^*\cP$ is a trivial torsor. Thus, the morphism $\hat f\colon\cV\to\cT$ lifts to a morphism $f'\colon\cV\to\cP$. The generic fibre of~$f'$ differs from~$f$ by an element of $\Map_\bQ(V,\bG_m^r)=\bG_m^r(\bQ)$ so, adjusting the isomorphism $\cP_\bQ\cong P$ if necessary, we are free to assume that~$f$ is the generic fibre of~$f'$. So we have shown the existence of the desired simple model~$\cP$. Uniqueness follows by reversing the construction.
\end{proof}

\begin{example}\label{ex:model_with_point}
	Suppose that~$P$ comes with a chosen rational point~$\tilde0\in P(\bQ)$. By viewing this point as a morphism $\Spec(\bQ)\to P$ and applying Lemma~\ref{lem:producing_simple_models} to~$\cV=\Spec(\bZ)$, we find that there is a unique simple model~$\cP$ of~$P$ such that~$\tilde x\in\cP(\bZ)$. We say that~$\cP$ is the \emph{canonical model} of the pointed torsor~$(P,\tilde0)$.
\end{example}

In the setting of geometric quadratic Chabauty, we are considering a morphism $f\colon Y\to P$ from the smooth curve~$Y$ to a $\bG_m^r$-torsor~$P$ over an abelian variety~$A$. For each simple open~$\cU\subseteq\cY_\sm$, Lemma~\ref{lem:producing_simple_models} provides a unique simple model~$\cP$ of~$P$ such that~$f$ is the generic fibre of a morphism
\[
f\colon \cU \to \cP
\]
of $\bZ$-schemes. Note that the set~$\cP(\bZ)$ has a particularly simple structure: it is a $\bG_m^r(\bZ)$-torsor over~$\cT(\bZ)$ (using the triviality of the class group of~$\bZ$). This allows us to define the geometric quadratic Chabauty locus.

\begin{definition}
	The \emph{geometric quadratic Chabauty locus}~$\cU(\bZ_p)_f^\GQC$ for the simple open~$\cU$ relative to the morphism $f$ is defined to be the set of~$x\in\cU(\bZ_p)$ such that~$f(x)$ lies in the closure of~$\cP(\bZ)$ in $\cP(\bZ_p)$. It clearly contains $\cU(\bZ)$. 
	
	The geometric quadratic Chabauty locus for~$\cY$ is defined to be the union over simple opens:
	\[
	\cY(\bZ_p)_f^\GQC \coloneqq \bigcup_\cU\cU(\bZ_p)_f^\GQC \,.
	\]
	It is a subset of~$\cY(\bZ_p)$ and contains~$\cY(\bZ)$.
\end{definition}

\begin{remark}\label{rmk:comparison_vs_GQC}
	Let us comment on how the geometric quadratic Chabauty method we have described above is related to the original method of Edixhoven and Lido when~$Y=X$ is projective. Where we have described an obstruction relative to an arbitrary morphism~$f\colon X\to P$, the method of Edixhoven and Lido uses one particular morphism to one particular~$P$, which is a $\bG_m^{\rho-1}$-torsor over the Jacobian~$J$ (see \cite[\S2]{edixhoven-lido:geometric_quadratic_chabauty} for the construction). Even for this~$P$, the method we have described produces a possibly finer obstruction than the original method. The key difference lies in the kinds of integral models we consider: we have used a model~$\cP$ of~$P$ which is a $\bG_m^{\rho-1}$-torsor over one component of the N\'eron model~$\cJ$, whereas Edixhoven--Lido use a model~$\cP^\EL$ which is a $\bG_m^{\rho-1}$-torsor over all of~$\cJ$. This means that the intersection $\overline{\cP^\EL(\bZ)}\cap\cU(\bZ_p)$ studied in \cite{edixhoven-lido:geometric_quadratic_chabauty} could potentially be larger than the geometric quadratic Chabauty locus~$\overline{\cP(\bZ)}\cap\cU(\bZ_p)\eqqcolon\cU(\bZ_p)_f^\GQC$ as we have defined it. On the other hand, unlike \cite[\S9.2]{edixhoven-lido:geometric_quadratic_chabauty}, we do not give any general criteria for finiteness of the locus~$\cU(\bZ_p)_f^\GQC$ (though this can presumably be done).
	
	One further remark: the $\bG_m^{\rho-1}$-torsor constructed in \cite[\S2]{edixhoven-lido:geometric_quadratic_chabauty} is similar to the split quadratic Albanese variety of~$X$, but the two are not actually the same. Besides the obvious difference (the extra factor~$m$ appearing in~\cite[\S2]{edixhoven-lido:geometric_quadratic_chabauty}), there is another distinction coming from the fact that not every $\bG_m$-torsor over~$J$ is a pullback of the Poincar\'e torsor. That is, the construction in \cite[\S2]{edixhoven-lido:geometric_quadratic_chabauty} involves considering the kernel of the composition
	\begin{equation}\label{eq:homs_to_neron-severi}
	\Hom(J,J^\vee) \to \NS(J) \to \NS(X) \,,
	\end{equation}
	in which the left-hand map sends a homomorphism $f\colon J\to J^\vee$ to the pullback $(\id,f)^*P$ of the Poincar\'e torsor $P\to J\times J^\vee$. This left-hand map is not surjective in general. Indeed, under the identification $\NS(J)=\Hom(J,J^\vee)^+$ \cite[(2.10)]{edixhoven-lido:geometric_quadratic_chabauty}, this map is the symmetrisation map $f\mapsto f+f^\vee$. In particular, if the involution $f\mapsto f^\vee$ acts trivially on~$\Hom(J,J^\vee)$ (e.g.~if $J$ is isogenous to the product of two non-isogenous non-CM elliptic curves), then the left-hand map in~\eqref{eq:homs_to_neron-severi} is the doubling map on~$\Hom(J,J^\vee)=\Hom(J,J^\vee)^+$, and is certainly not surjective.
	
	What this means concretely is that if the Galois action on~$\NS(J_\Qbar)$ is trivial and if~$\Hom(J,J^\vee)=\Hom(J,J^\vee)^+$, then the $\bG_m^{\rho-1}$-torsor constructed in \cite[\S2]{edixhoven-lido:geometric_quadratic_chabauty} is actually the $2m$th tensor power of the quadratic Albanese variety of~$X$, where~$m$ is the exponent of the product of the component groups of the N\'eron model of~$J$.
\end{remark}

\subsection{Isogeny geometric quadratic Chabauty}

The reason that we said that geometric quadratic Chabauty only gives an answer \emph{of sorts} to the question of geometrically describing the locus~$\cY(\bZ_p)_{U^p}$ when~$U^p$ is realised by a morphism to an AT variety~$P$ is that the geometric quadratic Chabauty locus~$\cY(\bZ_p)_f^\GQC$ is not actually equal to~$\cY(\bZ_p)_{U^p}$ in general. Instead, at least for the particular AT variety studied in \cite[\S2]{edixhoven-lido:geometric_quadratic_chabauty}, Duque-Rosero, Hashimoto and Spelier have proved that the geometric quadratic Chabauty locus is contained in~$X(\bQ_p)_{U^p}$, but the inclusion can sometimes be strict \cite{duque-rosero-hashimoto-spelier:geometric_quadratic_chabauty}. This phenomenon already arises for abelian Chabauty, where the locus cut out by Chabauty's argument can be a proper subset of the Chabauty--Coleman locus \cite{hashimoto-spelier:geometric_linear_chabauty}.

Thus, it is natural to speculate on the existence of a variant of the geometric quadratic Chabauty method which recovers the locus~$\cY(\bZ_p)_{U^p}$ exactly. The ``isogeny geometric quadratic Chabauty method'' which we envisage would assign to every morphism $f\colon Y\to P$ to an AT variety~$P$ and every prime number~$p$ (of good reduction or no) an obstruction locus
\[
\cY(\bZ_p)_f\subseteq\cY(\bZ_p) \,,
\]
containing~$\cY(\bZ)$. This locus should have the following properties:
\begin{itemize}
	\item (functoriality) for every morphism~$\phi\colon P\to P'$, we have
	\begin{equation}\label{eq:containment_of_loci}
		\cY(\bZ_p)_f\subseteq\cY(\bZ_p)_{\phi\circ f} \,;
	\end{equation}
	\item (isogeny-invariance) if~$\phi$ is an isogeny, then~\eqref{eq:containment_of_loci} is an equality;
	\item (relationship to geometric quadratic Chabauty) if~$G$ is split and~$T$ has a rational point, then we have
	\[
	\cY(\bZ_p)_f^\GQC \subseteq \cY(\bZ_p)_f \,;
	\]
	\item (relationship to quadratic Chabauty) if $f$ induces a surjection on pro-unipotent fundamental groups, and~$U^p$ is the associated quotient of~$\pi_1^{\bQ_p}(Y_\Qbar;b)$ for~$p$ a prime of good reduction, then we have
	\[
	\cY(\bZ_p)_f = \cY(\bZ_p)_{U^p} \,.
	\]
\end{itemize}

We will not say much about how one might try to set up such a method in general, but instead will give a precise definition in a special case, when~$P$ is a $\bG_m^r$-torsor over an abelian variety~$A$ of Mordell--Weil rank~$0$.

The key idea is a construction of a certain canonical system of isogenies out of any such~$P$. For any line bundle~$L$ on an abelian variety~$A$, the line bundles
\[
L_n \coloneqq L^{\otimes(1+n)}\otimes[-1]^*L^{\otimes(1-n)}
\]
have the property that~$[n]^*L_n\simeq L^{\otimes2n^2}$ and $[m]^*L_{mn}\simeq L_n^{\otimes m^2}$ by the theorem of the cube. Since~$\bG_m^r$-torsors are the same thing as $r$-tuples of line bundles, can use these identities to perform the following construction (using the symbol $\otimes$ to denote the operation on $\bG_m^r$-torsors induced by tensor products of line bundles).

\begin{definition}\label{def:Pn}
	Let~$P$ be a $\bG_m^r$-torsor over an abelian variety~$A$ over a field~$k$. We define~$P_n$ to be the $\bG_m^r$-torsor over~$A$ given by
	\[
	P_n \coloneqq P^{\otimes(1+n)}\otimes[-1]^*P^{\otimes(1-n)} \,.
	\]
	By the theorem of the cube, $[n]^*P_n$ and $P^{\otimes2n^2}$ are isomorphic as $\bG_m^r$-torsors over~$A$, and so we define a morphism of varieties $\beta_n\colon P\to P_n$ to be the composition
	\[
	P \to P^{\otimes 2n^2} \simeq [n]^*P_n \to P_n \,,
	\]
	where the first map is the $2n^2$th power map and the final map is the projection from the pullback. In the terminology of Definition~\ref{def:over_under}, $\beta_n$ is a morphism under~$[2n^2]\colon\bG_m^r\to\bG_m^r$ over~$[n]\colon A\to A$. In particular, it is an isogeny.
\end{definition}

\begin{remark}
	The definition of the morphism~$\beta_n$ involves a choice of isomorphism $P^{\otimes 2n^2} \simeq [n]^*P_n$, but this choice only affects~$\beta_n$ up to scaling by an element of~$\bG_m^r(k)$. In particular, the pair~$(P_n,\beta_n)$ is unique up to unique isomorphism in the category of AT varieties with a morphism from~$P$.
\end{remark}

\begin{definition}\label{def:maps_between_Pn}
	In the setting of Definition~\ref{def:Pn}, the theorem of the cube implies that~$[m]^*P_{mn}$ and~$P_n^{\otimes m^2}$ are isomorphic as $\bG_m^r$-torsors over~$A$, so we can define a morphism $\beta_{n,m}\colon P_n\to P_{mn}$ in the same way, lying over~$[m]\colon A\to A$ and under~$[m^2]\colon\bG_m^r\to\bG_m^r$. The morphisms~$\beta_{n,m}$ are unique once we require that they satisfy the identity
	\[
	\beta_{n,m}\circ\beta_n = \beta_{mn} \,.
	\]
	This identity implies that they automatically satisfy the identity
	\[
	\beta_{mn,l}\circ\beta_{n,m} = \beta_{n,lm} \,.
	\]
	In other words, the varieties~$(P_n)_{n\geq1}$ and morphisms $(\beta_{n,m})_{n,m\geq1}$ form a filtered diagram indexed by~$\bN$ with the divisibility ordering, and the morphisms~$(\beta_n)_{n\geq1}$ form a cone over this diagram with vertex~$P$.
\end{definition}

We are going to be especially interested in the filtered colimit
\[
\varinjlim_nP_n(k)
\]
taken over the morphisms~$\beta_{n,m}$, as well as the function~$\beta_\infty\colon P(k)\to \varinjlim_nP_n(k)$ induced from the morphisms~$\beta_n$. The set~$\varinjlim_nP_n(k)$ is in a natural way a torsor under~$\varinjlim_nG(k)$ over~$\varinjlim_nA(k)$, because filtered colimits preserve torsor structures. Moreover, we have $\varinjlim_nA(k)=\bQ\otimes A(k)$ (because the colimit is over the maps $[m]\colon A(k)\to A(k)$) and $\varinjlim_nG(k)=\bQ\otimes G(k)$ (because the colimit is over the maps~$[m^2]\colon G(k)\to G(k)$), so one can think of~$\varinjlim_nP_n(k)$ as some kind of ``tensor product'' of the torsor~$P(k)$ with~$\bQ$. (We will not make this assertion precise, it is intended purely motivationally.) Like any good construction, the set~$\varinjlim_nP_n(k)$ is functorial in~$P$, and the function~$\beta_\infty\colon P(k)\to\varinjlim_nP_n(k)$ is a natural transformation.

\begin{lemma}\label{lem:P_n_functorial}
	The assignment $P\mapsto P_n$ is functorial in the AT variety~$P$, and the maps $\beta_n$ and~$\beta_{n,m}$ are the components of natural transformations.
\end{lemma}
\begin{proof}
	Suppose that~$P$ is a $\bG_m^r$-torsor over~$A$ and~$P'$ is a $\bG_m^{r'}$-torsor over~$A'$, and that~$f\colon P\to P'$ is a morphism of AT varieties. Then~$f$ lies under a homomorphism $f_G\colon \bG_m^r\to\bG_m^{r'}$ and over a morphism $f_T\colon A\to A'$, the latter of which can be written as a composition~$\tau_x\circ f_A$ where~$f_A\colon A\to A'$ is a homomorphism and~$\tau_x$ is translation by $x\in A'(k)$. Recall that
	\[
	\lambda_{P'}(y)\coloneqq\tau_y^*P'\otimes P^{\prime\otimes-1} = \tau_y^*[-1]^*P'\otimes[-1]^*P^{\prime\otimes-1}
	\]
	defines a homomorphism from~$A'(k)$ to $\Pic^0(A')^r$. So
	\begin{align*}
		f_A^*\tau_{nx}^*(P'_n) &\simeq f_A^*\bigl(\tau_{nx}^*P^{\prime\otimes(1+n)}\otimes\tau_{nx}^*[-1]^*P^{\prime\otimes(1-n)}\bigr) \\
		&\simeq f_A^*\bigl(P^{\prime\otimes(1+n)}\otimes[-1]^*P^{\prime\otimes(1-n)}\otimes\lambda_{P'}(x)^{\otimes2n}\bigr) \\
		&\simeq f_A^*\bigl(\tau_x^*P^{\prime\otimes(1+n)}\otimes[-1]^*\tau_x^*P^{\prime\otimes(1-n)}\bigr) \\
		&\simeq f_T^*P^{\prime\otimes(1+n)}\otimes[-1]^*f_T^*P^{\prime\otimes(1-n)} \\
		&\simeq f_{G*}\bigl(P^{\otimes(1+n)}\otimes[-1]^*P^{\otimes(1-n)}\bigr) = f_{G*}(P_n)
	\end{align*}
	as~$G'$-torsors over~$A$. In particular, there exists a morphism $f_n\colon P_n\to P'_n$ lying over~$\tau_{nx}\circ f_A$ and under~$f_G$; this morphism is unique if we additionally require that the square
	\begin{center}
	\begin{tikzcd}
		P \arrow[r,"f"]\arrow[d,"\beta_n"] & P' \arrow[d,"\beta_n"] \\
		P_n \arrow[r,"f_n"] & P'_n
	\end{tikzcd}
	\end{center}
	commute. This implies that~$P\mapsto P_n$ is functorial and~$\beta_n$ is a natural transformation as claimed. Moreover, unicity implies that the squares
	\begin{center}
	\begin{tikzcd}
		P_n \arrow[r,"f_n"]\arrow[d,"\beta_{n,m}"] & P'_n \arrow[d,"\beta_{n,m}"] \\
		P_{mn} \arrow[r,"f_{mn}"] & P'_{mn}
	\end{tikzcd}
	\end{center}
	also commute, so the~$\beta_{n,m}$ are also natural transformations.
\end{proof}

Over~$k=\bQ$, we can extend this construction to integral models.

\begin{lemma}\label{lem:Pn_models}
	Let~$P$ be a $\bG_m^r$-torsor over an abelian variety~$A/\bQ$, and let~$\cP$ be a simple model of~$P$. Then each~$P_n$ has a unique simple model~$\cP_n$ such that $\beta_n\colon P\to P_n$ is the generic fibre of a morphism $\beta_n\colon \cP\to \cP_n$. Moreover, the morphism~$\beta_{n,m}\colon P_n\to P_{mn}$ is the generic fibre of a morphism $\beta_{n,m}\colon \cP_n\to\cP_{mn}$ for all~$m,n$.
\end{lemma}
\begin{proof}
	The first assertion is a special case of Lemma~\ref{lem:producing_simple_models}. For the second, by Lemma~\ref{lem:producing_simple_models} again there is a simple model $\cP_{mn}'$ of~$P_{mn}$ such that~$\beta_{n,m}\colon P_n\to P_{mn}$ is the generic fibre of a morphism $\beta_{n,m}\colon\cP_n\to\cP_{mn}$. This implies that $\beta_{mn}=\beta_{n,m}\circ\beta_n$ is the generic fibre of a morphism $\cP\to\cP'_{mn}$, and so~$\cP'_{mn}=\cP_{mn}$ by the unicity clause of Lemma~\ref{lem:producing_simple_models}.
\end{proof}

\begin{lemma}\label{lem:colimit_of_Pn}
	In the setting of Lemma~\ref{lem:Pn_models}, if~$A$ has Mordell--Weil rank~$0$, then
	\[
	\varinjlim_n\cP_n(\bZ)
	\]
	consists of a single point.
\end{lemma}
\begin{proof}
	We want to show that some~$\cP_n(\bZ)$ is non-empty, and that for any two elements~$\tilde x,\tilde y\in\cP_n(\bZ)$, there exists some~$m$ such that~$\beta_{n,m}(\tilde x)=\beta_{n,m}(\tilde y)$. For the first claim, since~$\cT_{mn}\supseteq m\cT_n$ and the component group of~$\cA$ is finite, there is some~$n$ such that~$\cT_n$ is the identity component of~$\cA$, and in particular~$\cT_n(\bZ)\neq\emptyset$. In particular, $\cP_n(\bZ)$ is non-empty, since it surjects onto~$\cT_n(\bZ)$.
	
	For the second claim, let~$x$ and~$y$ denote the images of~$\tilde x$ and~$\tilde y$ in~$\cT_n(\bZ)\subseteq \cA(\bZ)=A(\bQ)$. Since we are assuming that~$A(\bQ)$ is finite, there is some~$m$ such that~$mx=my$. This implies that~$\beta_{n,m}(\tilde x)$ and~$\beta_{n,m}(\tilde y)$ lie in the same fibre of the projection $\cP_{mn}\to\cT_{mn}$, so differ by an element of $\bG_m(\bZ)^r=\{\pm1\}^r$. Since~$\beta_{mn,2}$ is a morphism lying under~$[4]\colon\bG_m^r\to\bG_m^r$, we must have $\beta_{mn,2}(\beta_{n,m}(\tilde x))=\beta_{mn,2}(\beta_{n,m}(\tilde y))$, and so~$\beta_{n,2m}(\tilde x)=\beta_{n,2m}(\tilde y)$ as claimed.
\end{proof}

Now let us to return to the setting of interest, where we have a morphism $f\colon Y\to P$ where~$Y$ is a smooth curve and~$P$ is a $\bG_m^r$-torsor over an abelian variety~$A$, all defined over~$\bQ$, and we fix a regular model~$\cY$ of~$Y$ over~$\bZ$. We write~$f_n\colon Y\to P_n$ for the composition~$\beta_n\circ f$, and write~$f_\infty\colon Y(\bQ)\to\varinjlim_nP_n(\bQ)$ for the induced map into the colimit (and similarly for~$\bQ_\ell$-points). For each simple open~$\cU\subseteq\cY_\sm$, let~$\cP$ be the unique simple model of~$P$ so that~$f\colon Y\to P$ is the generic fibre of a morphism $f\colon\cU\to\cP$, and let~$\cP_n$ be the system of simple models of~$P_n$ produced by Lemma~\ref{lem:Pn_models}. It follows from the compatibility of all the maps involved that we have a commuting square
\begin{equation}\label{diag:isogeny_QC_square}
\begin{tikzcd}
	\cU(\bZ) \arrow[r,hook]\arrow[d,"f_\infty"] & \cU(\bZ_p) \arrow[d,"f_\infty"] \\
	\varinjlim_n\cP_n(\bZ) \arrow[r] & \varinjlim_n\cP_n(\bZ_p)
\end{tikzcd}
\end{equation}
for all primes~$p$.

When~$A$ has Mordell--Weil rank~$0$, let us define $W\subseteq\varinjlim_nP_n(\bQ)$ to be the union of the singleton sets~$\varinjlim_n\cP_n(\bZ)$ as~$\cU$ ranges over all simple opens in~$\cY$. This~$W$ is a finite set, and can equally be thought of as a subset of~$\varinjlim_nP_n(\bQ_p)$ because the natural map $\varinjlim_nP_n(\bQ)\to\varinjlim_nP_n(\bQ_p)$ is injective (it is a filtered colimit of injections).

\begin{definition}\label{def:isogeny_geometric_quadratic_chabauty_locus}
	Suppose that $A$ has Mordell--Weil rank~$0$, and let~$W\subseteq\varinjlim_nP_n(\bQ)$ be the finite subset defined above. We define the \emph{isogeny geometric quadratic Chabauty locus}
	\[
	\cY(\bZ_p)_f \subseteq \cY(\bZ_p)
	\]
	to be the inverse image of~$W$ under the function $f_\infty\colon\cY(\bZ_p)\to\varinjlim_nP_n(\bQ_p)$. The locus~$\cY(\bZ_p)_f$ contains the integer points~$\cY(\bZ)$, for if~$x\in\cY(\bZ)$ is integral, then $x$ lies in a simple open~$\cU\subseteq\cY$, and so~$f_\infty(x)\in W$ by the commutativity of~\eqref{diag:isogeny_QC_square}.
\end{definition}

The isogeny geometric quadratic Chabauty locus admits a decomposition in terms of the simple opens~$\cU\subseteq\cY$. For this, note that any simple open~$\cU$ is itself a regular model of~$Y$ (because of our relaxed definition of the word ``model''), and so the isogeny geometric quadratic Chabauty locus~$\cU(\bZ_p)_f$ is defined: it is simply the set of points~$x\in\cU(\bZ_p)$ such that~$f_\infty(x)$ is equal to the unique element of~$\varinjlim_n\cP_n(\bZ)$.

\begin{lemma}\label{lem:isogeny_locus_union_over_simple_opens}
	We have a decomposition
	\[
	\cY(\bZ_p)_f = \bigcup_\cU\cU(\bZ_p)_f \,,
	\]
	where the union is taken over all simple opens~$\cU\subseteq\cY$.
\end{lemma}
\begin{proof}
	One containment is easy: for each simple open~$\cU\subseteq\cY$ we have $\cU(\bZ_p)\subseteq\cY(\bZ_p)$ and $\varinjlim_n\cP_n(\bZ)\subseteq W$, which implies that $\cU(\bZ_p)_f\subseteq\cY(\bZ_p)_f$. So we have the inclusion $\bigcup_\cU\cU(\bZ_p)_f\subseteq\cY(\bZ_p)_f$.
	
	For the converse inclusion, suppose that~$x\in\cY(\bZ_p)_f$. We need to show that there is a simple open~$\cU\subseteq\cY$ such that~$x\in\cU(\bZ_p)$ and~$f_\infty(x)$ is the unique element of~$\varinjlim_n\cP_n(\bZ)$. For this, because~$f_\infty(x)\in W$, there certainly exists some simple open~$\cU\subseteq\cY$ such that $f_\infty(x)$ is the unique element of $\varinjlim_n\cP_n(\bZ)$, where~$(\cP_n)_{n\geq1}$ is the sequence of simple models of~$(P_n)_{n\geq1}$ corresponding to~$\cU$. Let~$\cU'\subseteq\cY$ be the simple open defined by~$\cU'_{\bF_\ell}=\cU_{\bF_\ell}$ for~$\ell\neq p$, and~$\cU'_{\bF_p}$ is the component of~$\cY_{\bF_p}$ containing the reduction of~$x\in\cY(\bZ_p)$. So~$x\in\cU'(\bZ_p)$, and we are going to show that~$\varinjlim_n\cP'_n(\bZ)=\varinjlim_n\cP_n(\bZ)$, where~$(\cP'_n)_{n\geq1}$ is the sequence of simple models corresponding to~$\cU'$.
	
	For this, choose some~$m$ such that~$\cP_m$ and~$\cP'_m$ are both $\bG_m^r$-torsors over the identity component~$\cA^0$, and such that $f_m(x)\in\cP_m(\bZ)$. (This is possible because $f_\infty(x)\in\varinjlim_n\cP_n(\bZ)$.) Because the generic fibre map $\Pic(\cA^0)\to\Pic(A)$ is an isomorphism, we know that~$\cP_m$ and~$\cP'_m$ are isomorphic as torsors. Fixing an isomorphism $\cP_m\xrightarrow\sim\cP'_m$, the induced map on generic fibres is an automorphism of~$P_m$, so is multiplication by some element~$\lambda\in\bG_m^r(\bQ)$. Concretely, this means that~$\cP'_m(R)=\lambda\cdot\cP_m(R)$ as subsets of~$P_m(\bQ\otimes R)$, for all rings~$R$.
	
	Now on the one hand, $\cP_m(\bZ_p)$ and~$\cP'_m(\bZ_p)$ have non-empty intersection inside~$P_m(\bQ_p)$: they both contain~$f_m(x)$ because~$f_m(x)\in\cP_m(\bZ)$ by assumption and~$f_m(x)\in\cP'_m(\bZ_p)$ because~$x\in\cU'(\bZ_p)$. Because~$\cP_m(\bZ_p)$ is a $\bG_m^r(\bZ_p)$-torsor, this implies that~$\lambda\in\bG_m^r(\bZ_p)$. On the other hand, for any prime~$\ell\neq p$, $\cP_m(\bZ_\ell)$ and~$\cP'_m(\bZ_\ell)$ have non-empty intersection inside~$P_m(\bQ_\ell)$: they both contain~$f_m(x_\ell)$ for any choice of~$x_\ell\in\cU(\bZ_\ell)=\cU'(\bZ_\ell)$. This implies that~$\lambda\in\bG_m^r(\bZ_\ell)$ also. We conclude that~$\lambda\in\bG_m^r(\bZ)$, whence~$\cP'_m=\cP_m$ as models of~$P_m$. By Lemma~\ref{lem:producing_simple_models}, we have that~$\cP'_{km}=\cP_{km}$ for all~$k\geq1$, and so $\varinjlim_n\cP'_n(\bZ)=\varinjlim_n\cP_n(\bZ)$. We conclude that~$x\in\cU'(\bZ_p)_f$ and we are done.
\end{proof}

\begin{corollary}\label{cor:isogeny_locus_union_of_geometric_loci}
	In the setup of Definition~\ref{def:isogeny_geometric_quadratic_chabauty_locus}, the isogeny geometric quadratic Chabauty locus is the union of geometric quadratic Chabauty loci:
	\[
	\cY(\bZ_p)_f = \bigcup_n\cY(\bZ_p)_{f_n}^\GQC \,.
	\]
\end{corollary}
\begin{proof}
	By Lemma~\ref{lem:isogeny_locus_union_over_simple_opens}, it suffices to prove this when~$\cY=\cU$ is simple. A point~$x\in\cU(\bZ_p)$ satisfies $f_\infty(x)\in\varinjlim_n\cP_n(\bZ)$ if and only if~$f_n(x)\in\cP_n(\bZ)$ for some~$n$. Because~$\cP_n(\bZ)$ is finite, it is closed in~$\cP_n(\bZ_p)$ and so we have that $x\in\cU(\bZ_p)_f$ if and only if $x\in\cU(\bZ_p)_{f_n}^\GQC$ for some~$n$.
\end{proof}

\begin{remark}
	Corollary~\ref{cor:isogeny_locus_union_of_geometric_loci} is why we call~$\cY(\bZ_p)_f$ the isogeny geometric quadratic Chabauty locus: it is the union of geometric quadratic Chabauty loci over the system of maps~$f_n\colon Y\to P_n$ isogenous to the original~$f$. This is presumably special to the case that~$A$ has Mordell--Weil rank~$0$ (which is the only case within the scope of Definition~\ref{def:isogeny_geometric_quadratic_chabauty_locus}).
\end{remark}

\section{Comparison with quadratic Chabauty}\label{s:comparison}

The principal advantage of this more complicated version of geometric quadratic Chabauty is that it recovers the quadratic Chabauty locus exactly, and not just a subset thereof. We will prove
\begin{theorem}\label{thm:comparison}
	Let~$Y/\bQ$ be a smooth curve with a regular model~$\cY/\bZ$, and let $f\colon Y\to P$ be a morphism where~$P$ is a $\bG_m^r$-torsor over an abelian variety~$A$ of Mordell--Weil rank~$0$. Suppose that~$f$ induces a surjection on pro-unipotent fundamental groups, and fix a $\bQ$-rational base point~$b$ on~$Y$ (possibly tangential).
	
	Then
	\[
	\cY(\bZ_p)_{U^p} = \cY(\bZ_p)_f
	\]
	for every prime~$p$ of good reduction for~$(\cY,b)$, where~$U^p$ is the quotient realised by~$f$. Moreover, we have~$c(U^p)=\dim(P)$, where~$c(U^p)$ is the number of independent Coleman functions cutting out~$\cY(\bZ_p)_{U^p}$.
\end{theorem}

\begin{corollary}[{cf.~\cite[Theorem~A]{duque-rosero-hashimoto-spelier:geometric_quadratic_chabauty}}]
	In the setup of Theorem~\ref{thm:comparison}, the quadratic Chabauty locus contains the geometric quadratic Chabauty locus:
	\[
	\cY(\bZ_p)_{U^p} \supseteq \cY(\bZ_p)_f^\GQC \,.
	\]
\end{corollary}

\subsection{Selmer schemes and Kummer maps}\label{ss:non-ab_Chab}

Before we begin the proof of Theorem~\ref{thm:comparison}, let us recall the construction of the non-abelian Chabauty locus, in slightly more generality than usual. Let~$V/\bQ$ be a smooth geometrically connected quasi-projective variety with a $\bQ$-rational base point~$b$, possibly tangential, and let $\cV/\bZ$ be a regular model of~$V$ (a regular, flat, separated~$\bZ$-scheme of finite type with generic fibre~$V$). We say that a prime~$p$ is of \emph{good reduction} for~$(\cV,b)$ if~$\cV_{\bZ_p}$ is the complement of a relative normal crossings divisor in a smooth proper $\bZ_p$-scheme, and~$b$ is $\bZ_p$-integral on~$\cV$. Here, $\bZ_p$-integrality means that $b\in\cV(\bZ_p)$ if~$b$ is non-tangential, and means~$b\in\cV(\bZ_p\llparen t\rrparen)$ if~$b$ is tangential, where~$\bZ_p\llparen t\rrparen\coloneqq\bZ_p\llbrack t\rrbrack[t^{-1}]$ is the ring of~$\bZ_p$-integral Laurent series.

Fix a good reduction prime~$p$ and let~$\pi_1^{\bQ_p}(V_\Qbar;b)$ be the $\bQ_p$-pro-unipotent \'etale fundamental group of~$V_\Qbar$ based at~$b$. It is a finitely generated $\bQ_p$-pro-unipotent group, and comes with a continuous action of the Galois group~$\Gal_\bQ$ (in the sense of \cite[Definition--Lemma~4.1]{betts:motivic_local_heights}, for example). It also comes with a decreasing and separated $\Gal_\bQ$-invariant \emph{weight filtration}
\[
\pi_1^{\bQ_p}(V_\Qbar;b) = \rW_{-1}\pi_1^{\bQ_p}(V_\Qbar;b) \supseteq \rW_{-2}\pi_1^{\bQ_p}(V_\Qbar;b) \supseteq \dots \,.
\]
The weight filtration is defined as follows. Let~$\hat V$ be a smooth projective normal crossings compactification of~$V$. Then~$\rW_{-2}\pi_1^{\bQ_p}(V_\Qbar;b)$ is defined to be the kernel of the map
\[
\pi_1^{\bQ_p}(V_\Qbar;b) \to \pi_1^{\bQ_p}(\hat V_\Qbar;b)^\ab \,,
\]
and thereafter for~$n\geq3$, $\rW_{-n}\pi_1^{\bQ_p}(V_\Qbar;b)$ is generated by the commutator subgroups $[\rW_{-i}\pi_1^{\bQ_p}(V_\Qbar;b),\rW_{-j}\pi_1^{\bQ_p}(V_\Qbar;b)]$ for~$i+j=n$. The graded pieces $\gr^\rW_{-n}\pi_1^{\bQ_p}(V_\Qbar;b)\coloneqq \rW_{-n}\pi_1^{\bQ_p}(V_\Qbar;b)/\rW_{-n-1}\pi_1^{\bQ_p}(V_\Qbar;b)$ are all abelian unipotent groups (i.e.~vector groups).

\begin{lemma}\label{lem:semisimplicity}
	Each $\gr^\rW_{-n}\pi_1^{\bQ_p}(V_\Qbar;b)$ is a semisimple representation of~$\Gal_\bQ$.
\end{lemma}
\begin{proof}
	When~$V$ is a curve, this is \cite[Lemma~6.0.1]{betts:weight_filtrations}; we indicate how to generalise this to arbitrary~$V$. Let~$b_0=b$ if~$b$ is not tangential, and let~$b_0$ be the specialisation of~$b\in\hat V(\bQ\llparen t\rrparen)=\hat V(\bQ\llbrack t\rrbrack)$ at~$t=0$ if~$b$ is tangential. Let~$J$ be the Albanese variety of~$(\hat V,b_0)$, with Albanese map $\alb\colon \hat V\to J$. The Albanese map realises the abelianisation of the fundamental group of~$\hat V_\Qbar$ \cite[p.~331]{griffiths-harris:principles_algebraic_geometry}, so we have
	\[
	\gr^\rW_{-1}\pi_1^{\bQ_p}(V_\Qbar;b) = \pi_1^{\bQ_p}(\hat V_\Qbar;b)^\ab = \pi_1^{\bQ_p}(J_\Qbar;\alb(b)) = \pi_1^{\bQ_p}(J_\Qbar;0) = V_pJ \,,
	\]
	which is semisimple by \cite[Theorem~(a)]{faltings:endlichkeitssaetze}. This proves the lemma for~$n=1$.
	
	For~$n=2$, according to \cite[Theorem~8]{betts-litt:weight-monodromy}, $\gr^\rW_{-2}\pi_1^{\bQ_p}(V_\Qbar;b)$ is isomorphic to a quotient of
	\begin{equation}\label{eq:weight_minus-2}\tag{$\ast$}
	\bigwedge\nolimits^{\!\!2}V_pJ\oplus\bQ_p(1)^{\oplus \pi_0(D_\Qbar)} \,,
	\end{equation}
	where~$\pi_0(D_\Qbar)$ denotes the set of geometric components of the boundary divisor $D\coloneqq\hat V\smallsetminus V$. Because the class of semisimple representations is closed under tensor products \cite[Proposition~IV.5.2]{chevalley:lie_groups_3}, direct sums and quotients and contains all Artin--Tate representations, it follows that~\eqref{eq:weight_minus-2} is semisimple, and hence so too is~$\gr^\rW_{-2}\pi_1^{\bQ_p}(V_\Qbar;b)$.
	
	For~$n\geq3$, we proceed inductively, using that the commutator maps exhibit $\gr^\rW_{-n}\pi_1^{\bQ_p}(V_\Qbar;b)$ as a quotient of
	\[
	\bigl(\gr^\rW_{-1}\pi_1^{\bQ_p}(V_\Qbar;b)\otimes\gr^\rW_{1-n}\pi_1^{\bQ_p}(V_\Qbar;b)\bigr)\oplus\bigl(\gr^\rW_{-2}\pi_1^{\bQ_p}(V_\Qbar;b)\otimes\gr^\rW_{2-n}\pi_1^{\bQ_p}(V_\Qbar;b)\bigr) \,,
	\]
	which is semisimple by inductive assumption.
\end{proof}

Now fix a $\Gal_\bQ$-equivariant quotient~$U^p$ of~$\pi_1^{\bQ_p}(V_\Qbar;b)$, not necessarily finite-dimensional, and endow~$U^p$ with the filtration~$\rW_\bullet U^p$ induced from the weight filtration on~$\pi_1^{\bQ_p}(V_\Qbar;b)$.

\begin{lemma}\label{lem:Up_properties}
	$U^p$ satisfies the following properties:
	\begin{itemize}
		\item the $\Gal_\bQ$-action on~$U^p$ is ramified at only finitely many primes~$\ell$;
		\item the Lie algebra $\Lie(U^p)$ is pro-crystalline at~$p$;
		\item $\gr^\rW_{-n}U^p$ is pure of weight~$-n$ at all primes~$\ell$ (in the sense of the Weight--Monodromy Conjecture, see \cite[Definition~3]{betts-litt:weight-monodromy} for example).
	\end{itemize}
	In particular, we have $\bigl(\gr^\rW_{-n}U^p\bigr)^{\Gal_{\bQ_\ell}}=0$ for all~$n>0$ and all primes~$\ell$.
\end{lemma}
\begin{proof}
	Each~$\gr^\rW_{-n}U^p$ is a quotient of $\gr^\rW_{-n}\pi_1^{\bQ_p}(V_\Qbar;b)$, hence a direct summand by Lemma~\ref{lem:semisimplicity}, and so it suffices to prove the lemma in the case that~$U^p=\pi_1^{\bQ_p}(V_\Qbar;b)$. In this case, the first point is a consequence of the fact that~$(\cV,b)$ has good reduction at all but finitely many primes~$\ell$, the second point is \cite[Theorem~1.8 \& Corollary~9.29]{olsson:non-abelian_p-adic_hodge_theory}, and the third is \cite[Theorem~1.3(1)]{betts-litt:weight-monodromy}.
\end{proof}

\begin{remark}\label{rmk:weight_filtration}
	The third condition in Lemma~\ref{lem:Up_properties} characterises the weight filtration~$\rW_\bullet U^p$ uniquely; indeed, it suffices that $\rW_\bullet U^p$ is unramified and pure of weight~$-n$ at all but finitely many primes. It follows that~$\rW_\bullet U^p$ does not depend on the choice of compactification~$\hat V$.
\end{remark}

\subsubsection{Kummer maps}

The continuous Galois cohomology set $\rH^1(\bQ,U^p(\bQ_p))$ for the action of~$\Gal_\bQ$ on~$U^p(\bQ_p)$ parametrises $\Gal_\bQ$-equivariant torsors under~$U^p$. The function
\[
j_{U^p}\colon V(\bQ) \to \rH^1(\bQ,U^p(\bQ_p))
\]
sending a point $x\in V(\bQ)$ to the class of the torsor of paths $\pi_1^{\bQ_p}(V_\Qbar;b,x)$, pushed out to~$U^p$, is called the (non-abelian) \emph{Kummer map} associated to~$U^p$. For each prime~$\ell$, there are also analogously-defined \emph{local Kummer maps}
\[
j_{U^p,\ell}\colon V(\bQ_\ell) \to \rH^1(\bQ_\ell,U^p(\bQ_p)) \,.
\]
The local and global Kummer maps are related to one another by a commuting square
\begin{equation}\label{diag:local_vs_global_kummer}
\begin{tikzcd}
	V(\bQ) \arrow[r,hook]\arrow[d,"j_{U^p}"] & V(\bQ_\ell) \arrow[d,"j_{U^p,\ell}"] \\
	\rH^1(\bQ,U^p(\bQ_p)) \arrow[r,"\loc_\ell"] & \rH^1(\bQ_\ell,U^p(\bQ_p)) \,,
\end{tikzcd}
\end{equation}
where the \emph{localisation map} $\loc_\ell$ is restriction to a decomposition group at~$\ell$. These Kummer maps satisfy two useful compatibility conditions.

\begin{lemma}[Independence of base point, {\cite[\S2.9]{balakrishnan-dan-cohen-kim-wewers:non-abelian_conjecture}}]\label{lem:independence_of_base_point}
	Suppose that~$b'$ is another $\bQ$-rational base point on~$V$, possibly tangential. Then there is a corresponding quotient~$U^{\prime p}$ of~$\pi_1^{\bQ_p}(V_\Qbar;b')$ which is the Serre twist of~$U^p$ by~$j_{U^p}(b')$, and such that the square
	\begin{center}
	\begin{tikzcd}
		V(\bQ) \arrow[r,equal]\arrow[d,"j_{U^p}"] & V(\bQ) \arrow[d,"j_{U^{\prime p}}"] \\
		\rH^1(\bQ,U^p(\bQ_p)) \arrow[r,equal,"\sim"] & \rH^1(\bQ,U^{\prime p}(\bQ_p))
	\end{tikzcd}
	\end{center}
	commutes, where the bottom isomorphism is the Serre twisting bijection \cite[Proposition~I.35 bis]{serre:galois_cohomology}. Similarly for the local Kummer maps $j_{U^p,\ell}$.
\end{lemma}

\begin{lemma}\label{lem:functoriality_of_kummer}
	Let~$f\colon (V,b)\to(V',b')$ be a morphism of pointed varieties, and let~$U^{\prime p}$ be a quotient of~$\pi_1^{\bQ_p}(V'_\Qbar;b')$ such that the homomorphism $f_*\colon\pi_1^{\bQ_p}(V_\Qbar;b)\to\pi_1^{\bQ_p}(V'_\Qbar;b')$ factors through a homomorphism $f_*\colon U^p\to U^{\prime p}$. Then the square
	\begin{center}
	\begin{tikzcd}
		V(\bQ) \arrow[r,"f"]\arrow[d,"j_{U^p}"] & V'(\bQ) \arrow[d,"j_{U^{\prime p}}"] \\
		\rH^1(\bQ,U^p(\bQ_p)) \arrow[r,"f_*"] & \rH^1(\bQ,U^{\prime p}(\bQ_p))
	\end{tikzcd}
	\end{center}
	commutes, and similarly for the local Kummer maps $j_{U^p,\ell}$.
\end{lemma}

Although we have defined the Kummer maps on $V(\bQ)$ (resp.~$V(\bQ_\ell)$), we are primarily interested in studying them on~$\bZ$-points (resp.~$\bZ_\ell$-points). Regarding these, we recall two well-known facts.

\begin{lemma}\label{lem:kummer_image_good_reduction}
	If~$\ell\neq p$ is a prime of good reduction for~$(\cV,b)$, then the image of
	\[
	j_{U^p,\ell}\colon \cV(\bZ_\ell) \to \rH^1(\bQ_\ell,U^p(\bQ_p))
	\]
	is contained inside~$\{*\}$.
\end{lemma}

\begin{lemma}\label{lem:image_in_H1f}
	The image of
	\[
	j_{U^p,p}\colon \cV(\bZ_p) \to \rH^1(\bQ_p,U^p(\bQ_p))
	\]
	is contained inside
	\[
	\rH^1_f(\bQ_p,U^p(\bQ_p)) \coloneqq \ker\bigl(\rH^1(\bQ_p,U^p(\bQ_p))\to\rH^1(\bQ_p,U^p(\sB_\cris))\bigr) \,.
	\]
\end{lemma}
\begin{proof}
	It suffices to prove the result when~$U^p=\pi_1^{\bQ_p}(V_\Qbar;b)$ is the whole fundamental group. It is a result of Olsson that the torsor of paths $\pi_1^{\bQ_p}(V_{\Qbar_p};b,x)$ has the property that its affine ring is a crystalline representation of $\Gal_{\bQ_p}$ \cite{olsson:non-abelian_p-adic_hodge_theory}\footnote{This is not quite stated explicitly in \cite{olsson:non-abelian_p-adic_hodge_theory} in this level of generality. \cite[Theorem~1.8]{olsson:non-abelian_p-adic_hodge_theory} is the statement when~$b=x$ is not tangential, and \cite[Corollary~9.29]{olsson:non-abelian_p-adic_hodge_theory} is the statement when~$b=x$ is tangential. The statement for general non-tangential $b$ and $x$ is \cite[Theorem~1.11]{olsson:non-abelian_p-adic_hodge_theory}; it seems that the same argument should work also when~$b$ is tangential, but the author has not carefully verified this. In any case, in this paper we will only need the result when~$V$ is a $\bG_m^r$-torsor over an abelian variety, in which case one can avoid using tangential base points by the theory of the principal part (\S\ref{sss:tangential_points}).}. This implies that its class lies in~$\rH^1_f(\bQ_p,U^p(\bQ_p))$, see e.g.~\cite[Definition--Lemma~4.27]{betts:motivic_local_heights}.
\end{proof}

\subsubsection{Selmer schemes}

In the non-abelian Chabauty method, one replaces the local and global cohomology sets $\rH^1(\bQ_p,U^p(\bQ_p))$ and $\rH^1(\bQ,U^p(\bQ_p))$ with certain subsets which are rather more manageable, and importantly which carry the structure of affine $\bQ_p$-schemes. We have more or less seen the local Selmer scheme already. For a $\bQ_p$-algebra~$R$, let~$\rH^1_f(\bQ_p,U^p(R))$ denote the kernel of the natural map
\[
\rH^1(\bQ_p,U^p(R)) \to \rH^1(\bQ_p,U^p(\sB_\cris\otimes_{\bQ_p}R))
\]
on continuous Galois cohomology sets for the local Galois group $\Gal_{\bQ_p}$. For the conventions on how to topologise~$U^p(R)$, see \cite[\S1]{kim:siegel}. The assignment $R\mapsto\rH^1_f(\bQ_p,U^p(R))$ is functorial in~$R$, and this functor is representable by an affine $\bQ_p$-scheme, which is of finite type if~$U^p$ is finite-dimensional \cite[Proposition~1.4]{kim:tangential_localisation}.

\begin{definition}
	The \emph{local Selmer scheme} $\rH^1_f(\bQ_p,U^p)$ is the affine $\bQ_p$-scheme representing the functor~$\rH^1_f(\bQ_p,U^p(-))$.
\end{definition}

The global Selmer scheme will similarly parametrise Galois cohomology classes for the global Galois group~$\Gal_\bQ$ satisfying certain local conditions. The exact local conditions which one imposes varies somewhat in the literature: we will impose the conditions from \cite[Definition~2.2]{balakrishnan-dogra:quadratic_chabauty_1}, generalised appropriately to higher-dimensional~$V$. Recall that~$J$ was defined to be the Albanese variety of~$(\hat V,b_0)$; let~$\alb\colon V\to J$ denote the composition of the inclusion $V\hookrightarrow\hat V$ and the Albanese map $\hat V\to J$. It follows from the definition of the weight filtration that $\gr^\rW_{-1}U^p$ is a quotient of $\gr^\rW_{-1}\pi_1^{\bQ_p}(V_\Qbar;b)=V_pJ$.

\begin{definition}\label{def:global_selmer_scheme}
	Let~$R$ be a $\bQ_p$-algebra. We define~$\Sel^\circ_{U^p}(\cV/\bZ)(R)$ to be the set of global Galois cohomology classes $[\xi]\in\rH^1(\bQ,U^p(R))$ satisfying the following two conditions:
	\begin{enumerate}[label=\roman*), ref=(\roman*)]
		\item\label{condn:local_image} $\loc_\ell([\xi])$ lies in the image of $j_{U^p,\ell}\colon\cV(\bZ_\ell)\to\rH^1(\bQ_\ell,U^p(R))$ for all~$\ell\neq p$;
		\item\label{condn:crystalline} $\loc_p([\xi])$ lies in $\rH^1_f(\bQ_p,U^p(R))$.
	\end{enumerate}
	The cohomology group~$\rH^1(\bQ,\gr^\rW_{-1}U^p(R))$ is equal to $R\otimes_{\bQ_p}\rH^1(\bQ,\gr^\rW_{-1}U^p(\bQ_p))$ (see e.g.~\cite[Proposition~2.2.4]{betts:weight_filtrations}). We define $\Sel_{U^p}(\cV/\bZ)(R)\subseteq\Sel^\circ_{U^p}(\cV/\bZ)(R)$ to be the subset of cohomology classes which additionally satisfy:	
	\begin{enumerate}[resume, label=\roman*), ref=(\roman*)]
		\item\label{condn:albanese_image} $\alb_*([\xi])$ lies in the $R$-span of the image of the Kummer map
		\[
		J(\bQ) \to \rH^1(\bQ,\gr^\rW_{-1}U^p(R)) \,.
		\]
	\end{enumerate}
	
	The functor~$\Sel_{U^p}(\cV/\bZ)(-)$ is representable by an affine $\bQ_p$-scheme $\Sel_{U^p}(\cV/\bZ)$, which is of finite type if~$U^p$ is finite-dimensional. We call~$\Sel_{U^p}(\cV/\bZ)$ the \emph{global Selmer scheme}.
\end{definition}

\begin{remark}[cf.~{\cite[Remark~2.3]{balakrishnan-dogra:quadratic_chabauty_1}}]\label{rmk:balakrishnan-dogra_modification}
	If the Tate--Shafarevich group of~$J$ is finite, then the Kummer map
	\[
	\bQ_p\otimes J(\bQ) \to \rH^1(\bQ,V_pJ)
	\]
	is an isomorphism onto the subspace of cohomology classes which are crystalline at~$p$. Since the quotient map $V_pJ\twoheadrightarrow\gr^\rW_{-1}U^p$ splits (by semisimplicity), this would imply that the third condition in Definition~\ref{def:global_selmer_scheme} is vacuous, i.e.~that the inclusion $\Sel_{U^p}(\cV/\bZ)(R)\subseteq\Sel^\circ_{U^p}(\cV/\bZ)(R)$ is an equality. The reason we work with $\Sel^\circ_{U^p}(\cV/\bZ)$ rather than~$\Sel_{U^p}(\cV/\bZ)$ is so that our results are not conditional on the Tate--Shafarevich Conjecture.
\end{remark}

We briefly recall why the functor~$\Sel_{U^p}(\cV/\bZ)(-)$ in Definition~\ref{def:global_selmer_scheme} is representable. The local Kummer map
\[
j_{U^p,\ell}\colon\cV(\bZ_\ell) \to \rH^1(\bQ_\ell,U^p(\bQ_p))
\]
is locally constant in the $\ell$-adic topology for all~$\ell\neq p$ \cite[Theorem~1.1]{betts:local_constancy} (for curves, see \cite[Corollary~0.2]{kim-tamagawa:albanese_map}), and so has finite image because $\cV(\bZ_\ell)$ is compact ($\cV$ is of finite type). Since this image is equal to~$\{*\}$ for all but finitely many~$\ell$ (Lemma~\ref{lem:kummer_image_good_reduction}), it follows that the sets $(j_{U^p,\ell}(\cV(\bZ_\ell)))_{\ell\neq p}$ form a Selmer structure on~$U^p$ in the sense of \cite[\S3.2]{betts:weight_filtrations}. Then \cite[Proposition~3.2.4]{betts:weight_filtrations} (cf.~\cite[\S2.8]{balakrishnan-dan-cohen-kim-wewers:non-abelian_conjecture}) implies that~$\Sel^\circ_{U^p}(\cV/\bZ)(-)$ is representable by a $\bQ_p$-scheme~$\Sel^\circ_{U^p}(\cV/\bZ)$. One can check that the third condition in Definition~\ref{def:global_selmer_scheme} imposes a closed condition on~$[\xi]$, whence $\Sel_{U^p}(\cV/\bZ)$ is representable by a closed subscheme of~$\Sel^\circ_{U^p}(\cV/\bZ)$. \qed

It is clear from the definitions that the images of the local and global Kummer maps are contained in~$\rH^1_f(\bQ_p,U^p(\bQ_p))$ and~$\Sel_{U^p}(\cV/\bZ)(\bQ_p)$, respectively, so we have the commuting square
\begin{equation}\label{diag:c-k_square}
\begin{tikzcd}
	\cV(\bZ) \arrow[r,hook]\arrow[d,"j_{U^p}"] & \cV(\bZ_p) \arrow[d,"j_{U^p,p}"] \\
	\Sel_{U^p}(\cV/\bZ)(\bQ_p) \arrow[r,"\loc_p"] & \rH^1_f(\bQ_p,U^p(\bQ_p)) \,,
\end{tikzcd}
\end{equation}
which we call the \emph{non-abelian Chabauty square}. (This is a generalisation of the Chabauty--Coleman square \eqref{diag:chabauty-coleman_square}.) The map~$\loc_p$ appearing in~\eqref{diag:c-k_square} is calle the \emph{localisation map}, and is given by restriction to a decomposition group at~$p$ ($\loc_p$ comes from a morphism of schemes because it is a natural transformation of functors).

\begin{definition}
	In the above setup, the \emph{non-abelian Chabauty locus} relative to the quotient~$U^p$ is the set~$\cV(\bZ_p)_{U^p}$ consisting of elements~$x\in\cV(\bZ_p)$ such that~$j_{U^p,p}(x)$ lies in the scheme-theoretic image of~$\loc_p$.
\end{definition}

Because it is not relevant here, we avoid saying anything about the locus~$\cV(\bZ_p)_{U^p}$ beyond this definition. In particular, we will not need to discuss Coleman functions or the like. The one fact we will need to know about the locus is that it decomposes as a union over simple opens in~$\cV$.

\begin{lemma}\label{lem:locus_union_over_simple_opens}
	We have
	\[
	\cV(\bZ_p)_{U^p} = \bigcup_\cU\cU(\bZ_p)_{U^p} \,,
	\]
	the union being taken over simple opens~$\cU\subseteq\cV$.
\end{lemma}
\begin{proof}
	Because~$\cV$ is regular, we have $\cV(\bZ_\ell)=\bigcup_\cU\cU(\bZ_\ell)$ for all primes~$\ell$. This implies by definition that $\Sel_{U^p}(\cV/\bZ)=\bigcup_\cU\Sel_{U^p}(\cU/\bZ)$. Since $\cV_{\bF_p}$ is geometrically connected by assumption, it follows that~$\cV(\bZ_p)=\cU(\bZ_p)$ for all simple opens~$\cU$, and the result follows.
\end{proof}

\subsection{Selmer schemes and Kummer maps for AT varieties}

We now want to study the structure of Selmer schemes when~$V=P$ is a $\bG_m^r$-torsor over an abelian variety~$A$, all defined over~$\bQ$. We fix a rational base point~$\tilde0\in P(\bQ)$ (not tangential!) lying in the fibre over~$0\in A(\bQ)$. We will frequently make use of the following fact.

\begin{lemma}\label{lem:iso_on_pi1}
	Let~$P$ be a pointed $\bG_m^r$-torsor over an abelian variety~$A$, and let~$P'$ be a pointed $\bG_m^{r'}$-torsor over an abelian variety~$A'$, and let~$f\colon P\to P'$ be an isogeny which preserves base points. Then~$f$ induces an isomorphism on pro-unipotent fundamental groups.
\end{lemma}
\begin{proof}
	Saying that~$f$ is an isogeny means that it lies over an isogeny~$f_A\colon A\to A'$ and under an isogeny~$f_G\colon\bG_m^r\to\bG_m^{r'}$. So the induced map on fundamental groups fits into a commuting diagram
	\begin{center}
	\begin{tikzcd}
		1 \arrow[r] & \bQ_p(1)^r \arrow[r]\arrow[d,"f_{G*}"] & \pi_1^{\bQ_p}(P_\Qbar;\tilde0) \arrow[r]\arrow[d,"f_*"] & V_pA_\Qbar \arrow[r]\arrow[d,"f_{A*}"] & 1 \\
		1 \arrow[r] & \bQ_p(1)^{r'} \arrow[r] & \pi_1^{\bQ_p}(P'_\Qbar;\tilde0) \arrow[r] & V_pA'_\Qbar \arrow[r] & 1
	\end{tikzcd}
	\end{center}
	with exact rows, where~$V_pA_\Qbar=\pi_1^{\bQ_p}(A_\Qbar;0)$ denotes the $\bQ_p$-linear Tate module. The outermost vertical maps are isomorphisms because~$f_G$ and~$f_A$ are isogenies, and so the central vertical map is also an isomorphism by the five-lemma.
\end{proof}

Now suppose that~$\cP$ is a simple model of~$P$, and let~$\cP_n$ be the corresponding simple models of each~$P_n$ (Lemma~\ref{lem:Pn_models}). Let us write~$U^p=\pi_1^{\bQ_p}(P_\Qbar;\tilde0)$ for the $\bQ_p$-pro-unipotent \'etale fundamental group of~$P_\Qbar$. By Lemma~\ref{lem:iso_on_pi1}, we may also identify the $\bQ_p$-pro-unipotent \'etale fundamental group of each~$P_{n,\Qbar}$ with~$U^p$, via the isomorphism induced by the maps~$\beta_n$. Under these identifications, the map on fundamental groups induced by~$\beta_{n,m}\colon P_n\to P_{mn}$ is also the identity. So by naturality of Kummer maps, we have a commuting diagram
\begin{center}
\begin{tikzcd}
	\cP(\bZ) \arrow[r,"\beta_n"]\arrow[d,"j_P"] & \cP_n(\bZ) \arrow[r,"\beta_{n,m}"]\arrow[d,"j_{P_n}"] & \cP_{mn}(\bZ) \arrow[d,"j_{P_{mn}}"] \\
	\rH^1(\bQ,U^p(\bQ_p)) \arrow[r,equal] & \rH^1(\bQ,U^p(\bQ_p)) \arrow[r,equal] & \rH^1(\bQ,U^p(\bQ_p))
\end{tikzcd}
\end{center}
for all~$n$ and~$m$, as well as a similar diagram for the local Kummer maps~$j_{P,\ell}$. This means that the global and local Kummer maps induce functions out of the colimits
\[
j_P\colon \varinjlim_n\cP_n(\bZ) \to \rH^1(\bQ,U^p(\bQ_p)) \quad\text{and}\quad j_{P,\ell}\colon \varinjlim_n\cP_n(\bZ_\ell) \to \rH^1(\bQ_\ell,U^p(\bQ_p)) \,.
\]

At primes~$\ell\neq p$, the local Kummer maps turn out to be particularly simple.

\begin{lemma}\label{lem:AT_local_kummer_l}
	For any prime~$\ell\neq p$, the image of the local Kummer map
	\[
	j_{P,\ell}\colon \cP(\bZ_\ell) \to \rH^1(\bQ_\ell,U^p(\bQ_p))
	\]
	consists of a single point.
\end{lemma}
\begin{proof}
	We have that~$\cP(\bZ_\ell)\neq\emptyset$ because~$\cP$ is simple, so we only need to prove that~$j_{P,\ell}$ is constant. We first reduce to the case that~$\cP$ is the canonical model of~$(P,\tilde0)$ (Example~\ref{ex:model_with_point}). By the compatibility of the local Kummer maps, we have that~$j_{P,\ell}=j_{P_n,\ell}\circ\beta_n$, and so we are free to replace~$\cP$ with~$\cP_n$. In particular, we are free to assume that~$\cP$ is a $\bG_m^r$-torsor over the identity component~$\cA^0$ of~$\cA$. Then~$\cP(\bZ)\subseteq P(\bQ)$ contains a point~$\tilde0'$ in the fibre over~$0\in A(\bQ)$. Changing the base point on~$P$ to~$\tilde0'$ if necessary and using Lemma~\ref{lem:independence_of_base_point}, it suffices to prove the result in the case that~$\tilde0'=\tilde0$, i.e.~$\cP$ is the canonical model of~$(P,\tilde0)$ containing the base point~$\tilde0$. This implies that each~$\cP_n$ is also the canonical model of~$(P_n,\beta_n(\tilde0))$.
	
	Now the map~$j_{P,\ell}$ is locally constant in the $\ell$-adic topology by \cite[Theorem~1.1]{betts:local_constancy}. Moreover, because it is natural in~$(P,\tilde0)$, we have a commuting square
	\begin{equation}\label{diag:local_kummer_equivariant_l}
	\begin{tikzcd}
		\bG_m^r(\bZ_\ell)\times\cP(\bZ_\ell) \arrow[r,"\rho"]\arrow[d,"j_{\bG_m^r,\ell}\times j_{P,\ell}"] & \cP(\bZ_\ell) \arrow[d,"j_{P,\ell}"] \\
		\rH^1(\bQ_\ell,\bQ_p(1)^r)\times\rH^1(\bQ_\ell,U^p(\bQ_p)) \arrow[r,"\rho_*"] & \rH^1(\bQ_\ell,U^p(\bQ_p)) \,,
	\end{tikzcd}
	\end{equation}
	where~$\rho$ is the action of~$\bG_m^r$ on~$P$. Because $j_{\bG_m^r,\ell}\colon \bG_m^r(\bZ_\ell) \to \rH^1(\bQ_\ell,\bQ_p(1)^r)$ is a locally constant group homomorphism from a compact group to a torsion-free group, it must be the zero homomorphism, so the square~\eqref{diag:local_kummer_equivariant_l} says that~$j_{P,\ell}$ is $\bG_m^r(\bZ_\ell)$-invariant. Thus, it factors through a locally constant function
	\[
	\bar j_{P,\ell}\colon \cA^0(\bZ_\ell)\to\rH^1(\bQ_\ell,U^p(\bQ_p)) \,.
	\]
	By exactly the same argument, each map~$j_{P_n,\ell}$ factors through a locally constant function $\bar j_{P_n,\ell}\colon \cA^0(\bZ_\ell)\to\rH^1(\bQ_\ell,U^p(\bQ_p))$.
	
	Consider the fibre product~$P'=P\times_A[-1]^*P$, which is a $\bG_m^{2r}$-torsor over~$A$, and let~$\cP'$ be the canonical model of~$(P',(\tilde0,\tilde0))$. For each positive integer~$n$, there is a pointed morphism $\psi_n\colon P'\to P_n$ lying over the identity on~$A$ and lying under the homomorphism $\bG_m^{2r}\to\bG_m^r$ given by $(g,h)\mapsto g^{1+n}h^{1-n}$. By Lemma~\ref{lem:producing_simple_models}, $\psi_n$ is the generic fibre of a unique morphism $\psi_n\colon\cP'\to\cP_n$. So by Lemma~\ref{lem:functoriality_of_kummer}, we have a commuting square
	\begin{center}
	\begin{tikzcd}
		\cA^0(\bZ_\ell) \arrow[d,"\bar j_{P',\ell}"]\arrow[r,equal] & \cA^0(\bZ_\ell) \arrow[d,"\bar j_{P_n,\ell}"] \\
		\rH^1(\bQ_\ell,U^{\prime p}(\bQ_p)) \arrow[r] & \rH^1(\bQ_\ell,U^p(\bQ_p)) \,,
	\end{tikzcd}
	\end{center}
	where~$U^{\prime p}$ is the $\bQ_p$-pro-unipotent \'etale fundamental group of~$P'$.
	
	Because~$\bar j_{P',\ell}$ is locally constant, there exists an open subgroup~$V\subseteq\cA^0(\bZ_\ell)$ on which~$\bar j_{P',\ell}$ is constant. The commutativity of the above diagram then implies that $\bar j_{P_n,\ell}$ is constant on~$V$ for all~$n$. In particular, if we take~$n$ to be the index of~$V$ in~$\cA^0(\bZ_\ell)$, then for any points~$x,y\in\cA^0(\bZ_\ell)$, we have
	\[
	\bar j_{P,\ell}(x) = \bar j_{P_n,\ell}(nx) = \bar j_{P_n,\ell}(ny) = \bar j_{P,\ell}(y) \,,
	\]
	because~$nx,ny\in V$. So~$\bar j_{P,\ell}$ is actually a constant function, which is what we wanted to prove.
\end{proof}

\begin{remark}[Relation with local heights]\label{rmk:local_heights}
	Any rigidified line bundle~$L$ on an abelian variety~$A$ comes with a canonical family of $\ell$-adic metrics $||\cdot||_{L,\ell}\colon L(\Qbar_\ell)\to\bR_{\geq0}$ \cite[Theorem~9.5.7]{bombieri-gubler:heights}. These metrics play the role of components of the N\'eron--Tate canonical height function $\hat h_L$ \cite[Corollary~9.5.14]{bombieri-gubler:heights}. In \cite{betts:motivic_local_heights}, the author showed that these metrics are controlled by the local Kummer maps for the $\bG_m$-torsor $P$ corresponding to~$L$: the inclusion~$\bQ_p(1)\hookrightarrow U^p=\pi_1^{\bQ_p}(P_\Qbar;b)$ induces a bijection $\rH^1(\bQ_\ell,\bQ_p(1))\xrightarrow\sim\rH^1(\bQ_\ell,U^p(\bQ_p))$ for every~$\ell\neq p$, and the composition
	\[
	P(\bQ_\ell) \xrightarrow{j_{P,\ell}} \rH^1(\bQ_\ell,U^p(\bQ_p)) \cong \rH^1(\bQ_\ell,\bQ_p(1)) = \bQ_p
	\]
	is, up to a scalar, the logarithm of $||\cdot||_{L,\ell}$.
	
	Thus, Lemma~\ref{lem:AT_local_kummer_l} implies that $||\tilde x||_{L,\ell}=||\tilde y||_{L,\ell}$ for any two points $\tilde x,\tilde y\in\cP(\bZ_\ell)$. (This can also be proved directly.) More strongly, because $\cP(\bZ_\ell)$ is a $\bG_m(\bZ_\ell)$-torsor, it follows that the inverse image of~$\cT(\bZ_\ell)$ under the projection $\pi\colon P(\bQ_\ell)\to A(\bQ_\ell)=\cA(\bZ_\ell)$ is the disjoint union $\coprod_{n\in\bZ}\ell^n\cP(\bZ_\ell)$, and so $\cP(\bZ_\ell)$ is \emph{exactly} equal to the set of elements~$\tilde x\in P(\bQ_\ell)$ such that~$\pi(\tilde x)\in\cT(\bZ_\ell)$ and $||\tilde x||_{L,\ell}=||\tilde x_0||_{L,\ell}$, for any fixed~$\tilde x_0\in\cP(\bZ_\ell)$. See \S\ref{ss:local_heights} for some more discussion on local heights.
\end{remark}

Next we examine the local Kummer map at~$p$. Suppose that~$A$ has good reduction at~$p$ and that the base point~$\tilde0$ lies in~$\cP(\bZ_p)$. Because~$\cP_{\bZ_p}$ is Zariski-locally trivial over~$\cA_{\bZ_p}$, it has a relative compactification~$\hat\cP_{\bZ_p}$ which, Zariski-locally on~$\cA_{\bZ_p}$, looks like $(\bP^1)^r\times\cA_{\bZ_p}\supseteq\bG_m^r\times\cA_{\bZ_p}$. Since~$\cA_{\bZ_p}$ is proper by assumption, it follows that~$\hat\cP_{\bZ_p}$ is a smooth proper $\bZ_p$-scheme, containing~$\cP_{\bZ_p}$ as the complement of a relative normal crossings divisor. In other words, $(\cP,\tilde0)$ has good reduction at~$p$. This implies that the image of
\[
j_{P,n}\colon \cP(\bZ_p) \to \rH^1(\bQ_p,U^p(\bQ_p))
\]
is contained in~$\rH^1_f(\bQ_p,U^p(\bQ_p))$ (Lemma~\ref{lem:image_in_H1f}). Applying the same argument to each~$\cP_n$, we find that the image of
\[
j_{P,n}\colon \varinjlim_n\cP_n(\bZ_p) \to \rH^1(\bQ_p,U^p(\bQ_p))
\]
is also contained in~$\rH^1_f(\bQ_p,U^p(\bQ_p))$.

\begin{lemma}\label{lem:AT_local_kummer_p}
	Suppose that~$A$ has good reduction at~$p$ and that the base point~$\tilde0$ lies in~$\cP(\bZ_p)$. Then the local Kummer map is a bijection
	\[
	j_{P,p}\colon \varinjlim_n\cP_n(\bZ_p) \xrightarrow\sim \rH^1_f(\bQ_p,U^p(\bQ_p)) \,.
	\]
\end{lemma}
\begin{proof}
	To prove this, we need to say something about the structure of the set~$\rH^1_f(\bQ_p,U^p(\bQ_p))$. For this, we have the Bloch--Kato logarithm \cite[Proposition~1.4]{kim:tangential_localisation}
	\[
	\rH^1_f(\bQ_p,U^p(\bQ_p)) \cong \frac{\sD_\deR(U^p)(\bQ_p)}{\rF^0\sD_\deR(U^p)(\bQ_p)} \,,
	\]
	where~$\sD_\deR(-)$ is Fontaine's Dieudonn\'e functor and~$\rF^0\sD_\deR(U^p)$ is the $0$th step of its Hodge filtration (see \cite[\S4.2]{betts:motivic_local_heights} for a discussion of how to evaluate Dieudonn\'e functors on unipotent groups). Since~$U^p$ is de Rham (see e.g.~\cite[Theorem~1.4]{betts:local_constancy}) and a central extension of~$V_pA$ by~$\bQ_p(1)^r$, it follows that~$\sD_\deR(U^p)$ is a central extension of~$\frac{\sD_\deR(V_pA)}{\rF^0\sD_\deR(V_pA)}$ by~$\bQ_p^r$, strictly compatible with Hodge filtrations \cite[Th\'eor\`eme~3.8(ii)]{fontaine:representations_semi-stables}. In particular, the set $\rH^1_f(\bQ_p,U^p(\bQ_p))$ carries the structure of a $\rH^1_f(\bQ_p,\bQ_p(1)^r)$-torsor over $\rH^1_f(\bQ_p,V_pA)$, functorial in~$(P,\tilde0)$. By an Eckmann--Hilton-style argument, the action of~$\rH^1_f(\bQ_p,\bQ_p(1)^r)$ on~$\rH^1_f(\bQ_p,U^p(\bQ_p))$ is the natural one induced from the fact that~$\rH^1_f(\bQ_p,-)$ is a product-preserving functor.
	
	This in particular implies that the local Kummer map $j_{P,p}\colon\cP(\bZ_p)\to\rH^1_f(\bQ_p,U^p(\bQ_p))$ is compatible with torsor structures, in the sense that we have two commuting squares
	\begin{equation}\label{diag:local_kummer_projection_p}
	\begin{tikzcd}
		\varinjlim_n\cP_n(\bZ_p) \arrow[r,"j_{P,p}"]\arrow[d,"\pi"] & \rH^1_f(\bQ_p,U^p(\bQ_p))  \arrow[d,"\pi_*"] \\
		\varinjlim_n\cA^0(\bZ_p) \arrow[r,"j_{A,p}"] & \rH^1_f(\bQ_p,V_pA) \,.
	\end{tikzcd}
	\end{equation}
	\begin{equation}\label{diag:local_kummer_equivariant_p}
	\begin{tikzcd}[column sep=large]
		\varinjlim_n\bG_m^r(\bZ_p)\times\varinjlim_n\cP_n(\bZ_p) \arrow[r,"j_{\bG_m^r,p}\times j_{P,p}"]\arrow[d,"\rho"] & \rH^1_f(\bQ_p,\bQ_p(1)^r)\times\rH^1_f(\bQ_p,U^p(\bQ_p)) \arrow[d,"\rho_*"] \\
		\varinjlim_n\cP_n(\bZ_p) \arrow[r,"j_{P,p}"] & \rH^1_f(\bQ_p,U^p(\bQ_p))
	\end{tikzcd}
	\end{equation}
	where~$\pi\colon P\to A$ is the projection and~$\rho\colon G\times P\to P$ is the action.
	
	The group~$\varinjlim_n\cA^0(\bZ_p)$ appearing in~\eqref{diag:local_kummer_projection_p} is the colimit of $\cA^0(\bZ_p)$ along the maps $[n]\colon\cA^0(\bZ_p)\to\cA^0(\bZ_p)$, and thus is equal to~$\bQ_p\otimes_\Zhat\cA^0(\bZ_p)$ (as~$\cA^0(\bZ_p)$ is virtually a finitely generated $\bZ_p$-module). The Kummer map~$j_{A,p}$ is then the one induced from the Kummer sequences
	\[
	0 \to A[p^m] \to A \to A \to 0 \,.
	\]
	By \cite[Example~3.10.1]{bloch-kato:tamagawa_numbers}, this map fits into a commuting square
	\begin{center}
	\begin{tikzcd}
		\bQ_p\otimes_\Zhat\cA^0(\bZ_p) \arrow[r,"j_{A,p}"]\arrow[d,"\log_A"] & \rH^1_f(\bQ_p,V_pA) \arrow[d,"\log_\BK"] \\
		\Lie(A_{\bQ_p}) \arrow[r,equal,"\sim"] & \frac{\sD_\deR(V_pA)}{\rF^0\sD_\deR(V_pA)} \,,
	\end{tikzcd}
	\end{center}
	where~$\log_A$ is the logarithm for the $p$-adic Lie group~$\cA^0(\bZ_p)$ and $\log_\BK$ is the Bloch--Kato logarithm. Both~$\log_A$ and~$\log_\BK$ are isomorphisms, and so we deduce that~$j_{A,p}$ is an isomorphism of groups.
	
	Similarly, the group~$\varinjlim_n\bG_m^r(\bZ_p)$ appearing in~\eqref{diag:local_kummer_equivariant_p} is equal to~$\bQ_p\otimes_\Zhat\bG_m^r(\bZ_p)$ (it is the colimit of~$\bG_m^r(\bZ_p)$ along the maps $[n^2]\colon\bG_m^r(\bZ_p)\to\bG_m^r(\bZ_p)$). Exactly the same argument then shows that~$j_{\bG_m,p}$ is also an isomorphism of groups.
	
	Finally, a diagram-chase using~\eqref{diag:local_kummer_projection_p} and~\eqref{diag:local_kummer_equivariant_p} and the torsor structures shows that $j_{P,p}\colon\varinjlim_n\cP_n(\bZ_p)\to\rH^1_f(\bQ_p,U^p(\bQ_p))$ is an isomorphism, as claimed.
\end{proof}

\begin{lemma}\label{lem:AT_global_selmer}
	Suppose that~$A$ has good reduction at~$p$, that the base point~$\tilde0$ lies in~$\cP(\bZ_p)$, and that~$A$ has Mordell--Weil rank~$0$. Then the global Selmer scheme $\Sel_{U^p}(\cP/\bZ)$ consists of a single point, and the global Kummer map is a bijection
	\[
	j_P\colon \varinjlim_n\cP_n(\bZ) \xrightarrow\sim \Sel_{U^p}(\cP/\bZ)(\bQ_p) \,.
	\]
\end{lemma}
\begin{proof}
	According to \cite[Lemma~3.2.7]{betts:weight_filtrations}, the scheme $\Sel^\circ_{U^p}(\cP/\bZ)$ is a torsor under $\rH^1_f(\bQ,\bQ_p(1)^r)$ over a closed subscheme of~$\rH^1_f(\bQ,V_pA)$, where~$\rH^1_f(\bQ,-)$ denotes the subset of global Galois cohomology which is crystalline at~$p$ and unramified away from~$p$. The Albanese variety of the smooth compactification of~$P$ is~$A$ (e.g.~by using the compactification~$\hat P\supset P$ which Zariski-locally on~$A$ looks like $(\bP^1)^r\times A\supset \bG_m^r\times A$), and so $\gr^\rW_{-1}U^p=V_pA$. This implies that $\Sel_{U^p}(\cP/\bZ)$ is a torsor under~$\rH^1_f(\bQ,\bQ_p(1)^r)$ over a closed subscheme of the image of the Kummer map $\bQ_p\otimes A(\bQ) \to \rH^1_f(\bQ,V_pA)$.
	
	Now $\rH^1_f(\bQ,\bQ_p(1)^r)=\bQ_p\otimes\bG_m^r(\bZ)=0$ and $\bQ_p\otimes A(\bQ)=0$ since we are assuming that~$A(\bQ)$ is finite. This implies that $\Sel_{U^p}(\cP/\bZ)$ is the one-point scheme over~$\bQ_p$ (it is non-empty because it is a pointed scheme). The global Kummer map is then automatically a bijection by Lemma~\ref{lem:colimit_of_Pn}.
\end{proof}

\subsection{Proof of Theorem~\ref{thm:comparison}}

We henceforth assume that we are in the setup of Theorem~\ref{thm:comparison}, i.e.~$f\colon Y\to P$ induces a surjection on pro-unipotent fundamental groups, $b$ is a fixed $\bQ$-rational base point, possibly tangential, and~$p$ is a prime of good reduction for~$(\cY,b)$. We equip~$P$ with a rational base point~$\tilde0\in P(\bQ)$, where~$\tilde0=f(b)$ if~$b$ is non-tangential, and~$\tilde0=\prin(f(b))$ if~$b$ is tangential. We are free to assume that~$\tilde0$ lies in the fibre over~$0\in A(\bQ)$. In either case, there is a $\Gal_\bQ$-invariant path from~$f(b)$ to~$\tilde0$ by Proposition~\ref{prop:galois-invariant_path}, so we have a $\Gal_\bQ$-equivariant pushforward map
\[
f_*\colon \pi_1^{\bQ_p}(Y_\Qbar;b) \twoheadrightarrow \pi_1^{\bQ_p}(P_\Qbar;\tilde0)
\]
on fundamental groups, realising~$U^p\coloneqq\pi_1^{\bQ_p}(P_\Qbar;\tilde0)$ as a quotient of~$\pi_1^{\bQ_p}(Y_\Qbar;b)$.

\begin{lemma}\label{lem:A_good_reduction}
	$A$ has good reduction at~$p$.
\end{lemma}
\begin{proof}
	Because~$f$ induces a surjection on pro-unipotent fundamental groups, so too does the composition~$Y\xrightarrow{f} P \twoheadrightarrow A$. For any prime~$\ell\neq p$, the Galois action on the fundamental group~$\pi_1^{\bQ_\ell}(Y_{\Qbar_p};b)$ is unramified because~$(\cY,b)$ has good reduction at~$p$. So the action on the quotient~$\pi_1^{\bQ_\ell}(A_{\Qbar_p};b)=\rH^1_\et(A_{\Qbar_p},\bQ_\ell)^*$ is also unramified, and so~$A$ has good reduction by the N\'eron--Ogg--Shafarevich criterion.
\end{proof}

Let us now fix a simple open~$\cU\subseteq\cY_\sm$ and consider the simple model~$\cP$ for which~$f\colon Y\to P$ is the generic fibre of a morphism $f\colon \cU\to\cP$ (Lemma~\ref{lem:producing_simple_models}).

\begin{lemma}\label{lem:base_point_integral}
	The base point~$\tilde0\in P(\bQ)$ lies in~$\cP(\bZ_p)$.
\end{lemma}
\begin{proof}
	When~$b$ is non-tangential, this is straightforward: good reduction implies that~$\cY_{\bF_p}$ is smooth and geometrically connected, so~$\cU_{\bZ_p}=\cY_{\bZ_p}$ and~$b\in\cU(\bZ_p)$. Thus~$\tilde0=f(b)\in\cP(\bZ_p)$.
	
	When~$b=\vec b$ is tangential, we have $f(\vec b)\in\cP(\bZ_p\llparen t\rrparen)$, and may also write~$f(\vec b)=g\cdot\vec b_0$ for some~$g\in\bG_m^r(\bQ\llparen t\rrparen)_1$ and~$\vec b_0\in P(\bQ\llbrack t\rrbrack)$ by Lemma~\ref{lem:shift_by_G1}. Because~$\bG_m^r(\bQ\llparen t\rrparen)_1=(t^\bZ)^r$ as a subgroup of~$\bQ\llparen t\rrparen^\times$, it follows that~$g$ lies in~$\bG_m^r(\bZ_p\llparen t\rrparen)$. So $\vec b_0=g^{-1}\cdot f(\vec b)$ lies in both~$\cP(\bZ_p\llparen t\rrparen)$ and $P(\bQ\llbrack t\rrbrack)$, and so $\vec b_0\in\cP(\bZ_p\llbrack t\rrbrack)$. It follows that $\prin(f(\vec b))$, i.e.~the specialisation of~$\vec b_0$ at~$t=0$, lies in~$\cP(\bZ_p)$.
\end{proof}

We can now compare the Selmer schemes for~$\cU$ and~$\cP$ relative to the quotient~$U^p$. It is clear that the local Selmer schemes agree: they are both $\rH^1_f(\bQ_p,U^p)$. For the global Selmer schemes, we have
\begin{lemma}\label{lem:same_global_selmer}
	The induced map $f_*\colon\Sel_{U^p}(\cU/\bZ)\to\Sel_{U^p}(\cP/\bZ)$ is an isomorphism of $\bQ_p$-schemes.
\end{lemma}
\begin{proof}
	Both~$\Sel_{U^p}(\cU/\bZ)$ and~$\Sel_{U^p}(\cP/\bZ)$ are subfunctors of the cohomology functor~$\rH^1(\bQ,U^p)$, so we just need to check that the local conditions in Definition~\ref{def:global_selmer_scheme} for~$\cY$ and~$\cP$ agree. For condition~\ref{condn:crystalline} there is nothing to prove. For condition~\ref{condn:local_image}, Lemma~\ref{lem:functoriality_of_kummer} implies that
	\[
	j_{U^p,\ell}(\cU(\bZ_\ell))\subseteq j_{U^p,\ell}(\cP(\bZ_\ell))
	\]
	for all primes~$\ell\neq p$. The left-hand side is non-empty because~$\cU_{\bF_\ell}(\bF_\ell)\neq\emptyset$, and the right-hand side is a singleton by Lemma~\ref{lem:AT_local_kummer_l}. So this containment is actually an equality, and thus the condition~\ref{condn:local_image} is the same for~$\cY$ and~$\cP$.
	
	For condition~\ref{condn:albanese_image}, the Albanese variety of the smooth compactification~$X$ of~$Y$ is its Jacobian~$J=\Jac(X)$, while the Albanese variety of the smooth compactification of~$P$ is the abelian variety~$A$ (as we saw in the proof of Lemma~\ref{lem:AT_global_selmer}). By the universal property of~$J$, the corresponding Albanese maps fit into a commuting square
	\begin{center}
	\begin{tikzcd}
		Y \arrow[r,"f"]\arrow[d,"\alb_Y"'] & P \arrow[d,"\alb_P"] \\
		J \arrow[r,"g"] & A \,.
	\end{tikzcd}
	\end{center}
	Since~$f\colon Y\to P$ induces a surjection on pro-unipotent fundamental groups, the same must be true for the homomorphism of abelian varieties~$g\colon J\to A$, which is only possible if~$g$ is surjective. In particular, $g\colon J(\bQ)\to A(\bQ)$ has finite cokernel. Also, $\gr^\rW_{-1}U^p$ is identified with the $\bQ_p$-linear Tate module~$V_pA$.
	
	From the commuting square
	\begin{center}
	\begin{tikzcd}
		\bQ_p\otimes J(\bQ) \arrow[r,"g"]\arrow[d,"j_J"] & \bQ_p\otimes A(\bQ) \arrow[d,"j_A"] \\
		\rH^1(\bQ,V_pJ) \arrow[r,"g_*"] & \rH^1(\bQ,V_pA) \,,
	\end{tikzcd}
	\end{center}
	we see that~$\bQ_p\otimes J(\bQ)$ and~$\bQ_p\otimes A(\bQ)$ have the same image in~$\rH^1(\bQ,V_pA)=\rH^1(\bQ,\gr^\rW_{-1}U^p)$. This means that the condition~\ref{condn:albanese_image} is the same for~$\cY$ and~$\cP$, and we are done.
\end{proof}

Combining Lemmas~\ref{lem:same_global_selmer}, \ref{lem:AT_global_selmer}, \ref{lem:AT_local_kummer_p}, \ref{lem:functoriality_of_kummer} and the diagram~\eqref{diag:local_vs_global_kummer}, we find that in the setup of Theorem~\ref{thm:comparison}, the non-abelian Chabauty square~\eqref{diag:c-k_square} for~$\cU$ can be identified with the square
\begin{center}
\begin{tikzcd}
	\cU(\bZ) \arrow[r,hook]\arrow[d,"f_\infty"] & \cU(\bZ_p) \arrow[d,"f_\infty"] \\
	\varinjlim_n\cP_n(\bZ) \arrow[r] & \varinjlim_n\cP_n(\bZ_p) \,.
\end{tikzcd}
\end{center}
Since the global Selmer scheme~$\Sel_{U^p}(\cU/\bZ)$ is a single $\bQ_p$-valued point, it follows that the $\bQ_p$-points of the scheme-theoretic image of the localisation map $\loc_p\colon\Sel_{U^p}(\cU/\bZ)\to\rH^1_f(\bQ_p,U^p)$ is equal to the image of $\loc_p$ on~$\bQ_p$-points, i.e.~to the image of $\varinjlim_n\cP_n(\bZ)\to\varinjlim_n\cP_n(\bZ_p)$. So by the definition of the non-abelian Chabauty locus~$\cU(\bZ_p)_{U^p}$, we have
\[
\cU(\bZ_p)_{U^p} = \{x\in\cU(\bZ_p) \::\: f_\infty(x)\in\varinjlim_n\cP_n(\bZ)\} = \cU(\bZ_p)_f \,.
\]
Taking the union over all simple opens~$\cU$ and using Lemma~\ref{lem:locus_union_over_simple_opens}, we have proven the main assertion in Theorem~\ref{thm:comparison}. The final assertion, that~$c(U^p)=\dim(P)$, follows by noting that the local Selmer scheme has dimension~$\dim(P)=\dim(A)+r$ (because it is a torsor under~$\rH^1_f(\bQ_p,\bQ_p(1)^r)$ over~$\rH^1_f(\bQ_p,V_pA)$, as we saw in the proof of Lemma~\ref{lem:AT_local_kummer_p}), while the global Selmer scheme has dimension~$0$ because it is a finite set. So the codimension~$c(U^p)$ of the image of the localisation map~$\loc_p$ is necessarily equal to~$\dim(P)$. \qed

\section{Unlikely intersections}\label{s:unlikely_intersections}

From the perspective of Stoll's Conjecture, the advantage of having a geometric description of some non-abelian Chabauty locus is that it allows results from the theory of unlikely intersections to be brought to bear. As an illustration of this idea, consider the special case of Chabauty--Coleman for a morphism $f\colon X\to A$ where~$X$ is a smooth projective curve and~$A$ is an abelian variety of Mordell--Weil rank~$0$. In this case, the Chabauty--Coleman square~\eqref{diag:chabauty-coleman_square} becomes
\begin{center}
\begin{tikzcd}
	X(\bQ_p) \arrow[r,hook]\arrow[d] & X(\bQ_p) \arrow[d] \\
	0 \arrow[r] & \Lie(A_{\bQ_p}) \,,
\end{tikzcd}
\end{center}
so the Chabauty--Coleman locus is just the kernel of the map~$X(\bQ_p)\to\Lie(A_{\bQ_p})$, i.e.~$f^{-1}(A(\bQ_p)_\tors)$. If~$f$ induces a surjection on fundamental groups and $\dim(A)\geq2$, then $\im(f)$ is not contained in any translate of a strict subgroup of~$A$, and so the intersection of~$\im(f)$ and~$A_\tors$ is a finite subscheme of~$A$ by the Manin--Mumford Conjecture \cite{raynaud:manin-mumford_curves}. In this case,
\[
Z\coloneqq f^{-1}(A_\tors)
\]
is a finite subscheme of~$X$ with the property that the Chabauty--Coleman locus of~$X$ relative to~$f$ is $Z(\bQ_p)$ for all primes~$p$. So we have proved the easy
\begin{theorem}
	Stoll's Conjecture holds when~$U^p$ is the quotient realised by a morphism $f\colon X\to A$ when~$A$ is an abelian variety of dimension~$\geq2$ and Mordell--Weil rank~$0$.
\end{theorem}

We now want to do the same for the quadratic Chabauty, and prove
\begin{theorem}\label{thm:main_isogeny}
	Let~$Y/\bQ$ be a smooth curve with a regular model~$\cY/\bZ$, and let~$f\colon Y\to P$ be a morphism from~$Y$ to a $\bG_m^r$-torsor over an abelian variety~$A$ which induces a surjection on pro-unipotent fundamental groups.
	
	Suppose that~$A$ has Mordell--Weil rank~$0$ and~$\dim(P)\geq2$. Then there is a finite subscheme~$Z\subset Y$ with the property that
	\[
	\cY(\bZ_p)_f = \cZ(\bZ_p)
	\]
	for all primes~$p$, where~$\cZ\subset\cY$ is the closure of~$Z$.
\end{theorem}

Combined with Theorem~\ref{thm:comparison}, this implies our main Theorem~\ref{thm:main}. In fact, Theorem~\ref{thm:main_isogeny} is strictly more general: it has content even when the prime~$p$ has bad reduction or~$Y$ does not have a rational base point.

\subsection{Description of~$Z$}\label{ss:Z}

The subscheme~$Z$ appearing in Theorem~\ref{thm:main_isogeny} can be described explicitly, in a manner similar to the definition of the isogeny geometric quadratic Chabauty locus. For each simple open~$\cU\subseteq\cY$, we have a corresponding sequence of models~$(\cP_n)_{n\geq1}$ of the varieties~$(P_n)_{n\geq1}$. Because filtered colimits preserve injections, we have containments
\[
\varinjlim_n\cP_n(\bZ) \subseteq \varinjlim_nP_n(\bQ) \subseteq \varinjlim_nP_n(\Qbar) \,,
\]
where the leftmost set is a singleton by Lemma~\ref{lem:colimit_of_Pn} (using that~$A$ has Mordell--Weil rank~$0$). Recall that~$W\subseteq\varinjlim_nP_n(\bQ)$ was defined to be the union of the singleton subsets $\varinjlim_n\cP_n(\bZ)$ as~$\cU$ ranges over simple opens in~$\cY$. Similarly to before, we write~$f_\infty$ for the composition
\[
Y(\Qbar) \xrightarrow{f} P(\Qbar) \xrightarrow{\beta_\infty} \varinjlim_nP_n(\Qbar) \,,
\]
and set~$Z(\Qbar)\coloneqq f_\infty^{-1}(W)$ to be the inverse image of~$W$ under~$f_\infty$. Since~$f_\infty$ is equivariant for the action of the absolute Galois group of~$\bQ$, it follows that~$Z(\Qbar)$ is setwise invariant under the Galois action. We do not yet know $Z(\Qbar)$ to be a finite set (this is the main result we will prove in this section!), but once we know finiteness, it follows that~$Z(\Qbar)$ is the $\Qbar$-points of a finite subscheme~$Z\subset Y$. It is for this subscheme~$Z$ that we will prove Theorem~\ref{thm:main_isogeny}.

Postponing for the time being the question of finiteness, let us complete the rest of the proof of Theorem~\ref{thm:main_isogeny} for this~$Z$. Throughout the proof, we will fix an embedding $\Qbar\hookrightarrow\Qbar_p$, and make use of the commuting diagram
\begin{equation}
\begin{tikzcd}
	Y(\bQ_p) \arrow[r,hook]\arrow[d,"f_\infty"] & Y(\Qbar_p) \arrow[d,"f_\infty"] & Y(\Qbar) \arrow[l,hook']\arrow[d,"f_\infty"] \\
	\varinjlim_nP_n(\bQ_p) \arrow[r,hook] & \varinjlim_nP_n(\Qbar_p) & \varinjlim_nP_n(\Qbar) \arrow[l,hook'] \,,
\end{tikzcd}
\end{equation}
in which the horizontal arrows are all injective. In particular, for a point~$x$ of~$Y$ defined over~$\bQ_p^\alg\coloneqq\bQ_p\cap\Qbar$, its image under~$f_\infty\colon Y(\bQ_p)\to\varinjlim_nP_n(\bQ_p)$ lies in the set~$W$ if and only if its image under~$f_\infty\colon Y(\Qbar)\to\varinjlim_nP_n(\Qbar)$ lies in~$W$.

Because we are assuming that~$Z(\Qbar)$ is a finite set, it follows that~$\cZ(\bZ_p)=\cY(\bZ_p)\cap Z(\Qbar)$, i.e.~$\cZ(\bZ_p)$ is the set of points~$x\in\cY(\bZ_p)$ which are algebraic over~$\bQ$ and satisfy~$f_\infty(x)\in W$. So we have the containment~$\cZ(\bZ_p)\subseteq\cY(\bZ_p)_f$. For the converse containment, if~$x\in\cY(\bZ_p)_f$, then there is some~$n$ such that~$f_n(x)\in\cP_n(\bZ)$, where~$(\cP_n)_{n\geq1}$ is the sequence of models of~$(P_n)_{n\geq1}$ associated to a simple open~$\cU\subseteq\cY$. Because~$f\colon Y\to P$ is non-constant and~$\beta_n\colon P\to P_n$ is finite by Lemma~\ref{lem:isogenies_finite}, we have that~$f_n^{-1}(f_n(x))$ is a finite subscheme of~$Y$ defined over~$\bQ$, and in particular, all of its points are algebraic over~$\bQ$. So~$x\in\cZ(\bZ_p)$ as desired. \qed

\subsection{Higher torsion packets}\label{ss:higher_torsion}

The description of~$Z(\Qbar)$ above gives us a clear idea of its structure: when~$f$ is a locally closed immersion, then~$Z(\Qbar)$ is the intersection of~$Y_\Qbar$ with a finite number of fibres of the map $P(\Qbar)\to\varinjlim_nP_n(\Qbar)$, or in general~$Z(\Qbar)$ is the inverse image of a finite number of fibres under~$f$. Since sets of this kind will play a key role in proving the finiteness of~$Z(\Qbar)$, we will give them a name.

\begin{definition}
	Let~$P$ be an AT variety over a field~$k$ (usually algebraically or separably closed). A \emph{quadratic torsion coset} on~$P$ is a fibre of the map
	\[
	P(k) \to \varinjlim_nP_n(k) \,.
	\]
	If~$P$ comes with a chosen base point $\tilde0\in P(k)$, then we define $P_\tors(k)$ to be the quadratic torsion coset containing~$\tilde0$, and call $P_\tors(k)$ the set of \emph{quadratic torsion points} on~$P$.
	
	If~$Y$ is a variety over~$k$ with a morphism $f\colon Y\to P$, then a \emph{quadratic torsion packet} on~$Y$ relative to~$f$ is the inverse image of a quadratic torsion coset under~$f$.
\end{definition}

In this terminology, the set~$Z(\Qbar)$ is a finite union of quadratic torsion packets in~$Y_\Qbar$, and its finiteness is going to follow from a general finiteness theorem for quadratic torsion packets (Theorem~\ref{thm:manin-mumford} from the introduction). We note one lemma on the functoriality properties of quadratic torsion.

\begin{lemma}\label{lem:morphisms_preserve_quadratic_torsion}
	Let~$P$ and~$P'$ be AT varieties over a field~$k$ equipped with base points $\tilde0$ and~$\tilde0'$, respectively. Let~$\psi\colon P\to P'$ be a morphism such that~$\psi(\tilde0)\in P'_\tors$.
	
	Then~$\psi(P_\tors(k))\subseteq P'_\tors(k)$. If moreover~$\psi$ lies over and under homomorphisms $A\to A'$ and~$G\to G'$ with finite kernel, then additionally~$\psi^{-1}(P'_\tors(k))=P_\tors(k)$.
\end{lemma}
\begin{proof}
	The first assertion follows from the commuting square (Lemma~\ref{lem:P_n_functorial})
	\begin{center}
	\begin{tikzcd}
		P(k) \arrow[r,"\psi"]\arrow[d,"\beta_\infty"] & P'(k) \arrow[d,"\beta_\infty"] \\
		\varinjlim_nP_n(k) \arrow[r,"\psi_\infty"] & \varinjlim_nP'_n(k) \,.
	\end{tikzcd}
	\end{center}
	When~$\psi$ lies over and under homomorphisms with finite kernel, then the bottom arrow~$\psi_\infty$ in the square is a morphism of torsors lying over and under the induced injective homomorphisms $\bQ\otimes A(k)\hookrightarrow\bQ\otimes A'(k)$ and $\bQ\otimes G(k)\hookrightarrow \bQ\otimes G'(k)$. It follows that~$\psi_\infty$ is also injective, whence~$\psi^{-1}(P'_\tors(k))=P_\tors(k)$ as desired.
\end{proof}

\begin{lemma}
	Let~$B$ be a semiabelian variety, viewed as a pointed AT variety using its identity element as~$\tilde0$. Then the quadratic torsion subset~$B_\tors(k)$ is equal to the torsion subgroup (in the sense of commutative group schemes).
\end{lemma}
\begin{proof}
	The result is easy if~$B$ is an abelian variety (resp.~torus), for in this case all of the AT varieties~$B_n$ are equal to~$B$, and the morphisms $\beta_n\colon B\to B_n$ and $\beta_{n,m}\colon B_n\to B_{mn}$ are equal to multiplication by $n$ and $m$ (resp.~$2n^2$ and~$m^2$). So $\varinjlim_nB_n(k)=\bQ\otimes B(k)$, and the function $\beta_\infty\colon B(k)\to\varinjlim_nB_n(k)=\bQ\otimes B(k)$ is the usual map $x\mapsto 1\otimes x$. So it is clear that quadratic torsion agrees with usual torsion in this case.
	
	For the general case, suppose that~$B$ is an extension of an abelian variety~$A$ by a torus~$G$. If~$x\in B(k)$ is a torsion point, then choose some~$n>0$ such that $nx=0$. By Lemma~\ref{lem:morphisms_preserve_quadratic_torsion} applied to the morphism $[n]\colon B\to B$, we have $x\in[n]^{-1}(\tilde0)\subseteq B_\tors(k)$. So~$x$ is a quadratic torsion point. Conversely, if~$x\in B(k)$ is a quadratic torsion point, then its image in~$A(k)$ is a torsion point by Lemma~\ref{lem:morphisms_preserve_quadratic_torsion}. Choose some~$n>0$ such that~$nx\in G(k)$. By Lemma~\ref{lem:morphisms_preserve_quadratic_torsion} again, $nx\in B_\tors(k)\cap G(k)=G_\tors(k)$ is a torsion point on~$G$. So~$x$ is a torsion point.
\end{proof}

\subsection{AT subvarieties}

For the remainder of this section, we switch to working over the complex numbers~$\bC$ (or sometimes~$\Qbar$), and for notational convenience permit ourselves to conflate varieties over $\bC$ or~$\Qbar$ with their sets of closed points. We are going to study quadratic torsion packets by placing them in the framework of unlikely intersections. To do this, we need to identify a suitable class of ``weakly special'' subvarieties of an AT variety~$P$. This class will be the \emph{AT subvarieties} of~$P$: locally closed subvarieties which are themselves AT varieties. The most obvious kinds of AT subvarieties are torsors under subtori of~$G$ over translates of abelian subvarieties of~$A$; for a slightly more intricate example, we have

\begin{example}\label{ex:AT_subvariety}
	Let~$E$ be an elliptic curve, and let~$Q_1,Q_2\in E$ be points whose difference~$Q_1-Q_2$ is of order~$2$. Let~$P$ be the $\bG_m$-torsor over~$E$ corresponding to the divisor~$D=Q_1-Q_2$. Because~$D$ is two-torsion, it follows that~$P^{\otimes2}$ is the trivial $\bG_m$-torsor. Let~$E'\subseteq P$ be the kernel of the composition
	\[
	P \to P^{\otimes2} = \bG_m\times E \to \bG_m \,.
	\]
	It is a closed subvariety of~$P$, and is a $\{\pm1\}$-torsor over~$E$. By construction, the projection $P\to E$ is not split, and so~$E'$ must be a non-trivial torsor, hence connected, and so is itself an elliptic curve. In particular, $E'$ is an AT subvariety of~$P$ which is not a torsor under a subtorus of~$\bG_m$.
\end{example}

We now formulate and prove a few basic results regarding AT subvarieties.

\begin{lemma}\label{lem:AT_image_of_subvariety}
	Let~$\phi\colon P\to P'$ be a morphism of AT varieties over~$\bC$. Then the image of~$\phi$ is an AT subvariety of~$P'$.
\end{lemma}
\begin{proof}
	Let us first consider the case that~$P=A$ is an abelian variety. We may suppose that~$\phi$ lies over a homomorphism $\phi_A\colon A\to A'$ of abelian varieties. Consider the morphism $\xi\colon A\times\ker(\phi_A)\to G'$ given by
	\[
	\xi(x,y)\coloneqq \frac{\phi(x+y)}{\phi(x)} \,.
	\]
	(The right-hand side means the unique element~$g\in G'$ such that~$g\cdot\phi(x)=\phi(x+y)$.) Because each component of~$A\times\ker(\phi_A)$ is an abelian variety, it follows that $\xi(x,y)$ is independent of the first variable~$x$ and factors through a morphism $\bar\xi\colon \ker(\phi_A)\to G'$. By inspection, $\bar\xi$ is a homomorphism, so $\ker(\bar\xi)$ is an algebraic subgroup of~$A$. The morphism~$\phi$ is constant on cosets of~$\ker(\bar\xi)$ so, writing~$\bar A$ for the abelian variety $A/\ker(\bar\xi)$, we have that $\phi$ factors through a morphism $\bar\phi\colon \bar A\to P'$. We claim that~$\bar\phi$ is a closed immersion. To see this, note that $\im(\bar\xi)$ is a finite subgroup of~$G'$; let~$G''\coloneqq G'/\im(\bar\xi)$ be the quotient, and let~$P''$ be the pushout of~$P'$ along the quotient map~$G'\twoheadrightarrow G''$. By construction, the composition
	\[
	A \xrightarrow\phi P' \twoheadrightarrow P''
	\]
	factors through a morphism $\bar\phi'\colon A'\to P''$ which is a section of the projection $P''\twoheadrightarrow A'$. Thus the morphism~$\bar\phi'$ is a closed immersion; considering the commuting square
	\begin{center}
	\begin{tikzcd}
		\bar A \arrow[r,"\bar\phi"]\arrow[d,two heads] & P' \arrow[d,two heads] \\
		A' \arrow[r,"\bar\phi'"] & P''
	\end{tikzcd}
	\end{center}
	in which both vertical arrows are torsors under the finite group~$\im(\bar\xi)$, we deduce that~$\bar\phi$ is a closed immersion too.
	
	Now to handle the general case when~$P$ is not necessarily an abelian variety, let $\psi\colon G'\twoheadrightarrow\coker(\phi_G)$ be the canonical map to the cokernel, and let~$P''\coloneqq \psi_*P'$ be the pushout of~$P'$ along~$\psi$. By construction, the composition
	\[
	P \xrightarrow\phi P' \to P''
	\]
	is a morphism under the zero homomorphism $G\to\coker(\phi_G)$, and so factors through a morphism $\bar\psi\colon A\to P''$. By the previous part, the image of~$\bar\psi$ is an abelian variety~$B\subseteq P''$, and it follows from the construction that $\im(\phi)$ is the preimage of~$B$ in~$P'$. But~$P'\to P''$ is a torsor under~$\im(\phi_G)$, and so we find that~$\im(\phi)$ is a torsor under~$\im(\phi_G)$ over~$B$. So it is an AT subvariety.
\end{proof}

\begin{remark}
	It follows from the proof of Lemma~\ref{lem:AT_image_of_subvariety} that the image of any morphism $\phi\colon P\to P'$ of AT varieties is a closed subvariety of~$P'$. In particular, any AT subvariety of~$P$ is automatically closed.
\end{remark}

\begin{lemma}\label{lem:surjection_of_AT_varieties}
	Let~$\phi\colon P\to P'$ be a morphism of AT varieties, lying over a homomorphism $\phi_G\colon \bG_m^r\to \bG_m^{r'}$ and under a homomorphism $\phi_A\colon A\to A'$. Then $\phi$ is surjective if and only if both $\phi_A$ and $\phi_G$ are surjective.
\end{lemma}
\begin{proof}
	One direction is clear: if~$\phi_G$ and~$\phi_A$ are surjective, then~$\phi$ is surjective. For the converse, suppose that~$\phi$ is surjective; it is clear that this implies that~$\phi_A$ is surjective. Let~$\phi_G'\colon\bG_m^{r'}\twoheadrightarrow\coker(\phi_G)$ be the cokernel of~$\phi_G$, and consider the pushout $\phi'_{G*}P'$ (which is a $\coker(\phi_G)$-torsor over~$A'$). The map $\phi'\colon P' \to \phi'_{G*}P'$ is surjective (because~$\phi'_G$ is), and hence so is the composition
	\[
	P \xrightarrow{\phi} P' \to \phi'_{G*}P' \,.
	\]
	This composition is, by construction, invariant for the action of~$\bG_m^r$ on~$P$, and so factors as
	\[
	P \to A \xrightarrow{\phi''} \phi'_{G*}P'
	\]
	for some morphism~$\phi''$, which is also automatically surjective. Identifying the kernel of the projection $\phi'_{G*}P'\to A'$ with~$\coker(\phi_G)$, we find that~$\phi''$ restricts to a surjective morphism
	\[
	\ker(\phi_A) \twoheadrightarrow \coker(\phi_G) \,.
	\]
	But $\ker(\phi_A)$ is an extension of a finite abelian group by an abelian variety, so the only way that this morphism can be surjective is if~$\coker(\phi_G)=1$ is trivial. Thus, we have proved that~$\phi_G$ is surjective as claimed.
\end{proof}

\begin{remark}
	The analogous statement with injections is untrue: we saw in Example~\ref{ex:AT_subvariety} an example of an AT subvariety~$E'$ of a $\bG_m$-torsor over an elliptic curve such that the inclusion $E'\hookrightarrow P$ lies over a two-to-one map~$E'\to E$.
\end{remark}

\begin{lemma}\label{lem:surjection_on_pi1_and_AT_subvarieties}
	A morphism $f\colon Y\to P$ induces a surjection on pro-unipotent fundamental groups if and only if the image of~$f$ is not contained in any strict AT subvariety of~$P$.
\end{lemma}
\begin{proof}
	Let~$\qalb\colon Y\to Q$ be the quadratic Albanese map for~$Y$, so that~$f$ factors uniquely as $\phi\circ\qalb$ for some morphism $\phi\colon Q\to P$. The image of~$\phi$ is an AT subvariety of~$P$ by Lemma~\ref{lem:AT_image_of_subvariety}; it is automatically the smallest AT subvariety containing the image of~$f$. Since~$\qalb$ induces a surjection on pro-unipotent fundamental groups by Proposition~\ref{prop:surjection_on_pi1}, we know that~$f$ induces a surjection on pro-unipotent fundamental groups if and only if~$\phi$ does.
	
	Now~$\phi$ is a morphism under a homomorphism~$\phi_G\colon \bG_m^{\rho-1}\to\bG_m^r$ over a homomorphism $\phi_A\colon J\to A$, and the induced map on $\bQ$-unipotent Betti fundamental groups fits into a commuting diagram
	\begin{equation}\label{eq:map_between_pi1}
	\begin{tikzcd}
		1 \arrow[r] & \bQ(1)^{\rho-1} \arrow[r]\arrow[d,"\phi_{G*}"] & \Lie\pi_1^\bQ(Q) \arrow[r]\arrow[d,"\phi_*"] & \rH_1(J,\bQ) \arrow[r]\arrow[d,"\phi_{A*}"] & 1 \\
		1 \arrow[r] & \bQ(1)^r \arrow[r] & \Lie\pi_1^\bQ(P) \arrow[r] & \rH_1(A,\bQ) \arrow[r] & 1
	\end{tikzcd}
	\end{equation}
	with exact rows. By Hain--Zucker \cite{hain-zucker:unipotent_variations}, $\Lie\pi_1^\bQ(Q)$ and~$\Lie\pi_1^\bQ(P)$ carry a canonical mixed Hodge structure such that~$\phi_*$ is a morphism of mixed Hodge structures, and the rows of~\eqref{eq:map_between_pi1} are exact sequences of mixed Hodge structures. Because morphisms of mixed Hodge structures are automatically strict for their weight filtrations, we find that~$\phi_*$ is surjective if and only if both~$\phi_{G*}$ and~$\phi_{A*}$ are surjective. So we have the implications
	\begin{align*}
		\text{$f_*$ surjective} &\Leftrightarrow \text{$\phi_*$ surjective} \\
		&\Leftrightarrow \text{$\phi_{G*}$ and $\phi_{A*}$ both surjective} \\
		&\Leftrightarrow \text{$\phi_G$ and $\phi_A$ both surjective} \\
		&\Leftrightarrow \text{$\phi$ surjective} \,,
	\end{align*}
	using Lemma~\ref{lem:surjection_of_AT_varieties} in the last line. This completes the proof.
\end{proof}

\subsection{Manin--Mumford for curves in AT varieties}

We are now ready to complete the proof of Theorem~\ref{thm:main_isogeny}. As discussed above, the only aspect which remains to be proved is the finiteness of the set~$Z(\Qbar)$, for which it suffices to prove that all quadratic torsion packets in~$Y_\Qbar$ relative to the morphism~$f\colon Y_\Qbar\to P_\Qbar$ are finite. We are assuming that~$f$ induces a surjection on pro-unipotent fundamental groups, which is equivalent by Lemma~\ref{lem:surjection_on_pi1_and_AT_subvarieties} to assuming that the image of~$f$ is not contained in any strict AT subvariety (of~$P_\bC$). So we are reduced to proving the following unlikely intersections result.

\begin{theorem*}[= {Theorem~\ref{thm:manin-mumford}}]
		Let~$Y/\Qbar$ be a smooth hyperbolic curve, and let~$f\colon Y\to P$ be a morphism from~$Y$ to a $\bG_m^r$-torsor over an abelian variety. Suppose that $\dim(P)\geq2$ and that the image of~$f$ is not contained in any strict AT subvariety of~$P$. Then every quadratic torsion packet in~$Y$ relative to~$f$ is finite.
\end{theorem*}

We remark that the theorem applies to any hyperbolic curve~$Y/\Qbar$ embedded inside its quadratic Albanese variety~$Q$. Indeed, writing~$Y=X\smallsetminus D$ as usual, we have
\[
\dim(Q) = 
\begin{cases}
	g + \rank(\NS(J)) + \#D - 1 & \text{if~$X$ has genus~$g\geq1$,} \\
	\#D - 1 & \text{if~$X$ has genus~$0$,}
\end{cases}
\]
and the right-hand side is always at least~$2$. In particular, Theorem~\ref{thm:manin-mumford} has content even in the case that~$Y=E\smallsetminus\{\infty\}$ is a once-punctured elliptic curve, about which the usual Manin--Mumford does not say anything (because all maps from~$Y$ to a semi-abelian variety factor through~$E$).\smallskip

For the proof of Theorem~\ref{thm:manin-mumford}, let us fix a quadratic torsion coset in~$P$, whose inverse image in~$Y$ we wish to show is finite. Choose a base point~$\tilde 0\in P(\Qbar)$ lying in this quadratic torsion coset, so that~$P$ becomes a pointed AT variety and the quadratic torsion coset of interest is~$P_\tors$. Because the statement of Theorem~\ref{thm:manin-mumford} does not depend on the abelian group structure of~$A$, we are free to change the zero element of~$A$ in such a way that~$\tilde0$ lies over~$0\in A(\Qbar)$ as usual. We now prove that~$f^{-1}P_\tors$ is finite.

We begin with some easy cases which follow from usual Manin--Mumford for semiabelian varieties. First we handle the case that~$\dim(A)\geq2$. In this case, let~$\hat f\colon Y\to A$ denote the composition of~$f$ with the projection~$\pi\colon P\to A$. We observe that the inverse image under~$\pi$ of any translate of an abelian subvariety of~$A$ is an AT subvariety of~$P$, and so the image of~$\hat f$ cannot be contained in any translate of a strict subvariety of~$A$. The image of~$P_\tors$ in~$A$ is contained in~$A_\tors$ by Lemma~\ref{lem:morphisms_preserve_quadratic_torsion} and so~$f^{-1}P_\tors\subseteq\hat f^{-1}A_\tors$, and the latter is finite by Manin--Mumford for abelian varieties.

Next we handle the case that~$\dim(A)=0$. In this case, $P=G$ is a torus, and $\im(f)$ is not contained in a translate of a strict subtorus of~$G$ by assumption. So~$f^{-1}P_\tors=f^{-1}G_\tors$ is finite by Manin--Mumford for tori.

Next we handle the case that~$\dim(A)=1$ and~$r\geq2$. In this case, $P$ is the $\bG_m^r$-torsor on the elliptic curve~$E=A$ corresponding to an $r$-tuple of line bundles $L_1,\dots,L_r$. Pushing out along an automorphism $\bG_m^r\cong\bG_m^r$ if necessary, we are free to assume that~$L_1$ has degree~$0$. So if~$B\to E$ is the $\bG_m$-torsor associated to~$L_1$, then~$B$ has the structure of a semiabelian variety by the Barsotti--Weil Theorem \cite[Theorem~VII.16.6]{serre:algebraic_groups_class_fields}, being an extension of~$E$ by~$\bG_m$. There is a base point-preserving projection $\pi'\colon P\to B$ which makes $P$ into a $\bG_m^{r-1}$-torsor over~$B$. Let~$f'$ denote the composition of~$f$ with~$\pi'$. The image of~$f'$ cannot be contained in any translate of a strict semiabelian subvariety of~$B$: the only possibilities are points, fibres of the projection $B\to E$, or elliptic curves in~$B$ mapping finitely onto~$E$, and in all cases~$\im(f)$ would be contained in a strict AT subvariety of~$P$. So $f^{-1}P_\tors\subseteq f^{\prime-1}B_\tors$ is finite by Manin--Mumford for semiabelian varieties.\smallskip

The one remaining case is the most interesting: $\dim(A)=1$ and~$r=1$, so~$P$ is the $\bG_m$-torsor corresponding to a line bundle on the elliptic curve~$E=A$. Because we are working with an elliptic curve, we switch to writing~$\infty$ and~$\tilde\infty$ for~$0\in E(\Qbar)$ and~$\tilde0\in P(\Qbar)$. Let~$\pr_i\colon E\times E\to E$ denote the $i$th projection for~$i\in\{1,2\}$ and let~$m\colon E\times E\to E$ denote the addition map. Let~$P'$ denote the $\bG_m$-torsor
\[
m^*P\otimes\pr_1^*P^{\otimes-1}\otimes(-\pr_2)^*P^{\otimes-1}
\]
on~$E\times E$. By the theorem of the cube, $P'$ has the following properties:
\begin{enumerate}[label = (\alph*), ref = (\alph*)]
	\item\label{condn:P'_on_diagonal} the restriction of~$P'$ to the diagonal is~$P^{\otimes2}$;
	\item\label{condn:P'_on_horizontal_slice} the restriction of~$P'$ to~$E\times\{\infty\}$ is trivial; and
	\item\label{condn:P'_on_vertical_slices} the restriction of~$P'$ to each~$\{x\}\times E$ is the $\bG_m$-torsor corresponding to a degree~$0$ line bundle.
\end{enumerate}
By~\ref{condn:P'_on_diagonal}, there is a morphism $\psi\colon P\to P'$ of AT varieties lying over the diagonal $\Delta\colon E\to E\times E$ and under the squaring map~$[2]\colon\bG_m\to\bG_m$. Let~$f'\coloneqq\psi\circ f$, and make~$P'$ into a pointed AT variety via the base point~$\psi(\tilde\infty)$.

We will want to think of~$P'$ in two different ways: as an AT variety, and as a smooth family of surfaces over~$E$ via the composition
\[
P' \to E\times E \xrightarrow{\pr_1} E \,.
\]
By~\ref{condn:P'_on_horizontal_slice}, there is a section~$e$ of the projection $P'\to E$ lying over the zero section of the constant family~$\pr_1\colon E\times E\to E$. The section~$e$ is uniquely determined once we require that~$e(\infty)=\psi(\tilde\infty)$. By~\ref{condn:P'_on_vertical_slices}, $P'$ defines a section of the relative Picard scheme $\bPic^0(E\times E/E)$, and so has the structure of a semiabelian scheme over~$E$ with identity section~$e$ by the relative Barsotti--Weil Theorem \cite[Theorem~III.18.1]{oort:commutative_group_schemes}. The quadratic torsion on $P'$ as an AT variety is closely related to the torsion locus of the semiabelian scheme $P'\to E$.

\begin{lemma}\label{lem:P'_quadratic_torsion}
	We have
	\[
	P'_\tors = \bigcup_{x\in E_\tors}P'_{x,\tors} \,.
	\]
	(On the left-hand side, $P'_\tors$ denotes the quadratic torsion subset of the pointed AT variety~$(P',\psi(\tilde\infty))$, and on the right-hand side, $P'_{x,\tors}$ denotes the torsion subgroup of the semiabelian variety~$P'_x$.)
\end{lemma}
\begin{proof}
	The section $e\colon E\to P'$ is a morphism of pointed AT varieties, so $e(x)\in P'_\tors$ for all~$x\in E_\tors$ by Lemma~\ref{lem:morphisms_preserve_quadratic_torsion}. The inclusion $P'_x\hookrightarrow P'$ takes the identity element $e(x)$ to a quadratic torsion point on~$P'$, and so $P'_{x,\tors}\subseteq P'_\tors$ by Lemma~\ref{lem:morphisms_preserve_quadratic_torsion} again. This proves that $P'_\tors \supseteq \bigcup_{x\in E_\tors}P'_{x,\tors}$.
	
	For the converse inclusion, the projection~$P'\to E$ is also a morphism of pointed AT varieties, so every quadratic torsion point on~$P'$ lies in a fibre over a torsion point on~$E$. For each~$x\in E_\tors$, the inclusion $P'_x\hookrightarrow P$ is a morphism of AT varieties lying over and under injective homomorphisms, and so $P'_\tors\cap P'_x=P'_{x,\tors}$ by the second part of Lemma~\ref{lem:morphisms_preserve_quadratic_torsion}. So we have
	\[
	P'_\tors \subseteq \bigcup_{x\in E_\tors}(P'_\tors\cap P_x) = \bigcup_{x\in E_\tors}P'_{x,\tors}
	\]
	as desired.
\end{proof}

Now write~$\hat f$ for the composition $Y\xrightarrow{f} P\to E$ and define~$B\coloneqq \hat f^*P'$ to be the pullback of~$P'$ to~$Y$, so~$B$ is a semiabelian scheme over~$Y$, being an extension of the constant elliptic scheme $E_Y\coloneqq Y\times E$ by $\bG_{m,Y}\coloneqq Y\times \bG_m$. From the commuting diagram
\begin{equation}\label{eq:Y_mapping_to_abelian_scheme}
\begin{tikzcd}
	Y \arrow[r,"f"]\arrow[rd,"\hat f"'] & P \arrow[r,"\psi"]\arrow[d] & P' \arrow[d] \\
	 & E \arrow[r,"\Delta"]\arrow[rd,equal] & E\times E \arrow[d,"\pr_1"] \\
	 && E \,,
\end{tikzcd}
\end{equation}
the map $f'\colon Y\to P'$ induces a section~$s$ of the semiabelian scheme~$B\to Y$, lifting the section~$(1,\hat f)$ of~$Y\times E\to Y$. Lemma~\ref{lem:P'_quadratic_torsion} then implies
\begin{corollary}\label{cor:quadratic_torsion_vs_semiabelian_torsion}
	For every point~$y\in f^{-1}P_\tors$, we have that~$s(y)\in B_{y,\tors}$ is a torsion point on~$B_y$.
\end{corollary}
\begin{proof}
	By Lemma~\ref{lem:morphisms_preserve_quadratic_torsion}, if~$f(y)\in P_\tors$, then $f'(y) \in P'_\tors\cap P'_{\hat f(y)} = P'_{\hat f(y),\tors}$ and hence $s(y) \in B_{y,\tors}$.
\end{proof}

Accordingly, if~$f^{-1}P_\tors$ is infinite, then the image of the section~$s$ meets the torsion locus of the semiabelian scheme $B\to Y$ in infinitely many points. This is exactly the setup studied in \cite{bertrand-masser-pillay-zannier:relative_manin-mumford}, whose main theorem classifies the ways in which a section of a certain type of semiabelian scheme can meet the torsion locus at infinitely many points. The precise statement we need is
\begin{theorem}[{\cite[Theorem~2(ii)]{bertrand-masser-pillay-zannier:relative_manin-mumford}}]\label{thm:relative_manin-mumford}
	Suppose that~$Y/\Qbar$ is a smooth curve, that~$E/\Qbar$ is an elliptic curve with dual~$E^\vee$, and that $B$ is a semiabelian scheme over~$Y$ which is an extension of~$E_Y$ by~$\bG_{m,Y}$. Let~$s$ be a section of~$B\to Y$, write~$p\in E(Y)$ for the composition
	\[
	Y \xrightarrow{s} B \to E_Y \,,
	\]
	and write~$q\in E^\vee(Y)$ for the map classifying~$B$ under the Barsotti--Weil formula
	\[
	\Ext^1(E_Y,\bG_{m,Y}) = E^\vee(Y) \,.
	\]
	
	If~$s$ meets the torsion locus of~$B\to Y$ in infinitely many points, then either:
	\begin{itemize}
		\item $p\in E(Y)_\tors$ is torsion; or
		\item $q\in E^\vee(Y)_\tors$ is torsion; or
		\item $p$ and~$q$ are non-torsion, and $q=\alpha(p)$ modulo torsion for an antisymmetric element~$\alpha\in\bQ\otimes\Hom(E,E^\vee)$ (i.e.~$\alpha^\vee=-\alpha$).
	\end{itemize}
\end{theorem}

\begin{remark}
	The final option in Theorem~\ref{thm:relative_manin-mumford} can only occur if~$E$ has complex multiplication, and is related to the phenomenon of \emph{Ribet sections}: sections of certain semiabelian schemes which meet the torsion locus in infinitely many points, but do not lie in a strict subgroup scheme. See \cite{bertrand-edixhoven:semiabelian_families} for more discussion.
\end{remark}

Back in our setting of interest, the element~$p\in E(Y)$ is the map $\hat f\colon Y\to E$, while~$q\in E^\vee(Y)$ is equal to~$\lambda_P\circ\hat f$, where~$\lambda_P\colon E\to E^\vee=\Pic^0(E)$ is the homomorphism
\[
x\mapsto \tau_x^*P\otimes P^{-1} \,.
\]
Let us now suppose that~$f^{-1}P_\tors$ is infinite, which by Corollary~\ref{cor:quadratic_torsion_vs_semiabelian_torsion} implies that the section~$s$ meets the torsion locus of~$B\to Y$ in infinitely many points. Then, according to Theorem~\ref{thm:relative_manin-mumford}, either $\hat f$ is torsion, or $\lambda_P\circ\hat f$ is torsion, or $\lambda_P\circ\hat f=\alpha\circ\hat f$ as elements of~$\bQ\otimes E^\vee(Y)$, for some antisymmetric~$\alpha\in\bQ\otimes\Hom(E,E^\vee)$. In the latter case, we may choose a positive integer~$n$ such that $n\lambda_P\circ\hat f=n\alpha\circ\hat f$ as elements of~$E^\vee(Y)$, with~$n\alpha$ an antisymmetric element of $\Hom(E,E^\vee)$. There are then two possibilities. If~$n\lambda_P-n\alpha=0$, then we also have that
\[
0 = (n\lambda_P-n\alpha)^\vee = n\lambda_P+n\alpha^\vee
\]
because~$\lambda_P$ is symmetric, and so~$2n\lambda_P=0$. This implies that~$\lambda_P=0$ itself. Otherwise, if~$n\lambda_P-n\alpha\neq0$, then we know that the image of~$\hat f$ is contained in the kernel of the non-zero homomorphism $n\lambda_P-n\alpha\in\Hom(E,E^\vee)$. This kernel is finite, and so~$\hat f\in E(Y)_\tors$ again.

In every single case, we have that~$\lambda_P\circ\hat f\in E^\vee(Y)_\tors$ is torsion. If~$\hat f$ itself is torsion, then it is constant, meaning that the image of the morphism $f\colon Y\to P$ is contained in a single fibre of the projection~$P\to E$. This fibre is in particular a strict AT subvariety of~$P$. Otherwise, $\hat f$ is a non-constant map from~$Y$ to~$E$, so its image is dense in~$E$ and the only way that~$\lambda_P\circ\hat f$ can be torsion is if~$\lambda_P=0$. This implies that~$P\in\Pic^0(E)$, and so~$P$ is actually a semiabelian variety. So by Manin--Mumford for semiabelian varieties, the image of~$f$ must be contained in a strict semiabelian subvariety of~$P$, which is automatically an AT subvariety.

In summary, we have shown that if~$f^{-1}P_\tors$ is infinite, then the image of~$f$ is contained in a strict AT subvariety of~$P$, completing the proof of Theorem~\ref{thm:manin-mumford}.\qed

\section{Once-punctured elliptic curves}\label{s:elliptic_curves}

Finally, we make all of the theory we have developed completely explicit in the case of once-punctured elliptic curves of rank~$0$. To begin with, let~$E/\bQ$ be an elliptic curve, and let~$Y=E\smallsetminus\{\infty\}$ be the complement of the origin in~$E$. Let~$P$ denote the~$\bG_m$-torsor corresponding to the line bundle~$\cO(\infty)$ on~$E$, and let~$f\colon Y\to P$ be the morphism coming from the nowhere vanishing section $1\in\rH^0(Y,\cO(\infty))$.

We can describe~$P$ explicitly as the gluing of two affine surfaces. For this, fix a rational Weierstrass equation
\begin{equation}\label{eq:weierstrass}
y^2 + a_1xy + a_3y = x^3 + a_2x^2 + a_4x + a_6
\end{equation}
for~$E$, not necessarily minimal or integral. Thus~$Y$ is the affine plane curve described by the equation~\eqref{eq:weierstrass}. If we consider the rational functions
\[
t \coloneqq -\frac xy \quad\text{and}\quad u \coloneqq -\frac 1y \,,
\]
then a neighbourhood of the point at infinity in~$E$ is the plane curve
\begin{equation}\label{eq:weierstrass_at_infinity}
	u = t^3 + a_1tu + a_2t^2u + a_3u^2 + a_4tu^2 + a_6u^3 \,,
\end{equation}
see \cite[Chapter~IV.1]{silverman:elliptic_curves}. In this second chart, the point at infinity has $(t,u)$-coordinates $(0,0)$, and~$t$ is a local parameter at~$\infty$. We let~$U$ denote the locus in the second chart~\eqref{eq:weierstrass_at_infinity} where~$t^3/u$ is non-zero, so that~$E$ is the gluing of~$Y$ and~$U$ along their intersection.

To find a corresponding gluing for~$P$, note that~$\cO(\infty)$ trivialises over both~$Y$ and~$U$, via the rational functions~$1$ and~$t^{-1}$. Thus, $P$ is the gluing of the trivial torsors $Y\times\bG_m$ and $U\times\bG_m$ along their intersection $(Y\cap U)\times\bG_m$. If~$x,y,z$ and~$t,u,w$ denote the coordinates on~$Y\times\bG_m$ and~$U\times\bG_m$, respectively, then this gluing is via the identifications
\[
t = -\frac xy \quad\text{and}\quad u = -\frac 1y \quad\text{and}\quad w = tz \,.
\] 
The map $f\colon Y\to P$ is given in $(x,y,z)$-coordinates by $(x,y)\mapsto(x,y,1)$ (or in~$(t,u,w)$-coordinates by~$(t,u)\mapsto(t,u,t)$). We equip~$P$ with the base point~$\tilde\infty=(0,0,1)$ in~$(t,u,w)$-coordinates.
\smallskip

Next, we will give formulas for some endomorphisms of~$P$. For this, note that $[n]^*P$ is the $\bG_m$-torsor on~$E$ associated to the line bundle $\cO(E[n])\simeq\cO(n^2\infty)$, so~$[n]^*P\simeq P^{\otimes n^2}$ for all integers~$n>0$. Accordingly, there is a morphism of AT varieties
\[
\beta_n^s\colon P\to P
\]
lying over~$[n]\colon E\to E$ and under~$[n^2]\colon\bG_m\to\bG_m$, unique once we require that $\beta_n^s(\tilde\infty)=\tilde\infty$. We have the following explicit formula for~$\beta_n^s$.

\begin{proposition}\label{prop:explicit_phi}
	The map $\beta_n^s\colon P \to P$ is given in $(x,y,z)$-coordinates by
	\begin{equation}\label{eq:explicit_phi}
	\beta_n^s(x,y,z) = \bigl(x-\frac{\psi_{n+1}\psi_{n-1}}{\psi_n^2},\frac{\psi_{2n}}{\psi_n^4},(-1)^{n+1}\frac{z^{n^2}}{\psi_n}\bigr) \,,
	\end{equation}
	where~$\psi_n$ is the $n$th division polynomial of~$E$, as defined in \cite[Exercise~III.3.7]{silverman:elliptic_curves}.
\end{proposition}

\begin{lemma}\label{lem:leading_terms}
	For~$n\geq1$, the leading term of the Laurent series expansion of~$\psi_n$ at~$\infty$ is~$(-1)^{n+1}nt^{1-n^2}$.
\end{lemma}
\begin{proof}
	We proceed by induction on~$n$. For~$n=1,2,3,4$, this follows from the usual formulae for~$\psi_n$ and the fact that the leading terms of the Laurent series expansions of~$x$ and~$y$ at~$\infty$ are~$t^{-2}$ and~$-t^{-3}$, respectively. For~$n\geq5$ we proceed by strong induction on~$n$. If~$n=2m+1$ is odd, then
	\begin{align*}
		\psi_n &= \psi_{m+2}\psi_m^3-\psi_{m-1}\psi_{m+1}^3 \\
		&= ((-1)^{m+3}(m+2)t^{1-(m+2)^2}+\dots)((-1)^{m+1}mt^{1-m^2}+\dots)^3 - \\
		&\quad- ((-1)^m(m-1)t^{1-(m-1)^2}+\dots)((-1)^{m+2}(m+1)t^{1-(m+1)^2}+\dots)^3 \\
		&= ((m+2)m^3-(m-1)(m+1)^3)t^{-4m-4m^2}+\dots = (-1)^{n+1}nt^{1-n^2}+\dots \,,
	\end{align*}
	and so we have proved the inductive hypothesis in this case. A similar calculation shows the inductive hypothesis when~$n=2m$ is even.
\end{proof}

\begin{proof}[Proof of Proposition~\ref{prop:explicit_phi}]
	Let~$\beta_n'\colon P\dasharrow P$ be the rational map given by the right-hand side of~\eqref{eq:explicit_phi}. We claim that~$\beta_n'$ is a morphism. Because the first two coordinates of~$\beta_n'$ describe the morphism~$[n]\colon E\to E$, we only need consider the third coordinate. First, it is clear that $(-1)^{n+1}\frac{z^{n^2}}{\psi_n}$ is regular and invertible on~$P|_{E\smallsetminus E[n]}$ (the inverse image of~$E\smallsetminus E[n]$ in~$P$). On a neighbourhood of~$P|_{E[n]}$, the rational map~$\beta_n'$ lands in the second chart, and so it suffices to check that the $w$-coordinate of~$\beta_n'$ is invertible here. This $w$-coordinate is
	\[
	(-1)^n\cdot\frac{x-\frac{\psi_{n-1}\psi_{n+1}}{\psi_n^2}}{\frac{\psi_{2n}}{2\psi_n^4}}\cdot\frac{z^{n^2}}{\psi_n} = (-1)^n\cdot\frac{x\psi_n^2-\psi_{n-1}\psi_{n+1}}{\psi_{2n}/2\psi_n}\cdot z^{n^2} \,,
	\]
	which is clearly regular and invertible on a neighbourhood of~$P|_{E[n]\smallsetminus\{\infty\}}$.
	
	We can also write the $w$-coordinate of~$\beta_n'$ as
	\[
	(-1)^n\cdot\frac{x\psi_n^2-\psi_{n-1}\psi_{n+1}}{\psi_{2n}/2\psi_n}\cdot t^{-n^2}w^{n^2} \,.
	\]
	By Lemma~\ref{lem:leading_terms}, the Laurent expansion of~$x\psi_n^2-\psi_{n-1}\psi_{n+1}$ at~$\infty$ has leading term $t^{-2n^2}$, while the Laurent expansion of~$\psi_{2n}/2\psi_n$ has leading term $(-1)^nt^{-3n^2}$. All together, we find that the leading term of the Laurent expansion of $(-1)^n\frac{x\psi_n^2-\psi_{n-1}\psi_{n+1}}{\psi_{2n}/2\psi_n}\cdot t^{-n^2}$ is~$t^0$, and hence~$\beta_n'$ is also regular and invertible on a neighbourhood of~$P|_{\{\infty\}}$, with~$\beta_n'(\tilde\infty)=\tilde\infty$. All told, we have shown that~$\beta_n'$ is a morphism as claimed.
	
	Once we know that~$\beta_n'$ is a morphism, it is clear that it lies over~$[n]\colon E\to E$ and under~$[n^2]\colon\bG_m\to\bG_m$ (because the~$\bG_m$-action on~$P$ is given in $(x,y,z)$-coordinates by $\alpha\cdot(x,y,z)=(x,y,\alpha z)$). Since~$\beta_n'(\tilde\infty)=\tilde\infty$, this forces that~$\beta_n'=\beta_n^s$ and we are done.
\end{proof}

Next, we are going to describe the map $f_\infty\colon Y(\Qbar) \to \varinjlim_nP_n(\Qbar)$. In principle, this can be extracted directly from Proposition~\ref{prop:explicit_phi}, because all of the varieties~$P_n$ are equal to~$P^{\otimes2}$ and the morphisms $\beta_n\colon P\to P_n$ are just the composition of~$\beta_n^s$ with the squaring map~$P\to P^{\otimes2}$. But it will be easier for us to take a more indirect route which just determines the values of~$f_\infty$ on the torsion points $Y(\Qbar)_\tors \coloneqq E(\Qbar)_\tors\smallsetminus\{\infty\}$. For this, observe that $f_\infty$ maps any torsion point into the kernel of the map $\varinjlim_nP_n(\Qbar)\to\bQ\otimes E(\Qbar)$. This kernel is equal to $\bQ\otimes\bG_m(\Qbar)=\bQ\otimes\Qbar^\times$ using the morphism $\bG_m\hookrightarrow P$ picking out the fibre over~$\infty\in E(\bQ)$ (and sending~$1\in\bG_m(\bQ)$ to~$\tilde\infty\in P(\bQ)$). So~$f_\infty\colon Y(\Qbar)\to\varinjlim_nP_n(\Qbar)$ restricts to a function $Y(\Qbar)\to\bQ\otimes\Qbar^\times$, and it is this latter function which we will describe explicitly.

\begin{proposition}\label{prop:explicit_f_infty}
	For any $Q\in Y(\Qbar)_\tors$ such that~$nQ=\infty$ in~$E(\Qbar)$, we have
	\[
	f_\infty(Q) = \bigl((-1)^{n+1}n\cdot\Res_Q(\psi_n^{-1}\omega)\bigr)^{1/n^2} \,,
	\]
	where~$\omega\coloneqq\frac{\rd x}{2y+a_1x+a_3}$ is the standard invariant differential.
\end{proposition}

In the proposition and in what follows, we write the group~$\bQ\otimes\Qbar^\times$ multiplicatively, so~$b^a$ is a shorthand for~$a\otimes b$. We begin the proof with a preparatory lemma.

\begin{lemma}\label{lem:fn_on_torsion}
	Let~$Q\in E[n](\Qbar)\smallsetminus\{\infty\}$ be a non-zero $n$-torsion point. Then we have
	\[
	\beta_n^s(f(Q)) = \bigl(0,0,(-1)^{n+1}n\cdot\Res_Q(\psi_n^{-1}\omega)\bigr)
	\]
	in~$(t,u,w)$-coordinates.
\end{lemma}
\begin{proof}
	Write~$f_n^s\coloneqq\beta_n^s\circ f$ for short. Because $nQ=\infty=(0,0)$ in $(t,u)$-coordinates, it follows that the first two coordinates of~$f_n^s(Q)$ are zero. It remains to determine its $w$-coordinate, i.e.~the value of the rational function $f_n^{s*}w$ at~$Q$. By Proposition~\ref{prop:explicit_phi}, we have
	\[
	f_n^{s*}w = (-1)^{n+1}\frac{[n]^*t}{\psi_n} = (-1)^{n+1}\frac{\psi_n^{-1}\cdot\omega}{[n]^*t^{-1}\cdot\omega} \,,
	\]
	where both of the differentials on the right-hand side have simple poles at~$Q$. To evaluate this rational function at~$Q$, we thus want to take the ratio of the residues of these differentials at~$Q$. We have
	\[
	\Res_Q([n]^*t^{-1}\cdot\omega) = n^{-1}\Res_Q([n]^*(t^{-1}\omega)) = n^{-1}\Res_\infty(t^{-1}\omega) = n^{-1}
	\]
	using invariance of~$\omega$ and the fact that~$\omega=(1+O(t))\rd t$. Putting this all together shows that the value of~$f_n^{s*}w$ at~$Q$ is~$(-1)^{n+1}n\cdot\Res_Q(\psi_n^{-1}\omega)$, as claimed.
\end{proof}

\begin{proof}[Proof of Proposition~\ref{prop:explicit_f_infty}]
	By Lemma~\ref{lem:P_n_functorial}, we have a commuting diagram
	\begin{center}
	\begin{tikzcd}
		Y(\Qbar) \arrow[r,"f"]\arrow[rd,"f_\infty"'] & P(\Qbar) \arrow[r,"\beta_n^s"]\arrow[d,"\beta_\infty"] & P(\Qbar) \arrow[d,"\beta_\infty"] \\
		 & \varinjlim_mP_m(\Qbar) \arrow[r,"\beta_n^s"] & \varinjlim_mP_m(\Qbar) \,.
	\end{tikzcd}
	\end{center}
	By Lemma~\ref{lem:fn_on_torsion}, we know that $\beta_n^s(f(Q))$ equal to the image of~$(-1)^{n+1}n\cdot\Res_Q(\psi_n^{-1}\omega)\in\Qbar^\times$ under the inclusion $\bG_m\hookrightarrow P$ of the fibre over~$\infty\in E(\bQ)$. Hence we have $\beta_\infty(\beta_n^s(f(Q)))=(-1)^{n+1}n\cdot\Res_Q(\psi_n^{-1}\omega)$ as elements of $\bQ\otimes\Qbar^\times\subseteq\varinjlim_mP_m(\Qbar)$. The map $\beta_n^s\colon\varinjlim_mP_m(\Qbar)\to\varinjlim_mP_m(\Qbar)$ restricts to the $n^2$th power map on~$\bQ\otimes\Qbar^\times$, whence $f_\infty=\bigl((-1)^{n+1}n\cdot\Res_Q(\psi_n^{-1}\omega)\bigr)^{1/n^2}$ as claimed.
\end{proof}

\begin{remark}
	In the statement of Proposition~\ref{prop:explicit_f_infty}, the integer~$n$ does not need to be the exact order of~$Q$ in~$E(\Qbar)$, any multiple of the order is also fine. In particular, Proposition~\ref{prop:explicit_f_infty} implies a compatibility between the residues of the differential forms~$\psi_{mn}^{-1}\omega$ at $n$-torsion points for varying~$m$. With some care and using Lemma~\ref{lem:fn_on_torsion} one can show that this compatibility is
	\[
	mn\Res_Q(\psi_{mn}^{-1}\omega) = (-1)^{m+1}\cdot\bigl(n\Res_Q(\psi_n^{-1}\omega)\bigr)^{m^2}
	\]
	for all non-zero $n$-torsion points~$Q$. This identity seems rather hard to prove directly.
\end{remark}

\subsection{Local heights}\label{ss:local_heights}

Now let us assume that~$E$ has Mordell--Weil rank~$0$. Let~$\cE/\bZ$ denote the minimal regular model of~$E$, and let~$\cY\subseteq\cE$ denote the complement of the section at~$\infty$, so that~$\cY$ is a regular model of~$Y$. We are going to determine the finite subset $W\subseteq\varinjlim_nP_n(\bQ)$ described in \S\ref{ss:Z}. That is, if~$\cU\subseteq\cY$ is a simple open, and if~$(\cP_n)_{n\geq1}$ denotes the corresponding sequence of simple models of the~$(P_n)_{n\geq1}$, then we know that~$\varinjlim_n\cP_n(\bZ)$ is a singleton subset $\{a\}\subseteq\varinjlim_nP_n(\bQ)$, and we are going to determine all of the possibilities for~$a$. Because~$E$ has Mordell--Weil rank~$0$, the map $\bQ\otimes\bQ^\times \to \varinjlim_nP_n(\bQ)$ induced by the inclusion $\bG_m\hookrightarrow P$ is bijective, and so we may think of~$a$ as an element of~$\bQ\otimes\bQ^\times$. In particular, to pin down~$a$, it suffices to pin down its valuations~$v_\ell(a)$ at all primes~$\ell$. We will do this by considering local heights.

Recall that, as explained in Remark~\ref{rmk:local_heights}, the rigidified line bundle~$L=\cO(\infty)$ corresponding to~$P$ comes with a family of canonical locally bounded $\ell$-adic metrics \cite[Theorem~9.5.7]{bombieri-gubler:heights}
\[
||\cdot||_{L,\ell}\colon L(\Qbar_\ell) \to \bR_{\geq0} \,.
\]
On the other hand, Silverman defines a local height function \cite[\S1]{silverman:computing_heights}
\[
\hat\lambda_\ell\colon Y(\bQ_\ell) \to \bR
\]
explicitly in terms of the Weierstrass equation for~$E$. (This is the same local height as in \cite[Chapter~III.4]{lang:elliptic_curves}). Silverman's local height is uniquely characterised by the following properties:
\begin{enumerate}[label = (\alph*), ref = (\alph*)]
	\item\label{condn:local_height_boundedness} the function $\hat\lambda_\ell$ is bounded on the complement of each $\ell$-adic open neighbourhood of~$\infty\in E(\bQ_\ell)$;
	\item\label{condn:local_height_infinity} $\lim_{Q\to\infty}\bigl(\hat\lambda_\ell(Q)-\frac12\log|x(Q)|_\ell\bigr)$ exists; and
	\item\label{condn:local_height_doubling} we have
	\[
	\hat\lambda_\ell(2Q) = 4\hat\lambda_\ell(Q) - \log|2y(Q)+a_1x(Q)+a_3|_\ell
	\]
	for all~$Q\in(E\smallsetminus E[2])(\bQ_\ell)$.
\end{enumerate}
The relationship between Silverman's local height and the canonical metric $||\cdot||_{L,\ell}$ is as follows.

\begin{proposition}\label{prop:neron_functions_equal}
	Let~$f\in\rH^0(Y,L)$ be the section of~$L=\cO(\infty)$ corresponding to the rational function~$1$. Then we have
	\[
	\hat\lambda_\ell(Q) = -\log||f(Q)||_{L,\ell}
	\]
	for all~$Q\in Y(\bQ_\ell)$.
\end{proposition}

For the proof, we need to verify conditions~\ref{condn:local_height_boundedness}--\ref{condn:local_height_doubling} for the function $-\log||f||_{L,\ell}$. Condition~\ref{condn:local_height_boundedness} is immediate because~$||\cdot||_{L,\ell}$ is a locally bounded metric. The verifications of conditions~\ref{condn:local_height_infinity} and~\ref{condn:local_height_doubling} are contained in Lemmas~\ref{lem:doubling} and~\ref{lem:neron_function_asymptotic}.

\begin{lemma}\label{lem:doubling}
	For any $Q\in(E\smallsetminus E[2])(\Qbar_\ell)$, we have
	\[
	-\log||f(2Q)||_{L,\ell} = -4\log||f(Q)||_{L,\ell} - \log|2y(Q) + a_1x(Q) + a_3|_\ell \,.
	\]
\end{lemma}
\begin{proof}
	Because the map~$\beta_2^s\colon P\to P$ lies over~$[2]\colon E\to E$ and under~$[4]\colon \bG_m\to \bG_m$, we know that it factors as a composition
	\[
	P \xrightarrow{(\cdot)^{\otimes4}} P^{\otimes4} \xrightarrow[\sim]{\psi} [2]^*P \xrightarrow{\widetilde{[2]}} P
	\]
	in which the first map is~$\tilde Q\mapsto \tilde Q^{\otimes4}$, the second is an isomorphism of rigidified line bundles, and the third map is the projection from the pullback. The torsors~$P^{\otimes 4}$ and~$[2]^*P$ correspond to the rigidified line bundles~$L^{\otimes4}$ and~$[2]^*L$, respectively; let $||\cdot||_{L^{\otimes4},\ell}$ and~$||\cdot||_{[2]^*L,\ell}$ be their canonical $\ell$-adic metrics. By \cite[Theorem~9.5.7]{bombieri-gubler:heights}, for every~$\tilde Q\in P(\Qbar_\ell)$ we have
	\begin{equation}\label{eq:height_of_phi}
	||\beta_2^s(\tilde Q)||_{L,\ell} = ||\widetilde{[2]}\psi(\tilde Q^{\otimes4})||_{L,\ell} = ||\psi(\tilde Q^{\otimes4})||_{[2]^*L,\ell} = ||\tilde Q^{\otimes4}||_{L^{\otimes4},\ell} = ||\tilde Q||_{L,\ell}^4 \,.
	\end{equation}
	
	If we take~$\tilde Q=f(Q)$ for a point~$Q=(x,y)\in(E\smallsetminus E[2])(\Qbar_\ell)$ we have
	\[
	f(Q) = (x,y,1) \quad\text{and}\quad \beta_2^s(f(Q)) = (x-\frac{\psi_3}{\psi_2^2},\frac{\psi_4}{\psi_2^4},-\frac1{\psi_2})
	\]
	in~$(x,y,z)$-coordinates by Proposition~\ref{prop:explicit_phi}. In other words, $\beta_2^s(f(Q))$ and~$f(2Q)$ differ by the action of $-\psi_2^{-1}=-\frac1{2y+a_1x+a_3}\in\bQ_\ell^\times$, and so
	\begin{equation}\label{eq:height_of_phi_2}
	||\beta_2^s(f(Q))||_{L,\ell} = \frac{||f(2Q)||_{L,\ell}}{|2y+a_1x+a_3|_\ell} \,.
	\end{equation}
	Combining~\eqref{eq:height_of_phi} and~\eqref{eq:height_of_phi_2} gives the result.
\end{proof}

\begin{lemma}\label{lem:neron_function_asymptotic}
	We have
	\[
	-\log||f(Q)||_{L,\ell} = \frac12\log|x(Q)|_\ell
	\]
	for~$Q\neq\infty$ sufficiently close to~$\infty$ in the $\ell$-adic topology on~$E(\Qbar_\ell)$.
\end{lemma}
\begin{proof}
	Recall that we have the neighbourhood~$U$ of~$\infty\in E$ on which the rational functions $t\coloneqq -\frac xy$ and $u\coloneqq-\frac1y$ are both regular, and~$t$ has a simple zero at~$\infty$ and no other zeroes or poles on~$U$. The nowhere vanishing section $t^{-1}\in\rH^0(U,L)$ corresponds to a morphism $f'\colon U\to P$ lifting the open embedding $U\hookrightarrow E$, and the morphisms $f$ and~$f'$ are related by $f=tf'$ on $U\cap Y$. Because~$||\cdot||_{L,\ell}$ is a metric, this implies that
	\[
	-\log||f(Q)||_{L,\ell} = -\log||f'(Q)||_{L,\ell}-\log|t(Q)|_\ell
	\]
	for~$Q\in (Y\cap U)(\Qbar_\ell)$.
	
	On the one hand, since the Laurent series expansion of~$x$ at~$\infty$ is~$t^{-2}+O(t^{-1})$, it follows that~$t^2x\in 1+t\cO_{E,\infty}$ is a principal unit, and hence $2\log|t(Q)|_\ell+\log|x(Q)|_\ell=0$ for all~$Q$ sufficiently close to~$\infty$. On the other hand, by Lemma~\ref{lem:doubling} we have
	\[
	-\log||f'(2Q)||_{L,\ell}+4\log||f'(Q)||_{L,\ell} = \log\bigl|\frac{t(2Q)}{t(Q)^4(2y(Q)+a_1x(Q)+a_3)}\bigr|_\ell \,.
	\]
	Because $[2]\colon E\to E$ induces the doubling map on Lie algebras, it follows that~$[2]^*t$ has Laurent expansion $2t+O(t^2)$ at~$\infty$, and so $-\frac{[2]^*t}{t^4(2y+a_1x+a_3)}\in 1+t\cO_{E,\infty}$ is a principal unit by Lemma~\ref{lem:leading_terms}. This means that for~$Q$ in a neighbourhood of~$\infty\in E(\Qbar_\ell)$ we have
	\[
	-\log||f'(2Q)||_{L,\ell} = -4\log||f'(Q)||_{L,\ell} \,.
	\]
	But $-\log||f'||_{L,\ell}$ is also bounded on a neighbourhood of~$\infty$, and the only way this is possible is if~$-\log||f'(Q)||_{L,\ell}=0$ for~$Q$ in a neighbourhood of~$\infty$.
	
	Putting everything together, we have $-\log||f'(Q)||_{L,\ell}=0$ and~$-\log|t(Q)|_\ell=\frac12\log|x(Q)|_\ell$ for~$Q$ sufficiently close to~$\infty$, whence the result.
\end{proof}

Between Lemmas~\ref{lem:doubling} and~\ref{lem:neron_function_asymptotic}, we have shown that~$-\log||f||_{L,\ell}$ satisfies the conditions~\ref{condn:local_height_boundedness}--\ref{condn:local_height_doubling} uniquely characterising~$\hat\lambda_\ell$, and so we have proven Proposition~\ref{prop:neron_functions_equal}. \qed\smallskip

From now on, we will work with the renormalised local height
\[
\lambda_\ell \coloneqq \frac1{\log(\ell)}\hat\lambda_\ell = -\log_\ell||f||_{L,\ell} \,,
\]
which we will shortly see takes values in~$\bQ$. Because~$\hat\lambda_\ell$ is constant on the non-identity components of the reduction \cite[Theorem~11.5.1]{lang:elliptic_curves} (or by Remark~\ref{rmk:local_heights}), we know that~$\lambda_\ell$ takes some constant value~$b_\ell$ on~$\cU(\bZ_\ell)$. These constant values~$b_\ell$, as~$\ell$ varies, determine the value of the singleton set~$\varinjlim_n\cP_n(\bZ)\subseteq\bQ\otimes\bQ^\times$.

\begin{lemma}\label{lem:elliptic_singleton}
	Let~$\cU\subseteq\cY$ be a simple open, and let~$(\cP_n)_{n\geq1}$ be the corresponding sequence of models of~$(P_n)_{n\geq1}$. For each prime~$\ell$ let~$b_\ell\in\bQ$ be the constant value of $\lambda_\ell$ on~$\cU(\bZ_\ell)$. Then, as a subset of~$\varinjlim_nP_n(\bQ)=\bQ\otimes\bQ^\times$, we have $\varinjlim\cP_n(\bZ)=\{a\}$, where
	\[
	a = \prod_\ell\ell^{b_\ell} \,.
	\]
	(Equivalently, $a$ is the unique element of~$\bQ\otimes\bQ^\times$ with $v_\ell(a)=b_\ell$ for all~$\ell$.)
\end{lemma}
\begin{proof}
	We already know that~$\varinjlim_n\cP_n(\bZ)$ is a singleton set; let~$a$ be its unique element. Represent~$a$ by an element~$Q_n\in\cP_n(\bZ)$ in the fibre over~$0\in A(\bZ)$. Identifying the fibre of~$P_n(\bQ)$ over~$0$ with~$\bQ^\times$ using the base point~$\beta_n(\tilde\infty)$, $Q_n$ is identified with an element~$a_n\in\bQ^\times$, and~$a=a_n^{1/2n^2}$ as elements of~$\bQ\otimes\bQ^\times$.
	
	Now consider the function $-\log_\ell||\cdot||_{L,\ell}\colon P(\Qbar_\ell)\to\bQ$. By an argument similar to the one in the proof of Lemma~\ref{lem:doubling}, we have the equality
	\begin{equation}\label{eq:n-tupling}\tag{$\ast$}
		-\log_\ell||\beta_n(\tilde Q)||_{L,\ell} = -2n^2\log_\ell||\tilde Q||_{L,\ell}
	\end{equation}
	for all~$\tilde Q\in P(\Qbar_\ell)$. The function~$-\log_\ell||\cdot||_{L,\ell}$ is constant on~$\cP_n(\bZ_\ell)$ by Remark~\ref{rmk:local_heights}; taking~$\tilde Q=f(Q)$ for some~$Q\in\cU(\bZ_\ell)$ and using~\eqref{eq:n-tupling} and Proposition~\ref{prop:neron_functions_equal}, we find that the constant value of~$-\log_\ell||\cdot||_{L,\ell}$ on~$\cP_n(\bZ_\ell)$ is~$n^2b_\ell$.
	
	On the other hand, on the fibre of~$P(\Qbar_\ell)\to A(\Qbar_\ell)$ over~$0$, we know that the function $-\log_\ell||\cdot||_{L,\ell}$ coincides with the $\ell$-adic valuation $v_\ell\colon\Qbar_\ell^\times\to\bQ$ (because~$||\cdot||_{L,\ell}$ is a metric), and so we have
	\[
	v_\ell(a_n) = -\log_\ell||Q_n||_{L,\ell} = 2n^2b_\ell \,.
	\]
	This implies that $v_\ell(a)=b_\ell$, which is what we wanted to prove.
\end{proof}

Recall that we defined a finite subset~$W\subseteq\varinjlim_nP_n(\bQ)$, which is by definition the union of the singleton subsets~$\varinjlim_n\cP_n(\bZ)$ as~$\cU$ ranges over simple opens in~$\cY$. In light of Lemma~\ref{lem:elliptic_singleton}, in order to determine~$W$, it suffices to determine the possible values for the rational numbers~$b_\ell$ attached to each simple~$\cU$. In other words, we want to determine the set of values attained by the local height~$\lambda_\ell$ on~$\cY(\bZ_\ell)=\bigcup_\cU\cU(\bZ_\ell)$. This set of values of~$\lambda_\ell$ has been determined explicitly by Cremona, Prickett and Siksek \cite{cremona-prickett-siksek:height_difference_bounds}.

\begin{proposition}\label{prop:W_l}
	Let~$W_\ell=\lambda_\ell(\cY(\bZ_\ell))$ be the set of values taken by the function~$\lambda_\ell$ on the $\bZ_\ell$-integral points of~$\cY$. Then~$W_\ell$ is a finite subset of~$\bQ$, and when the Weierstrass equation for~$E$ is minimal at~$\ell$, then~$W_\ell=W_\ell^{\min}$ is as given in Table~\ref{tab:W_l} below.
	
	\begingroup
	\renewcommand{\arraystretch}{1.2}
	\begin{table}[!h]
		\begin{tabular}{||ccc|c||}
			\hline\hline
			Kodaira type & $c_\ell$ & condition & $W_\ell^{\min}$ \\\hline\hline
			$\I_0$ & $1$ & $\#\cE(\bF_\ell)\geq2$ & $\{0\}$ \\
			&& $\#\cE(\bF_\ell)=1$ & $\emptyset$ \\\hline
			$\I_m$ split & $m$ & $\ell>2$ & $\{\frac{-i(m-i)}{2m} \::\: 0\leq i\leq \lfloor\frac m2\rfloor\}$ \\
			&& $\ell=2$ & $\{\frac{-i(m-i)}{2m} \::\: 1\leq i\leq \lfloor\frac m2\rfloor\}$ \\
			$\I_m$ non-split & $1$ & $m$ odd & $\{0\}$ \\
			& $2$ & $m$ even & $\{0,-\frac{m}{8}\}$ \\\hline
			$\II$ & $1$ && $\{0\}$ \\\hline
			$\III$ & $2$ && $\{0,-\frac{1}{4}\}$ \\\hline
			$\IV$ & $1$ && $\{0\}$ \\
			& $3$ && $\{0,-\frac{1}{3}\}$ \\\hline
			$\I_0^*$ & $1$ && $\{0\}$ \\
			& $2$ or $4$ && $\{0,-\frac{1}{2}\}$ \\\hline
			$\I_m^*$ & $2$ && $\{0,-\frac{1}{2}\}$ \\
			& $4$ && $\{0,-\frac{1}{2},-\frac{m+4}{8}\}$ \\\hline
			$\IV^*$ & $1$ && $\{0\}$ \\
			& $3$ && $\{0,-\frac{2}{3}\}$ \\\hline
			$\III^*$ & $2$ && $\{0,-\frac{3}{4}\}$ \\\hline
			$\II^*$ & $1$ && $\{0\}$ \\\hline\hline
		\end{tabular}
		\caption{The sets~$W_\ell^{\min}$ in terms of the Kodaira type and Tamagawa number~$c_\ell$ of~$E$. The list is complete: it covers all possible combinations of Kodaira type and Tamagawa number.}\label{tab:W_l}
	\end{table}
	\endgroup
\end{proposition}
\begin{proof}
	Because changing the Weierstrass equation for~$E$ only changes~$\lambda_\ell$ by a rational additive constant \cite[Lemma~4]{cremona-prickett-siksek:height_difference_bounds}, it suffices to prove the second assertion. This is proved in \cite[\S2]{bianchi:bielliptic}, based on \cite{cremona-prickett-siksek:height_difference_bounds}, but we take a little care to describe carefully how to extract our statement from what is written in \cite{bianchi:bielliptic}. We partition
	\[
	\cY(\bZ_\ell) = \cY(\bZ_\ell)^{\text{good}}\sqcup\cY(\bZ_\ell)^{\text{bad}} \,,
	\]
	where $\cY(\bZ_\ell)^{\text{good}}$ is the set of $\bZ_\ell$-points whose reduction on the $\ell$-minimal Weierstrass model of~$E$ is smooth (equivalently, which reduce onto the identity component of the minimal regular model~$\cE$). Denoting~$W_\ell^{\text{good}}$ and~$W_\ell^{\text{bad}}$ their images under~$\lambda_\ell$, we then have
	\[
	W_\ell^{\min} = W_\ell^{\text{good}} \cup W_\ell^{\text{bad}} \,.
	\]
	
	The set~$W_\ell^{\text{bad}}$ is given in \cite[Table~1]{bianchi:bielliptic} (we need to divide the values in that table by~$2\log(\ell)$, since Bianchi's table displays the heights for $2\hat\lambda_v$, cf.\ \cite[\S4]{cremona-prickett-siksek:height_difference_bounds}). That table does not give any values in cases when the local Tamagawa number of~$E$ is~$1$: this is because local Tamagawa number~$1$ means that the only component of the special fibre of~$\cE$ which contains a smooth $\bF_\ell$-point is the identity component, and so~$\cY(\bZ_\ell)^{\text{bad}}=\emptyset$ in all of these cases.
	
	As for~$W_\ell^{\text{good}}$, it is equal to either~$\{0\}$ or~$\emptyset$, and is empty if and only if the reduction of the $\ell$-minimal Weierstrass model has no smooth $\bF_\ell$-points other than the point at infinity \cite[Proposition~2.3]{bianchi:bielliptic}. By~\cite[Lemma~2.4]{bianchi:bielliptic}, in the case of bad reduction, $W_\ell^{\text{good}}$ is empty if and only if~$E$ has split multiplicative reduction and~$\ell=2$.
	
	Combining these values for~$W_\ell^{\text{good}}$ and~$W_\ell^{\text{bad}}$ gives the claimed values for~$W_\ell$.
\end{proof}

\subsection{Description of the quadratic Chabauty locus}\label{ss:elliptic_description}

Now that we have a description of both the function~$f_\infty$ and the set~$W$, we can read off the description of the subscheme~$Z$, and hence the quadratic Chabauty locus, from \S\ref{ss:Z}. With an eye to future work, we give a slightly modified description which has the advantage of being independent of the choice of Weierstrass equation. Continue to assume that~$E/\bQ$ is an elliptic curve of Mordell--Weil rank~$0$ with a chosen Weierstrass equation
\[
y^2 + a_1xy + a_3 = x^3 + a_2x^2 + a_4x + a_6 \,,
\]
which may or may not be integral or minimal. Let~$\Delta$ be the discriminant of this equation, let~$\omega=\frac{\rd x}{2y+a_1x+a_3}$ be the standard invariant differential, and let~$\psi_n$ denote the $n$th division polynomial of~$E$.

\begin{definition}
	Let~$Y(\Qbar)_\tors=E(\Qbar)_\tors\smallsetminus\{\infty\}$ be the set of non-zero torsion points on~$E$, and define a function $H^\st\colon Y(\Qbar)_\tors \to \bQ\otimes\Qbar^\times$ by
	\[
	H^\st(Q) \coloneqq \bigl((-1)^{n+1}n\cdot\Res_Q(\psi_n^{-1}\omega)\bigr)^{1/n^2}\cdot\Delta^{1/12}
	\]
	for each non-zero $n$-torsion point~$Q$.
\end{definition}

It follows from Proposition~\ref{prop:explicit_f_infty} that~$H^\st(Q)=f_\infty(Q)\cdot\Delta^{1/12}$; in particular, $H^\st(Q)$ is independent of the choice of~$n$.

\begin{lemma}
	The function~$H^\st$ is independent of the choice of Weierstrass equation for~$E$.
\end{lemma}
\begin{proof}
	This follows by the standard formulae for how~$\Delta$, ~$\omega$ and~$\psi_n$ transform under changes of coordinates \cite[Table~III.1.3.1]{silverman:elliptic_curves}.
\end{proof}

For each prime number~$\ell$, define a finite subset~$W_\ell^\st\subseteq\bQ$ by
\[
W_\ell^\st \coloneqq W_\ell^{\min} + \frac1{12}v_\ell(\Delta_{\min}) \,,
\]
where~$\Delta_{\min}$ is the minimal discriminant of~$E$ and~$W_\ell^{\min}$ is the set listed in Table~\ref{tab:W_l} in terms of the Kodaira type and Tamagawa number of~$E$. The general theory we have developed specialises to the following explicit description of the quadratic Chabauty locus of the punctured elliptic curve.

\begin{theorem}\label{thm:main_elliptic_curve}
	With notation as above, let~$Z(\Qbar)\subseteq Y(\Qbar)_\tors$ be the set of torsion points~$Q\in Y(\Qbar)_\tors$ satisfying:
	\begin{enumerate}[label=(\roman*)]
		\item $H^\st(Q) \in \bQ\otimes\bQ^\times$; and
		\item $v_\ell(H^\st(Q))\in W_\ell^\st$ for all primes~$\ell$.
	\end{enumerate}
	
	\noindent Then:
	\begin{enumerate}
		\item $Z(\Qbar)$ is finite and Galois-stable, so is the $\Qbar$-points of a finite subscheme $Z\subset E$.
		\item For all primes~$p$, we have
		\[
		\cY(\bZ_p)_f = Z(\bQ_p) \,,
		\]
		where~$\cY=\cE\smallsetminus\overline{\{\infty\}}$ is the complement of the point at~$\infty$ in the minimal regular model~$\cE$ of~$E$.
	\end{enumerate}
\end{theorem}
\begin{proof}
	The statement to be proved is independent of the choice of Weierstrass equation for~$E$, so we are free to assume that the Weierstrass equation is minimal. According to Theorem~\ref{thm:main_isogeny}, the isogeny geometric quadratic Chabauty locus~$\cY(\bZ_p)_f$ is equal to the set of~$\bZ_p$-integral points on a finite subscheme~$Z$, where~$Z(\Qbar)=f_\infty^{-1}(W)$ by the description in \S\ref{ss:Z}. Because~$W$ is a subset of~$\varinjlim_nP_n(\bQ)$ and~$E$ has Mordell--Weil rank~$0$, it follows that $f_\infty^{-1}(W)$ is contained in the kernel of the map $Y(\Qbar)\to\bQ\otimes E(\Qbar)$, i.e.~in~$Y(\Qbar)_\tors$. By Proposition~\ref{prop:W_l}, $W$ is the set of elements~$a\in\bQ\otimes\bQ^\times$ such that~$v_\ell(a)\in W_\ell^{\min}$ for all~$\ell$, and so~$Z(\Qbar)=f_\infty^{-1}(W)$ is exactly the set of all~$Q\in Y(\Qbar)_\tors$ such that~$f_\infty(Q)\in\bQ\otimes\bQ^\times$ and~$v_\ell(f_\infty(Q))\in W_\ell^{\min}$ for all~$\ell$. This is equal to the set~$Z(\Qbar)$ defined in the theorem statement, so it follows that~$Z(\Qbar)$ is finite, Galois-invariant, and the isogeny geometric quadratic Chabauty locus is equal to the $\bZ_p$-integral points on~$Z$ for all~$p$.
	
	The only point which remains to be explained is why every~$\bQ_p$-point on~$Z$ is~$\bZ_p$-integral. Suppose for contradiction that~$Z$ contained a $\bQ_p$-point~$Q$ which was not~$\bZ_p$-integral. Then~$Q$ reduces to the identity on the special fibre of~$\cE$, and so we have
	\[
	\lambda_p(Q) = -\frac12v_p(x(Q)) > 0 \,,
	\]
	see \cite[\S5]{silverman:computing_heights}. On the other hand, we have~$\lambda_p(Q)\in W_\ell^{\min}$ because~$Q$ lies in~$Z$, but the set~$W_\ell^{\min}$ from Table~\ref{tab:W_l} contains only non-positive integers. This contradicts the assumption that~$Q$ was not $\bZ_p$-integral.
\end{proof}

\subsubsection{Computing~$H^\st$ via isogenies}

We now briefly state a more general formula for~$H^\st$ involving an isogeny in place of multiplication by~$n$. The significance of this formula will not be evident here, but will feature in an upcoming computational study of quadratic Chabauty with Jennifer Balakrishnan. Let~$E/\bQ$ be an elliptic curve of rank~$0$ with a chosen Weierstrass equation, and let~$\phi\colon E\to E'$ be an isogeny of degree~$d$ such that~$\ker(\phi)\cap E[2]$ has degree~$1$ or~$4$. This ensures that~$\ker(\phi)-d\infty$ is a principal divisor, so the divisor of a rational function~$g_\phi$. We normalise~$g_\phi$ so that the leading term of its Laurent expansion at~$\infty$ is~$t^{1-d}$. Let~$\omega=\frac{\rd x}{2y+a_1x+a_3}$ be the standard invariant differential on~$E$.

\begin{proposition}
	In the above setup, we have
	\[
	H^\st(Q) = \bigl(\Res_Q(g_\phi^{-1}\omega)\bigr)^{1/\deg(\phi)}\cdot\Delta^{1/12}
	\]
	for all~$Q\in\ker(\phi)(\Qbar)\smallsetminus\{\infty\}$.
\end{proposition}
\begin{proof}
	Pick a Weierstrass equation for~$E'$ (the codomain of~$\phi$). We denote the variables for this Weierstrass equation by~$x'$ and~$y'$, and then define~$t'$, $u'$, $\omega'$, $Y'$ and~$U'$ exactly as for~$E$. Let~$P'$ be the~$\bG_m$-torsor on~$E'$ corresponding to the line bundle~$\cO(\infty)$. Because~$\phi^*\cO(\infty)$ is isomorphic to~$\cO(\deg(\phi)\cdot\infty)$, there is a morphism $\tilde\phi\colon P\to P'$ lying over~$\phi\colon E\to E'$, under~$[d]\colon\bG_m\to\bG_m$. Explicitly, we can construct~$\tilde\phi$ as a composition
	\begin{equation}\label{eq:lift_of_phi}
		P \to P^{\otimes d} \xrightarrow\sim \phi^*P' \to P' \,.
	\end{equation}
	Here, $P^{\otimes d}$ and~$\phi^*P'$ are the $\bG_m$-torsors corresponding to the line bundles~$\cO(d\infty)$ and~$\cO(\ker(\phi))$ on~$E$. Thinking of these line bundles as subsheaves of the sheaf of rational functions on~$E$, the first map in~\eqref{eq:lift_of_phi} is given by~$h\mapsto h^d$, and we choose the second map to be equal to~$h\mapsto g_\phi^{-1}h$. (There is a $\bG_m(\bQ)$-ambiguity in the choice of~$\tilde\phi$; the choice of the second map picks out a unique~$\tilde\phi$.) The third map in~\eqref{eq:lift_of_phi} is the pullback of $\phi\colon E\to E'$ to~$P'$. Rather than describing it explicitly -- which is a little hard to formulate precisely -- we use that~$\phi^*P'\to P'$ is a torsor under~$\ker(\phi)$, and that this action of~$\ker(\phi)$ comes from the translation on on~$\cO(\ker(\phi))$ as a subsheaf of the sheaf of rational functions.
	
	Now equip each of the torsors~$P^{\otimes d}$, $\phi^*P'$ and~$P'$ with base points so that all of the maps in~\eqref{eq:lift_of_phi} preserve base points. So, for example, the base point of~$\phi^*P'$ is the image of~$\infty$ under the section of~$\phi^*P'$ corresponding to the rational function $t^{-d}g_\phi^{-1}$. For a point~$Q\neq\infty$ in the kernel of~$\phi$, the image of~$f(Q)$ inside~$\phi^*P'$ is the image of~$Q$ under the section corresponding to the rational function $g_\phi^{-1}$. Under the identifications
	\[
	(\phi^*P')_Q \cong (\phi^*P')_\infty \cong \bG_m
	\]
	(the first coming from the $\ker(\phi)$-action, the second coming from the chosen base point on~$\phi^*P'$), this means that the image of~$f(Q)$ is equal to the value of the rational function $\frac{\tau_Q^*g_\phi^{-1}}{t^{-d}g_\phi^{-1}}$ at~$\infty$. We can again calculate this value using residues:
	\[
	\frac{\tau_Q^*g_\phi^{-1}}{t^{-d}g_\phi^{-1}}(\infty) = \frac{\Res_\infty(\tau_Q^*g_\phi^{-1}\omega)}{\Res_\infty(t^{-d}g_\phi^{-1}\omega)} = \Res_Q(g_\phi^{-1}\omega)
	\]
	because~$\tau_Q^*\omega=\omega$ and~$g_\phi$ was chosen in such a way that~$t^{-d}g_\phi^{-1}\omega$ has residue~$1$ at~$\infty$.
	
	Because the projection $\phi^*P'\to P'$ restricts to the identity on fibres over the points at infinity on~$E$ and~$E'$ (using the pointings), we have thus shown that
	\[
	\tilde\phi(f(Q)) = \Res_Q(g_\infty^{-1}\omega)
	\]
	as elements of~$P'_\infty=\bG_m$. Then the bottom arrow in the commuting square
	\begin{center}
	\begin{tikzcd}
		P(\Qbar) \arrow[r,"\tilde\phi"]\arrow[d,"\beta_\infty"] & P'(\Qbar) \arrow[d,"\beta_\infty"] \\
		\varinjlim_nP_n(\Qbar) \arrow[r,"\tilde\phi"] & \varinjlim_nP'_n(\Qbar)
	\end{tikzcd}
	\end{center}
	is an isomorphism, and so $f_\infty(Q) = (\beta_\infty(\tilde\phi(f(Q))))^{1/d} = \bigl(\Res_Q(g_\phi^{-1}\omega)\bigr)^{1/d}$ as desired.
\end{proof}

\subsection{Two examples}\label{ss:elliptic_examples}

The description of the quadratic Chabauty locus afforded by Theorem~\ref{thm:main_elliptic_curve} gives a necessary and sufficient condition for checking whether a torsion point~$Q$ lies in the quadratic Chabauty locus. We illustrate this with two worked examples, the first of which is borrowed from \cite{bianchi:bielliptic}. Consider the rank~$0$ elliptic curve with LMFDB label \LMFDBec{8712.u5}, which has minimal Weierstrass equation
\[
y^2 = x^3 + 726x + 9317
\]
with discriminant $-2^4\cdot3^7\cdot11^6$. Consider the $2$-torsion point~$Q=(\frac{11+33\sqrt{-3}}2,0)$. Since~$y$ is a local parameter at~$Q$, we have
\begin{align*}
	\Res_Q(\psi_2^{-1}\omega) &= \Res_Q\bigl(\frac{\rd x}{4y^2}\bigr) = \Res_Q\bigl(\frac1{2(3x^2+726)}\cdot\frac{\rd y}y\bigr) \\
	 &= \frac1{2\bigl(3\cdot\bigl(\frac{11+33\sqrt{-3}}2\bigr)^2+726\bigr)} = \frac1{2\times3^2\times11^2\times\sqrt{-3}}\cdot\frac{1-\sqrt{-3}}2 \,.
\end{align*}
Noting that~$\sqrt{-3}=3^{1/2}$ and~$\frac{1-\sqrt{-3}}2=1$ in~$\bQ\otimes\Qbar^\times$, we thus have
\[
H^\st(Q) = 2^{1/3}\cdot 3^{-1/24} \,.
\]

The elliptic curve~$E$ has bad reduction of Kodaira types~$\III$, $\I_1^*$ and~$\I_0^*$ at~$2$, $3$ and~$11$, respectively, and good reduction at all other primes. The corresponding local Tamagawa numbers are~$2$, $4$ and~$2$. Consulting Table~\ref{tab:W_l}, we find that
\[
W_2^\st = \Bigl\{\frac{1}{3},\frac{1}{12}\Bigr\} \,,\quad W_3^\st = \Bigl\{\frac{7}{12},\frac{1}{12},-\frac{1}{24}\Bigr\} \,,\quad W_{11}^\st = \Bigl\{\frac{1}{2},0\Bigr\} \quad\text{and}\quad W_\ell^\st = \Bigl\{0\Bigr\}
\]
for all primes~$\ell\notin\{2,3,11\}$.

In particular, we have~$H^\st(Q)\in\bQ\otimes\bQ^\times$, and $v_\ell(H^\st(Q))\in W_\ell^\st$ for all primes~$\ell$. So~$Q\in Z(\Qbar)$, and we conclude that $\iota(Q)\in\cY(\bZ_p)_2$ for every prime~$p\notin\{2,3,11\}$ and embedding~$\iota\colon \bQ(\sqrt{-3})\hookrightarrow\bQ_p$.

\begin{remark}
	Bianchi also shows in \cite[Example~4.11]{bianchi:bielliptic} that $\iota(Q)\in\cY(\bZ_p)_2$, but her argument is rather different to ours. Notably, Bianchi verifies that~$\iota(Q)\in\cY(\bZ_p)_2$ by checking a sufficient but not necessary condition; by contrast, our argument checks a necessary and sufficient condition, so we knew \emph{ab initio} that the calculation would tell us whether $\iota(Q)$ lies in~$\cY(\bZ_p)_2$ or not.
\end{remark}

\begin{example}\label{ex:49.a3}
	Consider the rank~$0$ elliptic curve with LMFDB label \LMFDBec{49.a3}, with minimal Weierstrass equation
	\[
	y^2 + xy = x^3 - x^2 - 37x - 78
	\]
	of discriminant~$7^3$, and let~$Q$ be the point $(2+2\sqrt{7},-1-\sqrt{7})$ of order~$2$. To simplify the algebra, we will work instead in the non-minimal short Weierstrass equation
	\[
	y^2 = x^3 - 595x - 55876
	\]
	of discriminant $2^{12}\cdot 7^3$, in which~$Q$ has coordinates $(7+8\sqrt{7},0)$. The same computation as in the earlier example shows that
	\[
	\Res_Q(\psi_2^{-1}\omega) = \frac1{2\bigl(3\bigl(7+8\sqrt{7}\bigr)^2-595\bigr)} = \frac1{2^5\cdot 7\cdot(8+3\sqrt{7})} \,,
	\]
	and so we have
	\[
	H^\st(Q) = (8+3\sqrt{7})^{-1/4} \,.
	\]
	The element~$8+3\sqrt{7}$ is a principal unit in~$\bQ(\sqrt{7})$, and so~$(8+3\sqrt{7})^{-1/4}$ does not lie in~$\bQ\otimes\bQ^\times$ (it is not fixed under the Galois group of~$\bQ(\sqrt{7})/\bQ$). We conclude that~$Q\notin Z(\Qbar)$, and so~$\iota(Q)\notin\cY(\bZ_p)_2$ for all embeddings~$\iota\colon\bQ(\sqrt{7})\hookrightarrow\bQ_p$ ($p\neq 7$).
\end{example}

\begin{remark}
	In \cite[Theorem~3.18]{bianchi:bielliptic}, Bianchi gives a necessary condition for a torsion point~$Q$ to lie in~$\cY(\bZ_p)_2$: it must have the property that the values of the local heights $\lambda_v(Q)$ for~$v$ a place of~$K=\bQ(Q)$ depend only on the rational prime below~$v$. Example~\ref{ex:49.a3} demonstrates that this necessary condition is not sufficient: we have
	\[
	\lambda_v(Q) = v(H^\st(Q)) - \frac1{12}v(\Delta) \,,
	\]
	and the two terms on the right-hand side depend only on the rational prime below~$v$ (because~$H^\st(Q)\in\bQ\otimes\bZ[\sqrt{7}]^\times$ has valuation~$0$ everywhere). However, as we have seen, $\iota(Q)\notin\cY(\bZ_p)_2$ for any embedding $\iota\colon\bQ(\sqrt{7})\hookrightarrow\bQ_p$, demonstrating that the criterion of \cite[Theorem~3.18]{bianchi:bielliptic} is not sufficient.
	
	We further remark that this example is even stronger: one can check that if~$\ell$ is the rational prime below~$v$, then we have $\lambda_v(Q)+\frac1{12}v(\Delta)\in W_\ell^\st$. So the local heights of~$Q$ satisfy all of the conditions we expect of points in~$\cY(\bZ_p)_2$, but nonetheless it does not lie in the locus.
\end{remark}

\appendix

\section{Remarks on non-realisability of certain quotients}\label{appx:non-realisability}

In this appendix, we wish to explain that the fact that the quotient~$U^p_\AT$ is realisable by a morphism to a smooth algebraic variety is rather special, and by far the majority of quotients of the fundamental group are not realisable in this way. This is because Hodge theory imposes strong restrictions on the possible fundamental groups of smooth algebraic varieties. For example, one has the following rather coarse necessary condition for a pro-unipotent group to be the pro-unipotent fundamental group of a smooth algebraic variety.

\begin{lemma}\label{lem:fundamental_group_criterion}
	Let~$\fg_\bC$ be a pro-nilpotent complex Lie algebra, and let
	\[
	d_n\coloneqq \dim\gr_\Gamma^n\fg_\bC
	\]
	denote the dimensions of the graded pieces of the descending central series. If~$\fg_\bC$ is the complex Mal\u cev Lie algebra of the fundamental group of a smooth connected complex variety, then
	\begin{equation}\label{eq:fundamental_group_criterion_inequality}
		d_3\geq d_1\cdot\bigl(d_2-\frac{25}{54}d_1^2+\frac12d_1-\frac13\bigr) \,.
	\end{equation}
\end{lemma}
\begin{proof}
	Let~$\fg$ be the $\bQ$-Mal\u cev Lie algebra of the fundamental group of a smooth connected complex variety~$V$, whose complexification is~$\fg_\bC$. By Hain--Zucker \cite{hain-zucker:unipotent_variations}, $\fg$ carries a pro-$\bQ$-mixed Hodge structure compatible with the Lie bracket. This mixed Hodge structure allows one to write down a presentation of~$\fg_\bC$: there is a canonical isomorphism
	\[
	\fg_\bC\cong \prod_n\gr^\rW_{-n}\fg_\bC
	\]
	of the weight filtration, compatible with the Lie bracket, and $\prod_n\gr^\rW_{-n}\fg$ is canonically isomorphic to the quotient of the free pro-nilpotent Lie algebra~$\ff$ generated by~$\gr^\rW_\bullet\rH_1(V,\bQ)$, modulo the ideal~$\fr$ generated by the image of a morphism of Hodge structures $\gr^\rW_\bullet\rH_2(V,\bQ)\to\ff$. The Hodge decomposition of the pure Hodge structures~$\gr^\rW_{-n}\fg_\bC$ thus gives a bigrading
	\[
	\fg_\bC\cong\prod_{i,j}\fg^{i,j} \,,
	\]
	arising from corresponding bigradings on~$\ff$ and~$\fr$. Note that the Lie algebra generators of~$\ff$ lie in bidegrees~$(0,-1)$, $(-1,0)$ and $(-1,-1)$, while the Lie ideal generators of~$\fr$ lie in bidegrees~$(0,-2)$, $(-1,-1)$, $(-1,-2)$, $(-2,0)$, $(-2,-1)$, $(-2,-2)$.
	
	Let~$x_1,\dots,x_a$, $y_1,\dots,y_b$ and $z_1,\dots,z_c$ be generators of~$\ff$, with each~$x_i\in\ff^{0,-1}$, $y_j\in\ff^{-1,0}$ and $z_k\in\ff^{-1,-1}$. Any relation in~$\fr^{-1,-1}$ is of the form
	\[
	\sum_{i=1}^a\sum_{j=1}^b\lambda_{ij}[x_i,y_j] + \sum_{k=1}^c\mu_kz_k \,.
	\]
	If any such relation has a coefficient~$\mu_k\neq0$, we may use this relation to eliminate the generator~$z_k$ without changing the possible bidegrees of the generators and relations. So we are free to assume that all relations in~$\fr^{-1,-1}$ are of the form
	\[
	\sum_{i=1}^a\sum_{j=1}^b\lambda_{ij}[x_i,y_j] \,.
	\]
	
	Now equip the ideal~$\fr$ with the restriction of the descending central series filtration from~$\ff$, so that the exact sequence
	\[
	0 \to \fr \to \ff \to \fg_\bC \to 0
	\]
	is strict for filtrations. If we let~$r_n\coloneqq\dim\gr_\Gamma^n\fr$ denote the dimensions of the graded pieces of this filtration, then we have
	\[
	d_n = \frac1n\sum_{m\mid n}\mu(n/m)(a+b+c)^m - r_n \,,
	\]
	using Witt's formula for the dimensions of the graded pieces of the free pro-nilpotent Lie algebra~$\ff$ on~$a+b+c$ generators. Our assumption on the shape of the relations in~$\fr^{-1,-1}$ implies that~$r_1=0$. As for~$r_3$, the elements of~$\gr_\Gamma^3\fr$ lie in the span of the elements of the form
	\[
	[w,x_i] \,,\, [w,y_j] \,,\, [w,z_k] \,,\, [[x_i,y_j],x_{i'}] \,,\, [[x_i,y_j],y_{j'}] \,,\, [[x_i,z_k],y_j] \,,\, [[y_j,z_k],x_i]
	\]
	for~$w\in\gr_\Gamma^2\fr$. This gives the bound
	\[
	r_3 \leq (a+b+c)r_2 + ab(a+b+2c) \,.
	\]
	Hence~$d_1=a+b+c$ and
	\begin{align*}
		d_3-d_1d_2 &= \frac13(d_1^3-d_1) - r_3 - \frac12(d_1^3-d_1^2) + d_1r_2 \\
		&\geq -\frac16d_1^3 + \frac12d_1^2 - \frac13d_1 - ab(a+b+2c) \\
		&\geq -\frac16d_1^3 + \frac12d_1^2 - \frac13d_1 - \bigl(\frac{2a+2b+2c}3\bigr)^3 = -\frac{25}{54}d_1^3 + \frac12d_1^2 - \frac13d_1 \,,
	\end{align*}
	using the AM--GM inequality in the last line. This is the inequality we wanted to prove.
\end{proof}

\begin{corollary}\label{cor:U2_not_realisable}
	Let~$Y=X$ be a smooth projective curve of genus~$\geq4$ over~$\Qbar$ with a point~$b\in X(\Qbar)$, and let~$U^p_2$ be the maximal $2$-step unipotent quotient of~$\pi_1^{\bQ_p}(X;b)$. Then there is no smooth connected complex algebraic variety whose $\bQ_p$-pro-unipotent fundamental group is isomorphic to~$U^p_2$.
	
	In particular, the quotient~$U^p_2$ is not realised by any morphism to a smooth connected variety.
\end{corollary}
\begin{proof}
	Let~$\fg_\bC$ denote the complex Mal\u cev Lie algebra of the maximal $2$-step nilpotent quotient of the surface group of genus~$g\geq4$. It suffices to show that this~$\fg_\bC$ is not the complex Mal\u cev Lie algebra of the fundamental group of a smooth connected complex variety. In the notation of Lemma~\ref{lem:fundamental_group_criterion}, we have~$d_1=2g$ and~$d_2=\frac12(2g)(2g-1)-1$, while~$d_3=0$. This violates~\eqref{eq:fundamental_group_criterion_inequality} as soon as~$g\geq4$.
\end{proof}

This corollary also justifies the assertion in Remark~\ref{rmk:higher_albanese_manifolds} that~$\Alb_2(X)$ is not an algebraic variety: its $\bQ_p$-pro-unipotent fundamental group is~$U^p_2$.

\bibliographystyle{alpha}
\bibliography{references}

\end{document}